\documentclass[11pt,draft, reqno]{amsart}

\allowdisplaybreaks

\usepackage[margin=16mm]{geometry}
\usepackage{amssymb, amsmath, amsthm, amscd}

\usepackage[T1]{fontenc}
\usepackage{eucal,mathrsfs,dsfont}%gives nicer set-names. Use \mathds instead of \mathds
\usepackage{color}%cyan,red;magenta,green;yellow,blue
\def\<{\langle}
\def\>{\rangle}

\def\1{{\mathbf{1}}}

\renewcommand{\geq}{\geqslant}
\renewcommand{\le}{\leqslant}
\renewcommand{\ge}{\geqslant}

\def\EE{{\mathcal E}}

\def\LL{{\mathcal L}}

\newcommand{\D}{\mathscr{D}}

\newcommand{\HH}{\mathcal{H}}

\definecolor{mno}{rgb}{0.5,0.1,0.5}

\newcommand{\R}{\mathds R}

\newcommand{\e}{\varepsilon}

\newcommand{\Pp}{\mathds P}
\newcommand{\Ee}{\mathds E}

\newcommand{\I}{\mathds 1}
\newcommand{\w}{\omega}

\newtheorem{theorem}{Theorem}[section]
\newtheorem{lemma}[theorem]{Lemma}
\newtheorem{proposition}[theorem]{Proposition}

\theoremstyle{definition}

\newtheorem{example}[theorem]{Example}
\newtheorem{remark}[theorem]{Remark}
\newtheorem{assumption}[theorem]{Assumption}

\numberwithin{equation}{section}

\begin{document}
\allowdisplaybreaks
\title[Non-local symmetric operators with divergence-free drift] {\bfseries Stochastic homogenization of diffusions in turbulence driven by
non-local symmetric L\'evy operators}
\author{Xin Chen, \quad Jian Wang\quad \hbox{and}\quad Kun Yin}

\date{}

\maketitle

\begin{abstract}
We investigate the stochastic homogenization of a class of turbulent diffusions generated by non-local symmetric L\'evy operators with divergence-free drift fields in ergodic random environments, where neither the drift fields nor their associated stream functions are assumed to be bounded. A pivotal step in our proof is the establishment of $W_{loc}^{1,q}$ estimates with $q\in (1,2)$ for the corresponding correctors, under mild prior regularity conditions imposed on the L\'evy measure and the stream function.

\medskip

\noindent\textbf{Keywords:} stochastic homogenization;  non-local symmetric L\'evy
operator; ergodic random environment; divergence-free drift; jump process.

\medskip

\noindent \textbf{MSC 2020:} 60G51; 60G52; 35B27; 82D30
\end{abstract}
\allowdisplaybreaks

\section{Introduction and main result}\label{section1}

\subsection{Background}

In this paper, we investigate the stochastic homogenization of the following random operator
\begin{equation}\label{e:operator1}
L^\w f(x):=L_0 f(x)+\langle b(x;\w),\nabla f(x)\rangle,\quad f\in C_b^2(\R^d),
\end{equation}
where $b(\cdot;\w)$ is a divergence-free random field.

When $L_0$ corresponds to the Laplacian operator, this model is closely related to a fundamental paradigm in statistical fluid mechanics, which describes the dynamics of diffusive particles convected by a random incompressible velocity field. In particular, the long-time behavior of the Markov process generated by the random operator $L^\w$ is characterized via the stochastic homogenization framework (or the invariance principle). A core component of this theory lies in the derivation of the effective diffusivity, a key parameter that determines the coefficients of the limiting Brownian motion induced by the drift field $b(x;\w)$ under appropriate regularity conditions.

Independently, Kozlov \cite{Koz} and Papanicolaou and Varadhan \cite{PV} pioneered the so-called -- seen from the particle -- method to construct the associated corrector in ergodic random environments. This technique has since emerged as a cornerstone in the research of stochastic homogenization. Under the assumption that the drift field $b(x;\w)$ admits a stream function (see Assumption \ref{a1-2} below), extensive studies have been conducted on properties of stochastic homogenization for $L^\w$. For instance, Osada \cite{O} established the quenched invariance principle under the condition of a bounded stream function. Landim, Olla and Yau \cite{LOY} proved the corresponding result for time-dependent environments,
 assuming the existence of a (time-dependent) bounded stream. Oelschl\"{a}ger \cite{Oe} relaxed the boundedness constraint and analyzed stochastic homogenization under finite $p$-moment conditions on the stream function for any $p<\infty$. With finite second moment assumptions on the stream function, Fannjiang and Papanicolaou \cite{FP} further demonstrated the $L_2$-convergence for the density of the associated Markov process. Leveraging Moser's iteration argument, Fannjiang and Komorowski \cite{FK,FK1} proved the quenched invariance principle for both static and time-dependent ergodic environments, respectively, under strengthened finite moment conditions on the stream function. Komorowski and Olla \cite{KO} proposed a criterion for stochastic homogenization in time-dependent models based on the spectral resolution of the drift field.
Readers are referred to the monograph \cite{KLO} and the references therein for details on related developments.
Fehrman \cite{Fe} explored stochastic homogenization in space-time ergodic settings, where the impact of temporal variables was completely analyzed, and the limiting process was shown to deviate from a standard Brownian motion in certain cases. Additionally, Fehrman \cite{Fe1} provided a simple proof of quenched stochastic homogenization, and also examined the large-scale H\"{o}lder regularity and  the first-order Liouville principle of the operator $L^\w$.

Recently, significant progress has been made for the critically correlated case, where the drift field $b(x;\w)$ does not possess a globally defined stream function. Cannizzaro, Haunschmid-Sibitz and Toninelli \cite{CHT} derived long-time second-moment estimates for a Brownian particle in
$\R^2$, subject to a random, time-independent drift given by the curl of the two-dimensional Gaussian Free Field. By combining stochastic homogenization techniques with refined cutoff arguments, Chatzigeorgiou, Morfe, Otto and Wang \cite{CMOW} improved these estimates to achieve the optimal order in the large-time regime. Furthermore, Armstrong, Bou-Rabee and Kuusi \cite{ABK} employed a renormalization group approach to analyze the coarse-grained diffusivity across different scales, proving both qualitative and quantitative versions of the quenched invariance principle with super-diffusive scaling. Notably, the limiting Brownian motion in their work is fully determined by the drift field. For additional properties of models with critically correlated divergence-free drift fields, we refer the reader to \cite{ABK1,MOW,MOW2,MOW1,OW} and the references therein.

In contrast, results for the case where $L_0$ is a non-local operator remain relatively scarce. For symmetric stable-like operators $L_0$, Chen and Yin \cite{CY} studied the stochastic homogenization of $L^\w$, showing that the limiting process is an $\alpha$-stable L\'evy process with no effective diffusivity term. A feature of this work is that the proof does not require the construction of a corrector. To the best of our knowledge, it remains an unknown question whether the limiting process in the stochastic homogenization of non-local operators $L_0$ can be a Brownian motion with a non-trivial effective diffusivity. In this paper, we focus on a class of non-local L\'evy operators $L_0$ whose L\'evy measures are $L_2$-integrable (see Assumption \ref{a1-2--} for details). In fact, for Markov processes generated by symmetric non-local operators with $L_2$-integrable jumping kernels, stochastic homogenization has been established by Biskup, Chen, Kumagai and Wang \cite{BCKW}, Flegel, Heida and Slowik \cite{FHS}, and Piatnitski and Zhizhina \cite{PZ} via distinct approaches; in all these cases, the limiting process is a Brownian motion with effective diffusivity. As discussed in Remark \ref{r1-2} below, a primary challenge in our work is to establish $W^{1,q}_{loc}$-estimates for the associated corrector, given the weaker regularity of both the non-local L\'evy operator $L_0$ and the drift field $b(x;\w)$. For recent advances in the stochastic homogenization of non-local operators, we also refer to \cite{CKW, CCKW1, CCKW, CKK, FH, KPZ, PZ, S} and the references therein.

\subsection{Framework and main result}\label{section1.2}
 Let $d\geq 2$. Suppose that $(\Omega, \mathcal{G}, \Pp)$ is a probability space
 endowed with a measurable group of transformations $\tau_x:\Omega \to \Omega$, for all $x\in \R^d$, such that
\begin{itemize}
\item [(a)] $\tau_0={\rm id}$, where ${\rm id}$ denotes the identity map on $\Omega$;

\item [(b)] $\tau_x \circ \tau_y=\tau_{x+y}$  for every $x, y\in \R^d$.
\end{itemize}
We assume that the following stationary, ergodic and measurable properties hold:
\begin{itemize}
\item[(i)] The stationary condition: $\Pp(\tau_x A)=\Pp(A)$ for all $A\in \mathcal{G}$ and $x\in \R^d$;

\item[(ii)] The ergodic condition: if $A\in \mathcal{G}$ and $\tau_xA=A$ for all $x\in \R^d$, then $\Pp(A)\in \{0,1\}$;

\item[(iii)] The measurable condition: the function $(x,\w)\mapsto \tau_x \w$ is $\mathscr{B}(\R^d)\times \mathcal{G}$-measurable.
\end{itemize}

Throughout the paper, we suppose that
$L_0$ in \eqref{e:operator1} is a symmetric non-local (deterministic) L\'evy operator having the following expression
$$L_0f(x):={\rm p.v.}\int_{\R^d}\left(f(x+z)-f(x)\right)\nu(z)\,dz,\quad f\in C_b^2(\R^d)$$
with
\begin{align}\label{jumping}
\nu(z):=\frac{1}{|z|^{d+\alpha}}\I_{\{|z|\le 1\}}+\nu(z)\I_{\{|z|>1\}},\quad z\in \R^d,
\end{align}
and the random drift $b(x;\w)$ in \eqref{e:operator1} is defined by
\begin{equation*}
b(x;\w):=\tilde b(\tau_x \w)=(\tilde b_1(\tau_x \w), \tilde b_2(\tau_x \w),\cdots, \tilde b_d(\tau_x \w)),\quad x\in \R^d,\ \w\in \Omega
\end{equation*}
with $\tilde b:\Omega \to \R^d$.

We always make the following assumption on the operator $L_0$.

\begin{assumption}\label{a1-2--} \it The density function $\nu(z)$ of the jumping measure for the operator $L_0$ fulfills all the following conditions:
\begin{itemize}
\item [(i)] The constant $\alpha$ in \eqref{jumping} satisfies that $\alpha\in (1,2)$.

\item [(ii)] $\nu(z)=\nu(-z)$ for all $z\in \R^d$, $\sup_{z\in \R^d:|z|>1}\nu(z)<\infty$ and
$$\displaystyle\int_{\{|z|>1\}}|z|^2 \nu(z)\,dz<\infty.$$
\item [(iii)]  There exist constants $c_0\ge1$ and $K_0\ge2$, and a function $\gamma:\R^d\to \R_+:=(0,\infty)$ such that for every $r\ge 1$ and  $x\in \R^d$ with $|x|\ge K_0r$,
\begin{equation}\label{t2-1-2a}
\int_{B\left(x,r\right)}\nu^2(z)\,d z\le c_0r^d\gamma^2(x)\nu^2(x)
\end{equation}
and
\begin{equation}\label{a1-2-1}
\int_{\{|x|\ge K_0\}}|x|^2\gamma^2(x)\nu(x)\,dx<\infty.
\end{equation}
\end{itemize}
\end{assumption}
\noindent In particular, under Assumption \ref{a1-2--}(ii), $\displaystyle \int_{\R^d}|z|^2\nu(z)\,dz<\infty;$ that is, the L\'evy measure $\nu(z)\,dz$ corresponding to the operator $L_0$ has finite second moment.  Example \ref{ex1-1} below indicates that Assumption \ref{a1-2--}(iii) can
apply to a large class of symmetric non-local L\'evy operators with regular jumping measures.

Let $\HH:=L^2(\Omega;\Pp)$. Using $\{\tau_x\}_{x\in \R^d}$, we can define the translation operator $T_x: \HH\to \HH$ and the derivative operator $D_j$
(if the limit exists) respectively as follows:
\begin{align*}
T_x \tilde F(\w):=&\tilde F(\tau_{-x} \w),\quad \tilde F\in \HH,\, x\in \R^d,\\
D_j \tilde F(\w):=&\lim_{\e\downarrow 0}\frac{T_{\e e_j}\tilde F(\w)-\tilde F(\w)}{\e},\quad \tilde F\in \HH,\, 1\le j \le d,
\end{align*}
where $\{e_j\}_{1\le j \le d}$ denotes the canonical orthonormal basis of $\R^d$. In particular, it is easy to verify that, if $D_j\tilde F$ exists, then
\begin{equation}\label{e1-3}
D_j \tilde F(\tau_x \w)=-\frac{\partial F(\cdot;\w) }{\partial x_j}(x),\quad x\in \R^d,
\end{equation}
where $F(x;\w):=\tilde F(\tau_x \w)$.
Set
\begin{equation*}
\D:=\{\tilde F\in L^\infty(\Omega;\Pp): D_j \tilde F\ {\rm \, and\, }D_k(D_j \tilde F)\ {\rm exist\, and\, belong\, to}\ L^\infty(\Omega;\Pp)\ {\rm for\ every}\ 1\le j,k \le d\}.
\end{equation*}
According to the stationary property of the transformation $\{\tau_x\}_{x\in \R^d}$ and the integration by parts formula, we have
\begin{equation}\label{e1-2a}
\Ee[D_j \tilde F \,\tilde G]=-\Ee[\tilde F D_j \tilde G],\quad\tilde F,\tilde G\in \D.
\end{equation}
Since $\{T_x\}_{x\in \R^d}$ is a strongly continuous unitary semigroup on $\HH$,  we have the following spectral expression:
\begin{equation}\label{e1-1a}
T_x=\int_{\R^d} e^{i \langle \xi, x\rangle }\,U(d\xi),\quad x\in \R^d,
\end{equation}
where $U(d\xi)$ is the associated projection valued measure on $\HH$. By \eqref{e1-2a} and \eqref{e1-1a}, we can define $\HH_1:=\overline{\D}^{\|\cdot\|_{\HH_1}}$ as
the closed extension of $\D$ under the $\|\cdot\|_{\HH_1}$-norm, where
\begin{equation*}
\|\tilde F\|_{\HH_1}^2:=\Ee[\tilde F^2]+\sum_{j=1}^d \Ee[|D_j \tilde F|^2]=\int_{\R^d}\left(1+|\xi|^2\right)\Ee[U(d\xi)\tilde F\cdot \tilde F],\quad \tilde F\in \D.
\end{equation*}
Hence, $D_j \tilde F$
is well defined for every $\tilde F\in \HH_1$, and the integration by parts formula \eqref{e1-2a} still holds
for every $\tilde F,\tilde G\in \HH_1$. It is also easy to see that
$F(\cdot,\w)\in W_{loc}^{1,2}(\R^d)$ for every $\tilde F\in \HH_{1}$ and a.e.\ $\w\in \Omega$, where $F(x;\w):=\tilde F(\tau_x \w)$.
We refer the readers
to \cite[Section 2]{PV} for the details of all the contents mentioned above.

We then put the following assumption on the drift term $b(x;\w)=\tilde b(\tau_x\w)$.

\begin{assumption}\label{a1-2} \it
Let $\tilde b=(\tilde b_1,\tilde b_2,\cdots, \tilde b_d):\Omega \to \R^d$ be a random vector such that
there exists a system of stream functions $\tilde H_{jl}:\Omega \to \R$, $1\le j,l\le d$, satisfying the following properties:
\begin{itemize}
\item [(i)] $\tilde H_{jl} \in \HH_1$ for every $1\le j,l \le d$; $\tilde H_{jl}=-\tilde H_{lj}$ for every $1\le j,l \le d$;
and the function $x\mapsto D_l \tilde H_{jl}(\tau_x \w)$ is continuous for every $1\le j,l \le d$ and a.s. $\w\in \Omega$.
\item [(ii)] $\tilde b_j=\sum_{l=1}^d D_l \tilde H_{jl}$ for every $1\le j \le d$.
\end{itemize}
\end{assumption}
\noindent It follows from Assumption \ref{a1-2} that (in the distributional sense)
\begin{equation}\label{e1-5}
{\rm div}b(\cdot;\w)(x)=-\sum_{j=1}^d \frac{\partial b_j(x;\w)}{\partial x_j}=\sum_{j=1}^d D_j \tilde b_j(\tau_x \w)=
\sum_{j,l=1}^d D_jD_l \tilde H_{jl}(\tau_x\w)=0, \quad x\in \R^d,
\end{equation}
where in the last equality we used the anti-symmetry of $\{\tilde H_{jl}\}_{1\le j,l\le d}$. In particular, the random drift $b(\cdot;\w)$ is divergence-free.

Now, for any $\varepsilon>0$, we define the scaled operator $L^{\e,\w}$ by
\begin{equation*}
L^{\e,\w}f(x)=L_0^\e f(x)+\e^{-1}\left\langle b\left(\frac{x}{\e};\w\right),\nabla f(x)\right\rangle,\quad f\in C_b^2(\R^d),
\end{equation*}
where
\begin{equation*}
L_0^\e f(x):=\e^{-d-2}{\rm p.v.}\int_{\R^d}\left(f(x+z)-f(x)\right)\nu\left(\frac{z}{\e}\right)\,dz=\e^{-2}{\rm p.v.}\int_{\R^d}\left(f(x+\e z)-f(x)\right)\nu({z})\,dz.
\end{equation*}
If $(X_t^{\w})_{t\ge 0}$ is a Markov process associated with the operator $L^\w$,
then
$(X_t^{\e,\w})_{t\ge 0}:=(\e X_{\e^{-2}t}^{\w})_{t\ge 0}$ is the Markov process associated with the scaled operator
$L^{\e,\w}$. For any $\lambda>0$ and $h\in C_b^\infty(\R^d)$, consider the following resolvent equation associated with $L^{\e,\w}$:
\begin{equation}\label{e1-6}
\lambda u^{\e}(x;\w)-L^{\e,\w}u^{\e}(x;\w)=h(x),\quad x\in \R^d.
\end{equation} The existence of a weak solution to the equation \eqref{e1-6} will be established in Proposition \ref{t1-1} below.

Define a $d\times d$ matrix
$\bar A=\{\bar a_{jk}\}_{1\le j,k \le d}$ by
\begin{equation}\label{e1-7}
\bar a_{jk}=\Ee\left[\int_{\R^d}\left(z_j+\phi_j(z;\w)\right)\left(z_k+\phi_k(z;\w)\right)\nu(z)\,dz\right],\quad 1\le j,k\le d,
\end{equation} where $\phi:\R^d\times \Omega \to \R^d$ is the corrector constructed in Theorem \ref{t2-1}. As shown in Lemma \ref{l1-2} below, $\bar A$ is a 
strictly positive definite and bounded matrix, so
\begin{equation*}
\bar L f(x):=\frac{1}{2}\sum_{j,k=1}^d \bar a_{jk}\frac{\partial^2 f(x)}{\partial x_j \partial x_k},\quad f\in C_c^2(\R^d)
\end{equation*} is a uniformly elliptic
second differential operator.
For every $\lambda>0$, set
$$\mathscr{G}_\lambda=\{
h\in C_b^\infty(\R^d)\cap W^{1,2}(\R^d):
(\lambda -\bar L)^{-1} h\in C_c^\infty(\R^d)\}.$$
Then, for any $\lambda>0$ and $h\in \mathscr{G}_\lambda$, there exists a unique solution $\bar u$ in $L^2(\R^d;dx)$ to the following equation
\begin{equation}\label{t1-2-2}
\lambda \bar u(x)-\bar L \bar u(x)=h(x),\quad x\in \R^d.
\end{equation}

The main theorem of this paper is concerned on the stochastic homogenization for the scaled 
resolvent 
equation \eqref{e1-6} to \eqref{t1-2-2} as $\e \to 0.$ For this, we further need the following assumption.

\begin{assumption}\label{a1-3} \it
Suppose that $d> 4(\alpha-1)$ and that there exists $q\in (2d/(d+2),2)$ such that
the following properties hold.
\begin{itemize}
\item [(i)] There exists a constant $r_0\ge 1$ such that
\begin{equation}\label{t2-1-3a}
\Ee\left[N_{r_0}(\w)^{\frac{\max\{4,
d\}}{2(\alpha-1)(2-q)}}\right]<\infty,
\end{equation}
where $$N_{r_0}(\w):=\sup_{x\in B(0,r_0), 1\le j,l \le d}\left|D_j\tilde H_{jl}(\tau_x \w)\right|.$$

\item [(ii)] For some $r> \max\left\{\frac{q}{q-1},\frac{2p'_0}{p'_0-2}\right\}$ with $p_0=\frac{dq}{d-q}\in (2,\infty)$ and $p_0'=\frac{2p_0\alpha}{4(\alpha-1)+p_0(2-\alpha)}\in (2,p)$,
\begin{equation}\label{a1-3-1}
\Ee[|\tilde H_{jl}|^{r}]<\infty,\quad 1\le j,l\le d
\end{equation}
and
\begin{equation}\label{a1-3-2}
\Ee\left[\left|\int_{\{|z|\le 1\}}
\frac{|\tilde H_{jl}(\tau_z \w)-\tilde H_{jl}(\w)|^2}{|z|^{d+2-\alpha}}\,dz\right|^{\frac{p_0}{p_0-2}}\right]<\infty,\quad
\ 1\le j,l\le d.
\end{equation}
\end{itemize}
\end{assumption}

\begin{theorem}\label{t1-2} Suppose that Assumptions $\ref{a1-2--}$, $\ref{a1-2}$  and $\ref{a1-3}$  are satisfied.
Then, for any $\lambda>0$, $h\in \mathscr{G}_\lambda$,  a.s. $\w\in \Omega$ and every $R\ge 1$, $p\ge 1$,
\begin{equation}\label{t1-2-1}
\lim_{\e \to 0}\int_{B(0,R)}|u^{\e}(x;\w)-\bar u(x)|^p\,dx=0.
\end{equation}
Here $u^{\e}(x;\w)$ and $\bar u$ are solutions to the equations \eqref{e1-6} and \eqref{t1-2-2}, respectively.
\end{theorem}

We give some comments on the proof of Theorem \ref{t1-2}.

\begin{remark}\label{r1-2}
\begin{itemize}
\item [(i)] Compared with the case of symmetric non-local operators with $L_2$-integrable jumping kernels \cite{BCKW, FHS, PZ} or the case where $L_0$ is a second-order differential operator \cite{FK, Fe, PV}, we will employ 
    a different method to handle the effect of the drift field $b(x;\w)$ in the construction of the associated corrector $\phi$; see \eqref{e:equ-01} below. In particular, additional analysis based on the spectral representation of the drift field is carried out; see the proof of Theorem \ref{t2-1} for details.

\item [(ii)] To identify the limiting operator $\bar L$, we must address a term involving the corrector $\phi$ in the test function,
see the function $f_{\e_m}(x;\w)$ defined by \eqref{e:test} in the proof of Theorem \ref{t1-2}. 
For this reason, the regularity property $\phi\in W^{1,q}_{loc}(\R^d;\R^d)$ for some $q\in (1,2)$ is essential. In the case $L_0=\Delta$, the property $\phi\in W^{1,2}_{loc}(\R^d;\R^d)$ holds automatically for the corrector, thanks to the a priori regularity estimate \begin{align}\label{r1-2-1}\Ee\left[|\nabla \phi(x;\w)|^2\right]<\infty,\quad x\in \R^d,\end{align} see, e.g., \cite{FK, Fe, PV} for further details. However, for the non-local operator $L_0$ considered in this paper, the main difficulty is that the a priori estimate \eqref{r1-2-1} fails to hold. Instead, we can only obtain weaker a priori estimates \eqref{t2-1-2} and \eqref{t2-1-3} below for the corrector corresponding to $L^\w$.
    We develop a novel procedure tailored to non-local operators for establishing the $W^{1,q}_{loc}(\R^d;\R^d)$ regularity of the corrector $\phi$ with $q\in (1,2)$,
    which is founded on these weaker a priori estimates combined with additional structural conditions imposed on the stream function $\{\tilde H_{jl}\}_{1\le j,l\le d}$.
Furthermore, owing to the reduced regularity of the non-local operator $L_0$, compared to the case of second-order differential operators we need encounter additional technical challenges in eliminating several residual terms throughout the stochastic homogenization.
   For instance, see step (3) in the proof of Theorem \ref{t1-2}. This partly explains why the assumption $\alpha\in (1,2)$ is necessary.

\item [(iii)]
We note that the approach here differs from those in \cite{CMP, CSZ, Si}, where Schauder estimates for elliptic equations induced by supercritical non-local operators have been established. This difference arises partly from the weaker a priori estimates \eqref{t2-1-2} and \eqref{t2-1-3}, as well as the lower regularity of the drift field $b(x;\w)$; for example, we cannot guarantee that the corrector $\phi$ is globally bounded on $\R^d$. Furthermore, in \cite{CMP, CSZ, Si}, the Schauder estimates depend only on the small-jump part of the associated non-local operator $L_0$. By contrast, the situation is different for our model: in the present setting, the large-jump part of the non-local operator $L_0$ also plays a crucial role in the $W_{loc}^{1,q}(\R^d;\R^d)$ estimate of the corrector $\phi$. Indeed, the regularity condition
\eqref{t2-1-2a}
can be interpreted as a kind of (large-scale) averaged lower bound for the jumping kernel associated with large jumps.
\end{itemize}
\end{remark}

Below we take some especial examples such that Assumption \ref{a1-2--}(iii) is satisfied.

\begin{example}\label{ex1-1}
The following conclusions hold.
\begin{itemize}
\item [(1)] Suppose that $\nu(z)=\frac{1}{|z|^{d+\beta_1}\log^{\beta_2}(2+|z|)}$ for every $|z|>1$ with
$\beta_1>2$ and $\beta_2\in \R$. Then,
\eqref{t2-1-2a} and \eqref{a1-2-1} hold with $K_0=2$ and $\gamma(x)\equiv 1$.
\medskip

\item [(2)] If $\nu(z)=\exp\left(-a|z|^{\beta}\right)$ for every $|z|>1$ with $a>0$ and $\beta>0$,
then  \eqref{t2-1-2a} and \eqref{a1-2-1} hold with
$K_0=\frac{1}{1-2^{-\frac{1}{\beta}}}+1$
and
$\gamma(x)=\exp\left(a\left(1-\left(1-\frac{1}{K_0}\right)^\beta\right)|x|^{\beta}\right)$.

\item [(3)] If $\nu(z)=0$ for every $|z|>1$, then
\eqref{t2-1-2a} and \eqref{a1-2-1} hold with $K_0=2$ and $\gamma(x)\equiv 1$.
\end{itemize}
\end{example}

\ \

The remainder of this paper is organized as follows. In the next section, we establish the existence of the corrector $\phi$ associated with the dual operator for $L^\w$. In Section \ref{section3}, we first prove the existence of a weak solution to the scaled resolvent equation \eqref{e1-6}, and then derive the local $W^{1,q}(\R^d;\R^d)$ regularity for the corrector $\phi$ with $q\in (1,2)$, as well as the relationships between $\phi$ and the drift field $b(x;\w)$. The final section is devoted to the proof of Theorem \ref{t1-2}.

\medskip

\noindent{\bf Notation}\,\, Throughout the paper, let
$B(x,R):=\{z\in \R^d: |z-x|<R\}$ be the open ball with center $x\in \R^d$ and radius $R>0$.
For every $1\le p\le \infty$, let $L^p(\R^d)$ denote the $L^p$ space with respect to the Lebesgue measure
endowed with the norm $\|f\|_{L^p(\R^d)}:=\left(\int_{\R^d}|f(x)|^p \,dx\right)^{1/p}$ for $1\le p<\infty$ and $\|f\|_{L^\infty(\R^d)}=\sup_{x\in \R^d}|f(x)|$. For any $1\le p<\infty$, let $W^{1,p}(\R^d)$ denote the Sobolev space as follows
\begin{align*}
W^{1,p}(\R^d):=\Big\{f\in L^p(\R^d):&\,\,\text{the\ weak\ gradient}\ \nabla f\ \text{exists}\\
&\,\,\text{and}\ \|f\|_{W^{1,p}(\R^d)}:=\|f\|_{L^p(\R^d)}+\left(\int_{\R^d}|\nabla f(x)|^p\, dx\right)^{1/p}<\infty\Big\}.
\end{align*}
For every $\beta\in (0,1)$ and $1\le p<\infty$, let $W^{\beta,p}(\R^d)$ be the Sobolev-Slobodeckij space
(i.e., the fractional Sobolev space) as follows
\begin{align*}
W^{\beta,p}(\R^d)&:=\left\{f\in L^p(\R^d):\ \|f\|_{W^{\beta,p}(\R^d)}:=
\|f\|_{L^p(\R^d)}+\left(\int_{\R^d}\int_{\R^d}\frac{|f(x)-f(y)|^p}{|x-y|^{d+\beta p}}\,dx\,dy\right)^{1/p}<\infty\right\}.
\end{align*}

\section{The corrector and homogenized coefficients}
Let $L^\w$ be the random operator given by \eqref{e:operator1}. Since  
 $L_0$ is a non-local symmetric L\'evy operator 
 and the random drift $b(\cdot;\w)$ is divergence-free such that \eqref{e1-5} is satisfied, the dual operator corresponding to $L^\w$ is formally given by
$$L^\w_*f(x)=L_0f(x)-\langle b(x;\w), \nabla f(x)\rangle.$$ The weak solution to the following equation
\begin{equation}\label{e:equ-01}
L^\w_*\phi_k(\cdot;\w)(x)=L_0 \phi_k(\cdot;\w)(x)-\langle b(x;\w), \nabla \phi_k(x;\w)\rangle=b_k(x;\w),
\quad x\in \R^d, 1\le k\le d
\end{equation} plays an important role in the homogenization theory
for $L^\w$, and it is usually called the (dual) corrector in the literature; e.g. see \cite{Fe}. Here we use the notation $\phi(x;\w)=(\phi_1(x;\w),\cdots, \phi_d(x;\w))$ with
$\phi_k(x;\w)$ denoting the $k$-th coordinate of $\phi(x;\w)$ for all $1\le k\le d$, and the same for $b(x;\w)=(b_1(x;\w),\cdots, b_d(x;\w))$.  The main purpose of this section is to verify the existence of the corrector $\phi(x;\w)$ to the equation \eqref{e:equ-01}, which is involved in the homogenized matrix coefficients $\{\bar a_{jk}\}_{1\le j,k\le d}$ given in \eqref{e1-7} for the limit operator $\bar L$.

\subsection{The existence of the corrector}

We begin with the following simple lemma.

\begin{lemma}\label{l1-1} Under Assumption $\ref{a1-2--}$, there exist positive constants $c_1\le c_2$ such that for all $\xi\in \R^d$,
\begin{equation}\label{l1-1-1}
c_1\left(|\xi|^2\I_{\{|\xi|\le 1\}}+
|\xi|^\alpha\I_{\{|\xi|> 1\}}\right)\le \int_{\R^d}(1-e^{i\langle \xi, z\rangle })\nu(z)\,dz\le c_2\left(|\xi|^2\I_{\{|\xi|\le 1\}}+
|\xi|^\alpha\I_{\{|\xi|> 1\}}\right).
\end{equation}
\end{lemma}
\begin{proof}
By the change of variable, for all $\xi\in \R^d$,
\begin{align*}
\int_{\{|z|\le 1\}}\frac{1-e^{-i\langle \xi, z\rangle}}{|z|^{d+\alpha}}\,dz&=
|\xi|^\alpha \int_{\{|z|\le |\xi|\}}\frac{1-e^{-i\left\langle \frac{\xi}{|\xi|}, z\right\rangle}}{|z|^{d+\alpha}}\,dz=|\xi|^\alpha \int_{\{|z|\le |\xi|\}}
\frac{1-{\rm cos}\left(\left\langle \frac{\xi}{|\xi|}, z\right\rangle\right)}{|z|^{d+\alpha}}\,dz.
\end{align*}

Note that there exist positive constants $r_0$, $c_1$ and $c_2$ such that for all $\xi,z\in \R^d$ with $|z|\le r_0$,
\begin{align*}
c_1\left|\left\langle \frac{\xi}{|\xi|}, z\right\rangle\right|^2\le 1-{\rm cos}\left(\left\langle \frac{\xi}{|\xi|}, z\right\rangle\right)\le c_2\left|\left\langle \frac{\xi}{|\xi|}, z\right\rangle\right|^2,
\end{align*}
which implies immediately that for all $\xi \in \R^d$ with $|\xi|\le r_0$,
\begin{align*}
c_3|\xi|^2\le \int_{\{|z|\le 1\}}\frac{1-e^{-i\langle \xi, z\rangle}}{|z|^{d+\alpha}}\,dz\le c_4|\xi|^2.
\end{align*}
On the other hand, for all $\xi \in\R^d$ with $|\xi|>r_0$,
\begin{align*}
c_5\le \int_{\{|z|\le r_0\}}
\frac{1-{\rm cos}\left(\left\langle \frac{\xi}{|\xi|}, z\right\rangle\right)}{|z|^{d+\alpha}}\,dz\le
\int_{\{|z|\le |\xi|\}}
\frac{1-{\rm cos}\left(\left\langle \frac{\xi}{|\xi|}, z\right\rangle\right)}{|z|^{d+\alpha}}\,dz\le
\int_{\R^d}
\frac{1-{\rm cos}\left(\left\langle \frac{\xi}{|\xi|}, z\right\rangle\right)}{|z|^{d+\alpha}}\,dz\le c_6,
\end{align*}
which yields that for all $\xi \in \R^d$ with $|\xi|>r_0$,
\begin{align*}
c_7|\xi|^\alpha\le \int_{\{|z|\le 1\}}\frac{1-e^{-i\langle \xi, z\rangle}}{|z|^{d+\alpha}}\,dz\le c_8|\xi|^\alpha.
\end{align*}
Combining with both estimates above, we find that
\begin{equation}\label{l1-1-2}
c_9\left(|\xi|^2\I_{\{|\xi|\le 1\}}+
|\xi|^\alpha\I_{\{|\xi|> 1\}}\right)\le \int_{\{|z|\le 1\}}\frac{1-e^{-i\langle \xi, z\rangle}}{|z|^{d+\alpha}}\,dz\le c_{10}\left(|\xi|^2\I_{\{|\xi|\le 1\}}+
|\xi|^\alpha\I_{\{|\xi|> 1\}}\right).
\end{equation}

Furthermore, by using the fact $\nu(z)=\nu(-z)$ for all $z\in \R^d$ and Taylor's expansion,
\begin{align*}
\int_{\{|z|> 1\}}(1-e^{-i\langle \xi, z\rangle})\nu(z)\,dz
&\le c_{11}\left[|\xi|^2\left(\int_{\{|z|>1\}}|z|^2 \nu(z)\,dz\right)\I_{\{|\xi|\le 1\}}+
\left(\int_{\{|z|>1\}}\nu(z)\,dz\right)\I_{\{|\xi|> 1\}}\right]\\
&\le c_{12}\left(|\xi|^2\I_{\{|\xi|\le 1\}}+
|\xi|^\alpha\I_{\{|\xi|> 1\}}\right).
\end{align*}

According to this and \eqref{l1-1-2}, we can prove \eqref{l1-1-1}.
\end{proof}

To move further, we need to introduce some notations. Define
\begin{align*}
\LL_0 \tilde F(\w)={\rm p.v.}\int_{\R^d}(\tilde F(\tau_z \w)-\tilde F(\w))\nu(z)\, dz,\quad \tilde F\in \D.
\end{align*}
Then,
$$
\LL_0 \tilde F(\tau_x \w)=L_0 F(\cdot;\w)(x),\quad x\in \R^d, \tilde F \in \D,
$$
where $F(x;\w):=\tilde F(\tau_x \w)$.
According to the stationary property of $\{\tau_x\}_{x\in \R^d}$ and the fact $\nu(z)=\nu(-z)$ for all $z\in \R^d$, we have \begin{equation}\label{e1-4}
\Ee[\LL_{0}\tilde F \cdot \tilde G]=\Ee[\tilde F\cdot \LL_{0}\tilde G]=-
\frac{1}{2}\Ee\left[\int_{\R^d}(\tilde F(\tau_z\w)-\tilde F(\w))(\tilde G(\tau_z\w)-\tilde G(\w))\nu(z)\,dz\right],\quad \tilde F,\tilde G\in \D.
\end{equation} In particular, for all $\tilde F\in \D$,
$$-\Ee[\LL_{0}\tilde F \cdot \tilde F]=\frac{1}{2}
\Ee\left[\int_{\R^d}(\tilde F(\tau_z\w)-\tilde F(\w))^2\nu(z)\,dz\right].$$
On the other hand, according to \eqref{e1-1a},
\begin{equation}\label{e1-6a}
\begin{split}
-\Ee[\LL_{0}\tilde F \cdot \tilde F]&=
-\Ee\left[\left(\int_{\R^d}(T_{-z}\tilde F(\w)-\tilde F(\w))\nu(z)\,dz\right)\cdot \tilde F(\w)\right]\\
&=-\Ee\left[\int_{\R^d} \left(\int_{\R^d}(e^{-i\langle \xi, z\rangle}-1)\nu(z)\,dz\right)[U(d\xi)\tilde F\cdot \tilde F]\right],\quad \tilde F\in \D.
\end{split}
\end{equation}

Next, we will prove the existence of the corrector $\phi(x;\w)$ to the equation \eqref{e:equ-01}.
\begin{theorem}\label{t2-1}
Suppose that Assumptions $\ref{a1-2--}$ and $\ref{a1-2}$ hold. Then
there exist $\phi:\R^d\times \Omega \to \R^d$ and a $\Pp$-null set $\Lambda \subset \Omega$ such that the following properties are satisfied.
\begin{itemize}
\item [(i)]
For every $1\le k \le d$, $f\in C_c^1(\R^d)$ and $\w\notin \Lambda$,
\begin{equation}\label{t2-1-1}
\begin{split}
&\frac{1}{2}\int_{\R^d}\int_{\R^d}
\left(\phi_k(x+z;\w)-\phi_k(x;\w)\right)
\left(f(x+z)-f(x)\right)\nu(z)\,dz\,dx\\
& -\sum_{j=1}^d\int_{\R^d}b_j(x;\w)\phi_k(x;\w)\frac{\partial f(x)}{\partial x_j}\,dx=-\int_{\R^d}b_k(x;\w)f(x)\,dx.
\end{split}
\end{equation}

\item [(ii)] $\Ee\left[\phi(x;\w)\right]=0$ for every $x\in \R^d$, and the
co-cycle property holds for all $\w\notin \Lambda$ and $x,z\in \R^d$,
\begin{equation}\label{t2-1-1a}
\phi(x+z;\w)-\phi(x;\w)=\phi(z;\tau_x \w).
\end{equation}

\item [(iii)] $\phi(\cdot;\w)\in W_{loc}^{\alpha/2,2}(\R^d,\R^d)$ for every $\w\notin \Lambda$,
\begin{equation}\label{t2-1-2}
\Ee\left[\int_{\R^d}|\phi(x;\w)|^2\nu(x)\,dx\right]<\infty,
\end{equation} and for every bounded subset $D\subset \R^d$,
\begin{equation}\label{t2-1-3}
\Ee\left[\int_D |\phi(x;\w)|^2 \,dx\right]+\Ee\left[\int_D\int_{\R^d}\left|\phi(x+z;\w)-\phi(x;\w)\right|^2\nu(z)\,dz\,dx\right]<\infty.
\end{equation}
\end{itemize}
\end{theorem}
\begin{proof}
(1) For any $R\ge1$,  let $\rho_R:\R \to
\R$ be a smooth cut-off function such that
\begin{align*}
\rho_R(s)=
\begin{cases}
s,\ \ &\ |s|\le R,\\
\in [-R,R],\ \ &\ R<|s|<2R,\\
0,\ \ &\ |s|\ge 2R,
\end{cases}
\end{align*}
$\sup_{R\ge 1}\sup_{s\in \R}|\rho_R'(s)|\le 2$ and $\rho_R(s)=-\rho_R(-s)$ for every $s\in \R$.

For every $\theta>0$ and $R\ge 1$, define
$$
\LL_{*,\theta,R} \tilde F(\w)=\LL_0 \tilde F(\w)+\theta \sum_{j=1}^d D_j (D_j \tilde F)(\w)+\sum_{j=1}^d \tilde b_j^R(\w)D_j \tilde F(\w),\quad\tilde F\in \D,
$$
where $\LL_0$ is given by \eqref{e1-6}, and
\begin{equation}\label{t1-1-2}
\tilde b_j^R(\w):=\sum_{l=1}^d D_l (\rho_R(\tilde H_{jl}))(\w).
\end{equation} According to \eqref{e1-3}, we can verify that
$$\LL_{*,\theta,R} \tilde F(\tau_x \w)=L^\w_{*,\theta, R} F(\cdot;\w)(x),$$ where
$F(x;\w)=\tilde F(\tau_x\w)$, $b^R(x;\w)=\tilde b^R(\tau_x \w)$ and
$$L^\w_{*,\theta, R}f(x)=L_0f(x)+\theta \Delta f(x)-\left\langle b^R(x;\w), \nabla f(x)\right\rangle,\quad f\in C_b^2(\R^d).$$

Furthermore, define
$$
\EE_{*,\theta,R}(\tilde F, \tilde G)=-\Ee[\LL_{*,\theta,R}\tilde F \cdot \tilde G],\quad \tilde F, \tilde G\in \D.
$$
By \eqref{l1-1-1} and \eqref{e1-6a}, we obtain that
$$
-\Ee[\LL_0 \tilde F \cdot \tilde F]\le c_0 \int_{\R^d}\left(1+|\xi|^2\right)\Ee[U(d\xi)\tilde F \cdot \tilde F]
=c_0\|\tilde F\|_{\HH_1}^2,\quad \tilde F\in \D.
$$ On the other hand, note that, for any $\tilde F\in \D$,
\begin{equation}\label{e:comment}\begin{split}
\sum_{j=1}^d\Ee[\tilde b_j^R \tilde F D_j \tilde F]
&=\sum_{l,j=1}^d\Ee[D_l(\rho_R(\tilde H_{jl})) \tilde F D_j \tilde F]\\
&=-\sum_{l,j=1}^d\Ee[\rho_R(\tilde H_{jl})D_l\tilde F D_j\tilde F]-\sum_{l,j=1}^d \Ee[\rho_R(\tilde H_{jl})\tilde F D_l(D_j\tilde F)]=0,\end{split}
\end{equation}
thanks to the anti-symmetry of $\{\tilde H_{jl}\}_{1\le j,l\le d}$, the fact that $\rho_R(s)=-\rho_R(-s)$ for all $s\in \R$ and the integration by parts formula \eqref{e1-2a}.
These, along with the boundedness of $\tilde b_j^R(\w)$, yield that
\begin{align*}
\EE_{*,\theta,R}(\tilde F, \tilde G)\le c_1(\theta,R)\|\tilde F\|_{\HH_1}\|\tilde G\|_{\HH_1},\ \EE_{*,\theta,R}(\tilde F, \tilde F)\ge
c_2(\theta,R)\sum_{j=1}^d\Ee[|D_j \tilde F|^2],\quad \tilde F, \tilde G\in \D.
\end{align*}
Therefore, $\EE_{*,\theta,R}$ can be extended to a bilinear closed form on $\HH_1\times \HH_1$. In particular, according to the
Lax-Milgram theorem (cf.\ see \cite[Section 2]{PV}), there exists a unique $\tilde \phi_{\theta,R}=(\tilde \phi_{1,\theta,R},\cdots, \tilde \phi_{d,\theta,R})
$ with $\tilde \phi_{k,\theta,R} \in \HH_1$, $1\le k \le d$, such that for every $1\le k \le d$ and $\tilde F \in \HH_1$,
\begin{equation}\label{t2-1-4}
\begin{split}
  \theta\Ee[\tilde\phi_{k,\theta,R}\tilde F]+\EE_{*,\theta,R}(\tilde \phi_{k,\theta,R},\tilde F)
&=\theta
\Ee[\tilde \phi_{k,\theta,R}\tilde F]+\theta\sum_{j=1}^d \Ee[D_j\tilde \phi_{k,\theta,R}D_j \tilde F]+\sum_{j=1}^d\Ee[\tilde b_j^R \tilde \phi_{k,\theta,R} D_j \tilde F]\\
&\quad +
\frac{1}{2}\Ee\left[\int_{\R^d}
(\tilde \phi_{k,\theta,R}(\tau_z \w)-\tilde \phi_{k,\theta,R}(\w))(\tilde F(\tau_z \w)-\tilde F(\w))\nu(z)\,dz\right]\\
&=-\Ee[\tilde b_k\tilde F],
\end{split}
\end{equation} where we also used the anti-symmetry of $\{\tilde H_{jl}\}_{1\le j,l\le d}$, the fact that $\rho_R(s)=-\rho_R(-s)$ for all $s\in \R$ and the integration by parts formula \eqref{e1-2a}.

According to the fact $\tilde \phi_{k,\theta,R}\in \HH_1$ and the standard approximation procedure
(also due to the fact that $\tilde b_j^R$ is bounded), we get by \eqref{e:comment} that
\begin{equation}\label{e:formula1}\begin{split}
\sum_{j=1}^d\Ee[\tilde b_j^R \tilde \phi_{k,\theta,R} D_j \tilde \phi_{k,\theta,R}]=0.
\end{split}
\end{equation}
Then, taking $\tilde F=\tilde \phi_{k,\theta,R}$ in \eqref{t2-1-4} (noting that the test function
$\tilde \phi_{k,\theta,R} \in \HH_1$) and using \eqref{e:formula1}, we further derive that
\begin{align*}
&  \theta\Ee[|\tilde \phi_{k,\theta,R}|^2]+\frac{1}{2}
\Ee\left[\int_{\R^d}
(\tilde \phi_{k,\theta,R}(\tau_z \w)-\tilde \phi_{k,\theta,R}(\w))^2\nu(z)\,dz\right]+
\theta\sum_{j=1}^d\Ee[|D_j \tilde \phi_{k,\theta,R}|^2]=-\Ee[\tilde b_k \tilde \phi_{k,\theta,R}].
\end{align*}

Under Assumption \ref{a1-2}, we obtain that for every $\delta\in (0,1)$,
\begin{align*}
&  |\Ee[\tilde b_k \tilde \phi_{k,\theta,R}]|=\left|\sum_{l=1}^d \Ee[D_l \tilde H_{kl}\tilde \phi_{k,\theta,R}]\right|\\
&=\left|-\sum_{l=1}^d\Ee\left[i\int_{\R^d}\xi_l\,[U(d\xi)\tilde H_{kl}\cdot \tilde \phi_{k,\theta,R}]\right]\right|\le \sum_{l=1}^d \int_{\R^d}|\xi_l||\Ee[U(d\xi)\tilde H_{kl}\cdot \tilde \phi_{k,\theta,R}]|\\
&\le \delta\int_{\R^d}\left(|\xi|^2 \I_{\{|\xi|\le 1\}}+|\xi|^\alpha\I_{\{|\xi|>1\}}\right)\,\Ee[
U(d\xi)\tilde \phi_{k,\theta,R}\cdot \tilde \phi_{k,\theta,R} ]+c_3(\delta,d)\sum_{l=1}^d\int_{\R^d}\left(1+|\xi|^2\right)\,
\Ee[U(d\xi)\tilde H_{kl}\cdot \tilde H_{kl} ]\\
&\le c_4\delta \Ee\left[\int_{\R^d}
(\tilde \phi_{k,\theta,R}(\tau_z \w)-\tilde \phi_{k,\theta,R}(\w))^2\nu(z)\,dz\right]+c_4c_3(\delta,d)\sum_{l=1}^d\|\tilde H_{kl}\|_{\HH_1}^2.
\end{align*}
Here $c_3(\delta,d)>0$ may depend on $\delta$ and $d$, $c_4>0$ is independent of $\delta$,
$U(d\xi)$ denotes the projection valued measure on $\HH$ associated with the unitary operator $\{T_x\}_{x\in \R^d}$ defined by \eqref{e1-1a}, in the second inequality
we have used Young's inequality, and the last inequality follows from \eqref{l1-1-1}, \eqref{e1-4} and \eqref{e1-6a}.

Putting all the estimates above together and choosing $\delta$ small enough, we arrive at that for $1\le k \le d$,
\begin{align*}
 \theta\Ee[|\tilde \phi_{k,\theta,R}|^2]+\frac{1}{4}
\Ee\left[\int_{\R^d}
(\tilde \phi_{k,\theta,R}(\tau_z \w)-\tilde \phi_{k,\theta,R}(\w))^2\nu(z)\,dz\right]+
\theta\sum_{j=1}^d\Ee[|D_j \tilde \phi_{k,\theta,R}|^2]
\le c_5\sum_{l=1}^d\|\tilde H_{kl}\|_{\HH_1}^2.
\end{align*}
In particular,
\begin{equation}\label{t2-1-5}
\begin{split}
&\sup_{\theta\in (0,1),R\ge 1}\theta\Ee\left[|\tilde \phi_{k,\theta,R}|^2\right]\le c_5\sum_{l=1}^d\|\tilde H_{kl}\|_{\HH_1}^2,\quad
\sup_{\theta\in (0,1),R\ge 1}\theta\left(\sum_{j=1}^d\Ee[|D_j \tilde\phi_{k,\theta,R}|^2]\right)\le c_5\sum_{l=1}^d\|\tilde H_{kl}\|_{\HH_1}^2,\\
&\sup_{\theta\in (0,1),R\ge 1}\Ee\left[\int_{\R^d}
(\tilde \phi_{k,\theta,R}(\tau_z \w)-\tilde \phi_{k,\theta,R}(\w))^2\nu(z)\,dz\right]\le 4c_5\sum_{l=1}^d\|\tilde H_{kl}\|_{\HH_1}^2.
\end{split}
\end{equation}

(2) Define $\phi_{\theta,R}=\left(\phi_{1,\theta,R},\cdots,\phi_{d,\theta,R}\right)$ by
\begin{align}\label{t2-1-5a}
\phi_{k,\theta,R}(x;\w)=\tilde \phi_{k,\theta,R}(\tau_x \w)-\tilde \phi_{k,\theta,R}(\w),\quad (x,\w)\in \R^d\times \Omega,\ 1\le k \le d.
\end{align}
It is easy to verify that $\Ee\left[\phi_{k,\theta,R}(x;\w)\right]=0$ for every $x\in \R^d$, and
\begin{equation}\label{t2-1-6}
\phi_{k,\theta,R}(x+z;\w)-\phi_{k,\theta,R}(x;\w)=\phi_{k,\theta,R}(z;\tau_x \w),\quad x,z\in \R^d,\ \w\in \Omega.
\end{equation}
This along with \eqref{t2-1-5} yields that
\begin{equation}\label{t2-1-6a}
\begin{split}
\sup_{\theta\in (0,1),R\ge 1}\Ee\left[\int_{\R^d}|\phi_{k,\theta,R}(z;\w)|^2\nu(z)\,dz\right]&=
\sup_{\theta\in (0,1),R\ge 1}\Ee\left[\int_{\R^d}\left|\tilde \phi_{k,\theta,R}(\tau_z\w)-\tilde \phi_{k,\theta,R}(\w)\right|^2\nu(z)\,dz\right]<\infty.
\end{split}
\end{equation} In particular, since $\nu(z)= |z|^{-d-\alpha}\ge1$ with $|z|\le1$, it holds that
$$\sup_{\theta\in (0,1),R\ge 1}\Ee\left[\int_{\{|z|\le 1\}}|\phi_{k,\theta,R}(z;\w)|^2 \,dz\right]\le \sup_{\theta\in (0,1),R\ge 1}\Ee\left[\int_{\R^d}|\phi_{k,\theta,R}(z;\w)|^2\nu(z)\,dz\right]<\infty.$$
Furthermore, by \eqref{t2-1-6}, we find
\begin{align*}
&|B(0,1)|\cdot\left(\sup_{\theta\in (0,1),R\ge 1}\int_{B(0,1)+B(0,1)}\int_\Omega |\phi_{\theta, R}(z;\w)|^2\,\Pp(d\w)\,dz\right)\\
&=\sup_{\theta\in (0,1),R\ge1}\int_{B(0,1)}\int_{B(0,1)}\int_\Omega |\phi_{\theta, R}(z_1+z_2;\w)|^2\,\Pp(d\w)\,dz_1\,dz_2\\
&=\sup_{\theta\in (0,1),R\ge1}\int_{B(0,1)}\int_{B(0,1)}\int_\Omega |\phi_{\theta, R}(z_2;\w)+\phi_{\theta, R}(z_1;\tau_{z_2}\w)|^2\,\Pp(d\w)\,dz_1\,dz_2\\
&\le 2|B(0,1)|\left(\sup_{\theta\in (0,1),R\ge1}\int_{B(0,1)}\int_\Omega |\phi_{\theta, R}(z_2;\w)|^2\,\Pp(d\w)\,dz_2+\sup_{\theta\in (0,1),R\ge1}\int_{B(0,1)}\int_\Omega |\phi_{\theta, R}(z_1;\w)|^2\,\Pp(d\w)\,dz_1\right)\\
&<\infty,
\end{align*}where we used the stationary property of the transformation $\{\tau_x\}_{x\in \R^d}$, and
$$B(0,1)+B(0,1):=\{z\in \R^d:\ \text{there\ exist}\ z_1,z_2\in B(0,1)\ \text{such\ that}\ z=z_1+z_2\}.$$
 Applying the estimate above iteratively, we can obtain that, for every bounded subset $D\subset \R^d$,
\begin{equation}\label{t2-1-8}
\begin{split}
&\sup_{\theta\in (0,1),R\ge 1}\int_D\int_{\Omega}|\phi_{k,\theta,R}(z;\w)|^2\,\Pp(d\w)\,dz<\infty.\end{split}
\end{equation}

On the other hand, according to \eqref{t2-1-6}, it holds that for every bounded subset $D\subset \R^d$,
\begin{equation}\label{t2-1-8a}
\begin{split}
& \sup_{\theta\in (0,1),R\ge 1}\int_{\Omega}\left(\int_D\int_{\R^d}
\left|\phi_{k,\theta,R}(x+z;\w)-\phi_{k,\theta,R}(x;\w)\right|^2\nu(z)\,dz\,\,dx\right)\Pp(d\w)\\
&=\sup_{\theta\in (0,1),R\ge 1}
 \int_{D}\int_{\Omega}\int_{\R^d}\left|  \phi_{k,\theta,R}(z;\tau_x\w)\right|^2\nu(z)\,dz\,\Pp(d\w)\,dx\\
 &=|D|\cdot \sup_{\theta\in (0,1),R\ge 1}\int_{\Omega}\int_{\R^d}\left|  \phi_{k,\theta,R}(z;\w)\right|^2\,\nu(z)\,dz\,\Pp(d\w)<\infty.
\end{split}
\end{equation} In particular,
\begin{align*}&\sup_{\theta\in (0,1),R\ge 1}\int_{\Omega}\left(\int_D\int_{\{|z|\le 1\}}\left|\phi_{k,\theta,R}(x+z;\w)-\phi_{k,\theta,R}(x;\w)\right|^2 \frac{1}{|z|^{d+\alpha}}\,dz\,dx\right)\,\Pp(d\w)\\
&\le |D|\cdot \sup_{\theta\in (0,1),R\ge 1}\int_\Omega\int_{\R^d}\left|  \phi_{k,\theta,R}(z;\w)\right|^2\nu(z)\,dz\,\Pp(d\w)<\infty.\end{align*}

Therefore,
$\{\phi_{\theta,R}\}_{\theta\in (0,1),R\ge 1}$ is weakly compact in $L^2 (\Omega, W_{loc}^{\alpha/2,2}(\R^d,\R^d);\Pp )$, and so
there exist a subsequence $\{\phi_{\theta_m,R_m}\}_{m\ge 1}$ (with $\theta_m \to 0$ and $R_m \to \infty$ as $m\to \infty$) and  $\phi\in L^2 (\Omega, W_{loc}^{\alpha/2,2}(\R^d,\R^d);\Pp)$ such that
$\phi_{\theta_m,R_m}$ converges weakly to $\phi$ in $L^2 (\Omega, W_{loc}^{\alpha/2,2}(\R^d,\R^d);\Pp )$ as $m \to \infty$.
Furthermore, by \eqref{t2-1-6a}, without loss of generality, we can also conclude that $\phi_{\theta_m,R_m}$ converges weakly to $\phi$ in
$L^2(\R^d\times\Omega, \R^d;\nu(z)\,dz\times\Pp)$ as $m \to \infty$. In particular,
$\phi\in L^2(\Omega, W_{loc}^{\alpha/2,2}(\R^d,\R^d);\Pp)\cap L^2(\R^d\times\Omega, \R^d;\nu(z)\,dz\times\Pp)$.
Using \eqref{t2-1-6} and following the approximation argument in the proof of \cite[Proposition 4.3]{PZ}, we know that
the co-cycle property \eqref{t2-1-1a} still holds for $\phi(x;\w)$, and that $\Ee\left[\phi(x;\w)\right]=0$ for every
$x\in \R^d$. Thus, from \eqref{t2-1-6a}, \eqref{t2-1-8} and \eqref{t2-1-8a}, we can prove \eqref{t2-1-2} and \eqref{t2-1-3}.

(3) In this part, we are going to prove \eqref{t2-1-1}. Let $\tilde G\in \D$. Taking $\tilde F(\w):=\tilde G(\tau_{-x}\w)$ in \eqref{t2-1-4}, we  obtain that
for every $f\in C_c^1(\R^d)$ and $1\le k \le d$,
\begin{equation}\label{t2-1-7}
\begin{split}
& \theta_m \int_{\Omega}\left(\int_{\R^d} \tilde \phi_{k,\theta_m,R_m}(\w)\tilde G(\tau_{-x}\w)f(x)\,dx\right)\,\Pp(d\w)\\
&+\theta_m\sum_{j=1}^d\int_{\Omega}\left(\int_{\R^d}D_j\tilde \phi_{k,\theta_m,R_m}(\w) f(x) D_j \tilde G(\tau_{-x}\w)\,dx\right)\,\Pp(d\w)\\
&+
\frac{1}{2}\int_{\Omega}\left(\int_{\R^d}\int_{\R^d}(\tilde \phi_{k,\theta_m,R_m}(\tau_z \w)-\tilde \phi_{k,\theta_m,R_m}(\w))
(\tilde G(\tau_{-x+z}\w)-\tilde G(\tau_{-x}\w))\nu(z)\,dzf(x)\,dx\right)\,\Pp(d\w)\\
&-\sum_{j=1}^d \int_{\Omega}\left(\int_{\R^d}\tilde b_{j}^{R_m}(\w)D_j\tilde \phi_{k,\theta_m,R_m}(\w)f(x)\tilde G(\tau_{-x}\w)\,dx\right)\,\Pp(d\w)\\
&=-
\int_{\Omega}\left(\int_{\R^d}\tilde b_k(\w)\tilde G(\tau_{-x}\w)f(x)\,dx\right)\,\Pp(d\w),
\end{split}
\end{equation}
where we used the integration by parts formula \eqref{e1-2a}, the anti-symmetry of $\{\tilde H_{jl}\}_{1\le j,l\le d}$ and the fact that $\rho_R(s)=-\rho_R(-s)$ for all $s\in \R$.

According to \eqref{t2-1-5},  for all $1\le k,j\le d$, we have
\begin{equation*}
\lim_{m \to \infty}\theta_m \int_{\Omega}\left(\int_{\R^d} \tilde \phi_{k,\theta_m,R_m}(\w)f(x)\tilde G(\tau_{-x}\w)\,dx\right)\,\Pp(d\w)=0
\end{equation*} and
\begin{equation*}
\lim_{m \to \infty}\theta_m \int_{\Omega}\left(\int_{\R^d}D_j\tilde \phi_{k,\theta_m,R_m}(\w) f(x) D_j \tilde G(\tau_{-x}\w)\,dx\right)\,\Pp(d\w)=0.
\end{equation*}
On the other hand, it holds that
\begin{align*}
& \int_{\Omega}\left(\int_{\R^d}\int_{\R^d}(\tilde \phi_{k,\theta_m,R_m}(\tau_z \w)-\tilde \phi_{k,\theta_m,R_m}(\w))
(\tilde G(\tau_{-x+z}\w)-\tilde G(\tau_{-x}\w))\nu(z)f(x)\,dz\,dx\right)\,\Pp(d\w)\\
&=\lim_{\delta \downarrow 0}\Bigg(\int_{\Omega}\left(\int_{\R^d}\int_{\{|z|> \delta\}}(\tilde \phi_{k,\theta_m,R_m}(\tau_z \w)-\tilde \phi_{k,\theta_m,R_m}(\w))
\tilde G(\tau_{-x+z}\w)\nu(z)f(x)\,dz\,dx\right)\,\Pp(d\w)\\
&\qquad\qquad -\int_{\Omega}\left(\int_{\R^d}\int_{\{|z|>\delta\}}(\tilde \phi_{k,\theta_m,R_m}(\tau_z \w)-\tilde \phi_{k,\theta_m,R_m}(\w))
\tilde G(\tau_{-x}\w)\nu(z)f(x)\,dz\,dx\right)\,\Pp(d\w)\Bigg)\\
&=:\lim_{\delta \downarrow 0}\left(I_{1,\delta}-I_{2,\delta}\right).
\end{align*}
Using the stationary of the transformation $\{\tau_{x}\}_{x\in \R^d}$, we get
\begin{align*}
I_{1,\delta}&=\int_{\Omega}\left(\int_{\R^d}\int_{\{|z|> \delta\}}(\tilde \phi_{k,\theta_m,R_m}(\tau_{x} \w)-\tilde \phi_{k,\theta_m,R_m}(\tau_{x-z}\w))
\tilde G(\w)\nu(z)f(x)\,dz\,dx\right)\,\Pp(d\w)\\
&=\int_{\Omega}\left(\int_{\R^d}\int_{\{|z|> \delta\}}(\tilde \phi_{k,\theta_m,R_m}(\tau_{x+z} \w)-\tilde \phi_{k,\theta_m,R_m}(\tau_{x}\w))
\tilde G(\w)\nu(z)f(x+z)\,dz\,dx\right)\,\Pp(d\w),
\end{align*}
where the second equality is due to the change of variable $\tilde x=x-z$.
By the stationary of the transformation $\{\tau_{x}\}_{x\in \R^d}$ again, it is easy to see that
\begin{align*}
I_{2,\delta}=\int_{\Omega}\left(\int_{\R^d}\int_{\{|z|> \delta\}}(\tilde \phi_{k,\theta_m,R_m}(\tau_{x+z} \w)-\tilde \phi_{k,\theta_m,R_m}(\tau_{x}\w))
\tilde G(\w)\nu(z)f(x)\,dz\,dx\right)\,\Pp(d\w).
\end{align*}
Thus, putting all the estimates above together yields that
\begin{align*}
&\int_{\Omega}\left(\int_{\R^d}\int_{\R^d}(\tilde \phi_{k,\theta_m,R_m}(\tau_{z} \w)-\tilde \phi_{k,\theta_m,R_m}(\w))
(\tilde G(\tau_{-x+z}\w)-\tilde G(\tau_{-x}\w))\nu(z)f(x)\,dz\,dx\right)\,\Pp(d\w)\\
&=\int_{\Omega}\left(\int_{\R^d}\int_{\R^d}(\tilde \phi_{k,\theta_m,R_m}(\tau_{x+z} \w)-\tilde \phi_{k,\theta_m,R_m}(\tau_x\w))
(f(x+z)-f(x))\nu(z)\,dz\,dx\right)\tilde G(\w)\,\Pp(d\w)\\
&=\int_{\Omega}\left(\int_{\R^d} \int_{\R^d}\phi_{k,\theta_m,R_m}(z;\tau_x\w)
(f(x+z)-f(x))\nu(z)\,dz \,dx\right)\tilde G(\w)\,\Pp(d\w),
\end{align*}
where in the last equality we used \eqref{t2-1-5a}.
Since $\phi_{\theta_m,R_m}$ converges weakly to $\phi$ in
$L^2\left(\R^d\times\Omega, \R^d;\nu(z)dz\times\Pp\right)$ as $m \to \infty$, we have
for every $f\in C_c^1(\R^d)$ and $x\in \R^d$,
\begin{align*}
&\lim_{m \to \infty}\int_{\Omega}\left(\int_{\R^d}\phi_{k,\theta_m,R_m}(z;\tau_x\w)
(f(x+z)-f(x))\nu(z)\,dz\right)\tilde G(\w)\,\Pp(d\w)\\
&=\int_{\Omega}\left(\int_{\R^d}\phi_k(z;\tau_x\w)
(f(x+z)-f(x))\nu(z)\,dz \right)\tilde G(\w)\,\Pp(d\w)\\
&=\int_{\Omega}\left(\int_{\R^d} (\phi_k(x+z;\w)-\phi_k(x;\w) )
(f(x+z)-f(x))\nu(z)\,dz\right) \tilde G(\w)\,\Pp(d\w),
\end{align*} where in the last equality we used the co-cycle property \eqref{t2-1-1a}. By \eqref{t2-1-6a} and the stationary of the transformation $\{\tau_x\}_{x\in \R^d}$, it is not difficult to verify that for every bounded subset $D \subset \R^d$,
\begin{align*}
\sup_{m\ge 1}\int_D\left(\int_{\Omega} \int_{\R^d}\phi_{k,\theta_m,R_m}(z;\tau_x\w)
 (f(x+z)-f(x))\nu(z)\,dz \tilde G(\w)\,\Pp(d\w)\right)^2\,dx<\infty
\end{align*} and
\begin{align*}
\lim_{R \uparrow \infty}\sup_{m\ge 1}\int_{\{|x|>R\}}\left(\int_{\Omega} \int_{\R^d}|\phi_{k,\theta_m,R_m}(z;\tau_x\w)|
|f(x+z)-f(x)|\nu(z)\,dz \tilde G(\w)\,\Pp(d\w)\right)\,dx=0.
\end{align*}
With all the properties above, we can  apply the routine limit arguments to deduce that for every $f\in C_c^1(\R^d)$,
\begin{align*}
& \lim_{m \to \infty}\int_{\Omega}\left(\int_{\R^d}\int_{\R^d} (\tilde \phi_{k,\theta_m,R_m}(\tau_{z} \w)-\tilde \phi_{k,m}(\w))
(\tilde G(\tau_{-x+z}\w)-\tilde G(\tau_{-x}\w))\nu(z)f(x)\,dz\,dx\right)\,\Pp(d\w)\\
&=\lim_{m \to \infty}\int_{\R^d} \left(\int_{\Omega} \int_{\R^d}\phi_{k,\theta_m,R_m}(z;\tau_x\w)
\left(f(x+z)-f(x)\right)\nu(z)\,dz \tilde G(\w)\,\Pp(d\w)\right)\,dx\\
&=\int_{\Omega} \left(\int_{\R^d}\int_{\R^d}\left(\phi_{k}(x+z;\w)-\phi_k(x;\w)\right)
\left(f(x+z)-f(x)\right)\nu(z)\,dz \,dx\right)\tilde G(\w)\,\Pp(d\w).
\end{align*}

Note that $\tilde \phi_{\theta_m,R_m}\in \HH_1$, so $\phi_{\theta_m,R_m}(\cdot;\w)\in W_{loc}^{1,2}(\R^d,\R^d)$ and
\begin{align*}
D_j \tilde \phi_{\theta_m,R_m}(\tau_x \w)=-\frac{\partial \phi_{\theta_m,R_m}(\cdot;\w)}{\partial x_j}(x).
\end{align*}
Then, by the integration by parts formula and the anti-symmetry of $\{\tilde H_{jl}\}_{1\le j,l\le d}$, we deduce that for every $f\in C_c^{1}(\R^d)$ and $\tilde G\in \D$,
\begin{align*}
& \int_{\Omega}\left(\int_{\R^d}\tilde b_{j}^{R_m}(\w)D_j\tilde \phi_{k,\theta_m,R_m}(\w)\tilde G(\tau_{-x}\w)f(x)\,dx\right)\,\Pp(d\w)\\
&=\int_{\Omega}\left(\int_{\R^d}\tilde b_{j}^{R_m}(\tau_x\w)D_j\tilde \phi_{k,\theta_m,R_m}(\tau_x\w)f(x)\,dx\right)\tilde G(\w)\,\Pp(d\w)\\
&=-\int_{\Omega}\left(\int_{\R^d}\tilde b_{j}^{R_m}(\tau_x\w)\frac{\partial \phi_{k,\theta_m,R_m}(x;\w)}{\partial x_j}f(x)\,dx\right)\tilde G(\w)\,\Pp(d\w)\\
&=\int_{\Omega}\left(\int_{\R^d}b_{j}^{R_m}(x;\w)\phi_{k,\theta_m,R_m}(x;\w)\frac{\partial f(x)}{\partial x_j}\,dx\right)\tilde G(\w)\,\Pp(d\w),
\end{align*} where $b_{j}^{R_m}(x;\w)=\tilde b_{j}^{R_m}(\tau_x\w).$
On the other hand, according to the definition of $b^R(x;\w)$ and the fact that $\tilde H_{jl}\in \HH_1$ for all $1\le j,l\le d$, it is easy to verify that for every bounded subset $D\subset \R^d$,
\begin{align*}
\lim_{m \to \infty}\int_{\Omega}\left(\int_D|b(x;\w)-b^{R_m}(x;\w)|^2\,dx\right)\,\Pp(d\w)=0.
\end{align*}
Therefore, for every $f\in C_c^{1}(\R^d)$ and $\tilde G\in \D$,
\begin{align*}
&\lim_{m \to \infty}\int_{\Omega}\left(\int_{\R^d}\tilde b_{j}^{R_m}(\w)D_j\tilde \phi_{k,\theta_m,R_m}(\w)f(x)\tilde G(\tau_{-x}\w)\,dx\right)\,\Pp(d\w)\\
&=\lim_{m \to \infty}\int_{\Omega}\left(\int_{\R^d}b^{R_m}_{j}(x;\w)\phi_{k,\theta_m,R_m}(x;\w)\frac{\partial f(x)}{\partial x_j}\,dx\right)\tilde G(\w)\,\Pp(d\w)\\
&=\lim_{m \to \infty}\int_{\Omega}\left(\int_{\R^d}b_{j}(x;\w)\phi_{k,\theta_m,R_m}(x;\w)\frac{\partial f(x)}{\partial x_j}\,dx\right)\tilde G(\w)\,\Pp(d\w)\\
&=\int_{\Omega}\left(\int_{\R^d}b_{j}(x;\w)\phi_k(x;\w)\frac{\partial f(x)}{\partial x_j}\,dx\right)\tilde G(\w)\,\Pp(d\w),
\end{align*}
where the last equality is due to the fact that $\phi_{\theta_m,R_m}$ converges weakly to $\phi$ in $L^2 (\Omega, W_{loc}^{\alpha/2,2}(\R^d,\R^d);\Pp )$ as $m\to \infty$.

Putting all the estimates above together into \eqref{t2-1-7} and letting $m \to \infty$, we will obtain that
for every $f\in C_c^{1}(\R^d)$ and $\tilde G\in \D$,
\begin{align*}
&\frac{1}{2}\int_{\Omega}\left(\int_{\R^d}\int_{\R^d}(\phi_{k}(x+z;\w)-\phi_k(x;\w))
(f(x+z)-f(x))\nu(z)\,dz\,dx\right)\tilde G(\w)\,\Pp(d\w)\\
&-\sum_{j=1}^d \int_{\Omega}\left(\int_{\R^d}b_{j}(x;\w)\phi_{k}(x;\w)\frac{\partial f(x)}{\partial x_j}\,dx\right)\tilde G(\w)\,\Pp(d\w)=
-\int_{\Omega}\left(\int_{\R^d} b_k(x;\w)f(x)\,dx\right)\tilde G(\w)\,\Pp(d\w),
\end{align*} where we also used the stationary of the transformation $\{\tau_x\}_{x\in \R^d}$ in the right hand side of the equality above.
Note that the equation above holds for every $\tilde G \in \D$.
By choosing a dense subset of test functions in $C_c^1(\R^d)$, we can find a common $\Pp$-null set $\Lambda \subset \Omega$ such that
\eqref{t2-1-1} holds for every $f\in C_c^1(\R^d)$ and $\w\notin \Lambda$. The proof is complete.
\end{proof}

\subsection{Homogenized coefficients}

\begin{lemma}\label{l1-2} Suppose that Assumptions $\ref{a1-2--}$ and $\ref{a1-2}$ hold.
Let $\bar A:=
\{\bar a_{jk}\}_{1\le j,k\le d}$ be the $d\times d$ matrix defined by \eqref{e1-7}. Then,
$\bar A$ is strictly positive definite and bounded.
\end{lemma}
\begin{proof}
For every non-zero $\xi:=\left(\xi_1,\cdots,\xi_d\right)\in \R^d$,
\begin{align*}
\sum_{j,k=1}^d \bar a_{jk}\xi_j\xi_k&=
\frac{1}{2}\Ee\left[\int_{\R^d}\sum_{j,k=1}^d \xi_j\xi_k\left(z_j+\phi_j(z;\w)\right)\left(z_k+\phi_k(z;\w)\right)\nu(z)\,dz\right]\\
&=\frac{1}{2}\Ee\left[\int_{\R^d}\bigg(\sum_{j=1}^d \xi_j(z_j+\phi_j(z;\w))\bigg)^2\nu(z)\,dz\right]\ge 0.
\end{align*} This along with \eqref{t2-1-2} implies that $\bar A$ is positive definite and bounded.

Next, we suppose that $\sum_{j,k=1}^d \bar a_{jk}\xi_j\xi_k=0$ for some non-zero $\xi\in \R^d$. Then, for a.s.\ $\w\in \Omega$ and any $z\in B(0,1)$
\begin{align}\label{l1-2-1}
\sum_{j=1}^d\xi_j\left(z_j+\phi_j(z;\w)\right)=0.
\end{align}
Meanwhile, by Theorem \ref{t2-1}, $\Ee[\phi(x;\w)]=0$ and so for all $z\in B(0,1)$,
\begin{align*}
\Ee\left[\sum_{j=1}^d\xi_j\left(z_j+\phi_j(z;\w)\right)\right]=\langle \xi, z\rangle,
\end{align*}
which contradicts with \eqref{l1-2-1}. Therefore, we conclude that $\sum_{j,k=1}^d \bar a_{jk}\xi_j\xi_k>0$ for every non-zero $\xi\in \R^d$, and so
$\bar A$ is strictly positive definite.
\end{proof}

\section{The scaled resolvent equation and properties of the corrector}\label{section3}
This section consists two parts. In the first subsection, we prove the existence of weak solution to the scaled resolvent equation \eqref{e1-6}, where the proof is partly inspired by that of Theorem \ref{t2-1}. In the second one, we will establish some regularity properties and estimates for the corrector $\phi$ to the equation \eqref{e:equ-01}, which are key ingredients to the proof of Theorem \ref{t1-2}.
\subsection{The existence of solution to the scaled resolvent equation \eqref{e1-6}}

\begin{proposition}\label{t1-1} Suppose that Assumptions $\ref{a1-2--}$ and $\ref{a1-2}$ hold.
Then, for every $h\in \mathscr{G}_\lambda$ with $\lambda>0$ and every $\e\in (0,1)$, there exists a weak solution $u^{\e,\w}:=u^\e(\cdot;\w)\in W^{\alpha/2,2}(\R^d)$ to the equation
\eqref{e1-6} such that the following statements hold for a.e. $\w\in \Omega$:
\begin{itemize}
\item [(i)] for every $f\in C_c^1(\R^d)$,
\begin{equation}\label{t1-1-0}
\begin{split}
& \frac{\e^{-d-2}}{2}\int_{\R^d}\left(u^{\e}(x+z;\w)-u^{\e}(x;\w)\right)(f(x+z)-f(x))\nu\left(\frac{z}{\e}\right)\,dz\\
&+\e^{-1}\sum_{j=1}^d \int_{\R^d}b_{j}\left(\frac{x}{\e};\w\right)\frac{\partial f(x)}{\partial x_j}
u^{\e}(x;\w)\,dx+\lambda\int_{\R^d}u^{\e}(x;\w)f(x)\,dx=\int_{\R^d}h(x)f(x)\,dx,
\end{split}
\end{equation}
\item [(ii)]
\begin{equation}\label{t1-1-1a}
\sup_{\e\in (0,1)}\|u^{\e}(\cdot;\w)\|_{L^\infty(\R^d)}<\infty,\quad \sup_{\e\in (0,1)}\|u^{\e}(\cdot;\w)\|_{L^2(\R^d)}<\infty,
\end{equation} and
\begin{equation}\label{t1-1-0a}
\sup_{\e\in (0,1)}
\e^{-d-2}\int_{\R^d}\int_{\R^d}\left(u^{\e}(x+z;\w)-u^{\e}(x;\w)\right)^2\nu\left(\frac{z}{\e}\right)\,dz\,dx<\infty.
\end{equation}
\end{itemize}
\end{proposition}
\begin{proof}
(1) Let
\begin{equation}\label{t2-1-9a}
\eta(x)=
\begin{cases}
C_0\exp\left(-\frac{1}{1-|x|^2}\right),\ \ &\ |x|<1,\\
0,\ \ &\ |x|\ge1,
\end{cases}
\end{equation}
where $C_0$ is a positive normalizing constant so that $\displaystyle\int_{\R^d}\eta(x)\,dx=1$.
For every $\delta\in (0,1)$, let $\eta_\delta(x):=\delta^{-d}\eta\left(\frac{x}{\delta}\right)$ be the standard smooth mollifier,
and denote $f*\eta_{\delta}(x):=\displaystyle\int_{\R^d}\eta_{\delta}(x-y)f(y)\,dy$ by the smooth convolution
of the function $f$ with respect to $\eta_{\delta}$.

For every $\theta>0$ and $R\ge 1$, define
$$
L^{\e,\w}_{\theta,R}f(x)=L_0^{\e}f(x)+\theta\Delta f(x)+
\e^{-1}\sum_{j=1}^d b_{j}^R*\eta_\theta\left(\frac{x}{\e};\w\right)\frac{\partial f(x)}{\partial x_j},\quad f\in C_c^2(\R^d),
$$
where $b_j^R(x;\w)=\tilde b_j^R(\tau_x\w)$ with $\tilde b_j^R$ being defined by \eqref{t1-1-2}. Then, by \eqref{e1-3}, the integration by parts formula
and the anti-symmetric property of $\{\tilde H_{jl}\}_{1\le j,l\le d}$ (which implies the anti-symmetry of
$\{\rho_R(\tilde H_{jl})\}_{1\le j,l\le d}$ thanks to the definition of $\rho_R$),
for every $f,g\in C_c^2(\R^d)$,
\begin{align*}
\EE^{\e,\w}_{\theta,R}(f,g):&=
-\int_{\R^d}L_{\theta,R}^{\e,\w}f(x)g(x)\,dx\\
&=\frac{\e^{-d-2}}{2}\int_{\R^d}\left(f(x+z)-f(x)\right)\left(g(x+z)-g(x)\right)\nu\left(\frac{z}{\e}\right)\,dz\\
&\quad+\theta\int_{\R^d}\langle \nabla f(x), \nabla g(x)\rangle\, dx
-\sum_{j,l=1}^d \int_{\R^d}\rho_R\left(H_{jl}\left(\cdot;\w\right)\right)*\eta_\theta\left(\frac{x}{\e}\right)\frac{\partial f(x)}{\partial x_j}\frac{\partial g(x)}{\partial x_l}\,dx,
\end{align*}
where $\rho_R:\R \to \R$ is the cut-off function used in \eqref{t1-1-2}.
This implies that for all $f,g\in C_c^2(\R^d)$,
\begin{equation}\label{t1-1-6}
\begin{split}
|\EE^{\e,\w}_{\theta,R}(g,h)|&\le c_1(\e,\theta,R)\left(\|f\|_{W^{\alpha/2,2}(\R^d)}+\|f\|_{W^{1,2}(\R^d)}\right)
\left(\|g\|_{W^{\alpha/2,2}(\R^d)}+\|g\|_{W^{1,2}(\R^d)}\right)\\
&\le c_2(\e,\theta,R)\|f\|_{W^{1,2}(\R^d)}\|g\|_{W^{1,2}(\R^d)},
\end{split}
\end{equation}
where $c_2(\e,\theta,R)$ is a positive constant that may depend on $\e$, $\theta$ and $R$.
In particular, we can extend
$\EE^{\e,\w}_{\theta,R}$ to a continuous bilinear form on $W^{1,2}(\R^d)\times W^{1,2}(\R^d)$.
On the other hand, by the anti-symmetry of
$\{\rho_R(\tilde H_{jl})\}_{1\le j,l\le d}$ again, we can show that for all $f\in W^{1,2}(\R^d)$,
\begin{equation*}
\EE_{\theta,R}^{\e,\w}(f,f)\ge c_3(\varepsilon,\theta,R)\int_{\R^d}|\nabla f(x)|^2\,dx.
\end{equation*}
Therefore, by the Lax-Milgram theorem (e.g., see \cite[Section 2]{PV}), for any $\lambda>0$ and
$h\in \mathscr{G}_\lambda$,
there exists a unique $u^{\e}_{\theta,R}(\cdot;\w)\in W^{1,2}(\R^d)$ such that for all $ f\in W^{1,2}(\R^d)$,
\begin{equation}\label{t1-1-4a}
\begin{split}
&\frac{\e^{-d-2}}{2}\int_{\R^d}\int_{\R^d}(u_{\theta,R}^{\e}(x+z;\w)-u_{\theta,R}^{\e}(x;\w))(f(x+z)-f(x))\nu\left(\frac{z}{\e}\right)\,dz\,dx\\
&+\theta\int_{\R^d}
\langle \nabla u_{\theta,R}^{\e}(x;\w), \nabla f(x)\rangle \,dx+\e^{-1}\sum_{j=1}^d \int_{\R^d}b_{j}^R*\eta_\theta\left(\frac{x}{\e};\w\right)\frac{\partial f(x)}{\partial x_j}
u_{\theta,R}^{\e}(x;\w)\,dx\\
&
+\lambda\int_{\R^d}u_{\theta,R}^{\e}(x;\w)f(x)\,dx=\int_{\R^d}h(x)f(x)\,dx.
\end{split}
\end{equation}
Taking $f=u_{\theta,R}^{\e}(\cdot;\w)$ in the equation \eqref{t1-1-4a}, and using \eqref{e:formula1} as well as its proof due to  the anti-symmetry of $\{\rho_R(H_{jl})\}_{1\le j,l\le d}$, we have
\begin{align*}
&\frac{\e^{-d-2}}{2}\int_{\R^d}\int_{\R^d}\left(u_{\theta,R}^{\e}(x+z;\w)-u_{\theta,R}^{\e}(x;\w)\right)^2\nu\left(\frac{z}{\e}\right)\,dz\,dx+\theta\int_{\R^d}
\left|\nabla u_{\theta,R}^{\e}(x;\w)\right|^2\,dx+\lambda\int_{\R^d}|u_{\theta,R}^{\e}(x;\w)|^2\,dx\\
&=\int_{\R^d}h(x)u_{\theta,R}^{\e}(x;\w)\,dx
\le \frac{\lambda}{2}\int_{\R^d}|u_{\theta,R}^{\e}(x;\w)|^2\,dx+\frac{1}{2\lambda}\int_{\R^d}|h(x)|^2\,dx.
\end{align*}
In particular,
\begin{align}\label{t1-1-6a}
\sup_{\e,\theta\in (0,1),R\ge 1}
\e^{-d-2}\int_{\R^d}\int_{\R^d}\left(u_{\theta,R}^{\e}(x+z;\w)-u_{\theta,R}^{\e}(x;\w)\right)^2\nu\left(\frac{z}{\e}\right)\,dz\,dx
\le c_4(\lambda)\int_{\R^d}|h(x)|^2 \,dx,
\end{align}
\begin{align}\label{t1-1-8}
\sup_{\e,\theta\in (0,1),R\ge 1}\int_{\R^d}|u_{\theta,R}^{\e}(x;\w)|^2\,dx\le \lambda^{-2}\int_{\R^d}|h(x)|^2\,dx
\end{align}
and
\begin{align}\label{t1-1-5}
\sup_{\e,\theta\in (0,1),R\ge 1}\theta\int_{\R^d}
\left|\nabla u_{\theta,R}^{\e}(x;\w)\right|^2\,dx\le c_4(\lambda)\int_{\R^d}|h(x)|^2\,dx.
\end{align}

Since $b^R(\cdot;\w)*\eta_\theta$ is a smooth bounded vector field in $\R^d$, there exists a
unique Markov semigroup $\{T_t^{\theta,R,\e,\w}\}_{t\ge 0}$ associated with the operator $L_{\theta,R}^{\e,\w}$ so that
\begin{align*}
u_{\theta,R}^{\e}(x;\w)=\int_0^\infty e^{-\lambda t}T_t^{\theta,R,\e,\w}h(x)\,dt,
\end{align*}
which implies that
\begin{equation}\label{t1-1-3}
\sup_{\e,\theta\in (0,1),R\ge 1}\|u_{\theta,R}^{\e}(\cdot;\w)\|_{L^\infty(\R^d)}
\le \int_0^\infty e^{-\lambda t}\|h\|_{L^\infty(\R^d)}\,dt\le \lambda^{-1}\|h\|_{L^\infty(\R^d)}.
\end{equation}
Hence, according to \eqref{t1-1-6a} and \eqref{t1-1-8} as well as the definition of $\nu(z)$, we know that for every fixed $\e\in (0,1)$, $\{u_{\theta,R}^{\e}(\cdot;\w)\}_{\theta\in (0,1),R\ge 1}$ is weakly compact in
$W^{\alpha/2,2}(\R^d)$, so it is compact in $L_{loc}^2(\R^d)$, thanks to the compact embedding of the space $W_{loc}^{\alpha/2,2}(\R^d)$; see \cite[Theorem 7.1]{NPV}.
 Then, there exist a subsequence $\{u_{\theta_m,R_m}^{\e}(\cdot;\w)\}_{m\ge 1}$ (with $\theta_m\to 0$ and $R_m\to\infty$ as $m\to \infty$) and $u^{\e}(\cdot;\w)$
in $W^{\alpha/2,2}(\R^d)$ such that
$u_{\theta_m,R_m}^{\e}(\cdot;\w)$ converges  to $u^{\e}(\cdot;\w)$ weakly in $W^{\alpha/2,2}(\R^d)$ and strongly in $L^2_{loc}(\R^d)$ as $m \to \infty$. That is,
\begin{equation}\label{t1-1-5a}
\lim_{m \to \infty}\int_{D}\left|u_{\theta_m,R_m}^{\e}(x;\w)-u^{\e}(x;\w)\right|^2\,dx=0\quad {\rm for\ any\ bounded\ subset}\ D\subset \R^d,
\end{equation}
and for every $f\in C_c^\infty(\R^d)$,
\begin{equation*}
\lim_{m \to \infty}\int_{\R^d}\int_{\R^d}\frac{((u_{\theta_m,R_m}^{\e}(x+z;\w)-u^{\e}(x+z;\w))-
(u_{\theta_m,R_m}^{\e}(x;\w)-u^{\e}(x;\w)))(f(x+z)-f(x))}{|z|^{d+\alpha}}\,dz\,dx=0.
\end{equation*}
These two estimates, along with the definition of $\nu(z)$ again, \eqref{t1-1-6a} and \eqref{t1-1-3}, in turn yield that for every $f\in C_c^\infty(\R^d)$,
\begin{equation}\label{t1-1-7}
\lim_{m \to \infty}\int_{\R^d}\int_{\R^d}((u_{\theta_m,R_m}^{\e}(x+z;\w)-u^{\e}(x+z;\w))-
(u_{\theta_m,R_m}^{\e}(x;\w)-u^{\e}(x;\w)))(f(x+z)-f(x))\nu\left(\frac{z}{\e}\right)\,dz\,dx=0.
\end{equation}

Furthermore, by \eqref{t1-1-5},  for all $f\in C_c^1(\R^d)$,
\begin{equation}\label{t1-1-7a}
\begin{split}
&\lim_{m \to \infty}\theta_m \left|\int_{\R^d}\left\langle \nabla u_{\theta_m,R_m}^{\e}(x;\w), \nabla f(x)\right\rangle\, dx\right|\\
&\le \lim_{m \to \infty}\theta_m\left(\int_{\R^d}\left|\nabla  u_{\theta_m,R_m}^{\e}(x;\w)\right|^2 \,dx\right)^{1/2}\cdot
\left(\int_{\R^d}|\nabla f(x)|^2 \,dx\right)^{1/2}\\
&\le c_5\lim_{m \to \infty}\theta_m^{1/2}=0.
\end{split}
\end{equation}

For every $f\in C_c^1(\R^d)$, set the support of $f$ by $V(f):={\rm supp}[f]$.
By the definition of $b_j^R(x;\w)$ and the fact that  $|\rho_R'(s)|\le 2$ for all $s\in \R$,
we know that for every fixed $\e\in (0,1)$ and $\w \in \Omega$,
\begin{align*}
\sup_{R\ge 1, x\in V(f)}\left|b_j^R\left(\frac{x}{\e};\w\right)\right|\le c_6\sup_{x\in \frac{V(f)}{\e}}\sup_{1\le j,l\le d}\left|\frac{H_{jl}(x;\w)}{\partial x_l}\right|=:c_7(\e,f,\w)<\infty,
\end{align*} where we used the fact that $x\mapsto D_l\tilde H_{jl}(\tau_x \w)$ is continuous.
Combining  this with \eqref{t1-1-3} and \eqref{t1-1-5a}, and applying the dominated convergence theorem, we get
\begin{align*}
\lim_{m \to \infty}\int_{\R^d}b_j^{R_m}*\eta_{\theta_m}\left(\frac{x}{\e};\w\right)\frac{\partial f(x)}{\partial x_j}u_{\theta_m,R_m}^{\e}(x;\w)\,dx&=
\lim_{m \to \infty}\int_{\R^d}b_j^{R_m}\left(\frac{x}{\e};\w\right)\frac{\partial f(x)}{\partial x_j}u^{\e}(x;\w)\,dx\\
&=\int_{\R^d}b_j\left(\frac{x}{\e};\w\right)\frac{\partial f(x)}{\partial x_j}u^{\e}(x;\w)\,dx,\quad f\in C_c^1(\R^d).
\end{align*}

Therefore, putting this, \eqref{t1-1-7}, \eqref{t1-1-7a} and \eqref{t1-1-5a} together, we can prove \eqref{t1-1-0} by letting $m \to \infty$ in \eqref{t1-1-4a}.
Furthermore, by \eqref{t1-1-6a}, \eqref{t1-1-8} and \eqref{t1-1-3}, we can get the desired conclusions \eqref{t1-1-1a} and \eqref{t1-1-0a}.
\end{proof}

\subsection{Properties of the corrector $\phi$} In this part, let $\phi:\R^d \times \Omega \to \R^d$ be the corrector  to the equation \eqref{e:equ-01} that was constructed in Theorem $\ref{t2-1}$.
We first consider the $W^{1,q}_{loc}$-regularity property (with $q\in (1,2)$) of the corrector $\phi$.

\begin{lemma}\label{l1-5} Suppose that Assumptions $\ref{a1-2--}$ and $\ref{a1-2}$ hold, and that
\eqref{t2-1-3a} holds with some $q\in (1,2)$ and
$d> 4(\alpha-1)$.
Then there are  
a $\Pp$-null set $\Lambda$ and $\tilde \Phi\in L^q(\Omega,\R^d\times \R^d,\Pp)$, so that $\phi(\cdot;\w)\in W_{loc}^{1,q}(\R^d,\R^d)$ for every $\w\notin \Lambda$, and  
$\nabla \phi(x;\w)=\tilde \Phi(\tau_x \w)$ for a.e. $x\in \R^d$ and  every $\w\notin \Lambda$.
\end{lemma}
\begin{proof} The proof is split into two parts.

(1) Let $\phi_{\theta_m,R_m}:\R^d \times \Omega \to \R^d$ be defined by \eqref{t2-1-5a} in the proof of Theorem \ref{t2-1},
which converges weakly to $\phi$ in $L^2(\Omega,W_{loc}^{\alpha/2,2}(\R^d,\R^d);\Pp)$ and $L^2(\R^d\times \Omega, \R^d;\nu(z)\,dz\times \Pp)$ as $m\to\infty$.
According to \eqref{t2-1-4}, \eqref{t2-1-7} and the proof of \eqref{t2-1-1}, we can verify that there
exists a $\Pp$-null set $\Lambda_0$, so that for every $f\in C_c^1(\R^d)$, $\w\notin \Lambda_0$ and $1\le k \le d$,
\begin{equation}\label{t2-1-9--}
\begin{split}
&\theta_m \int_{\R^d}\phi_{k,\theta_m,R_m}(x;\w)f(x)\,dx+\theta_m \tilde \phi_{k,\theta_m,R_m}(\w)\int_{\R^d}f(x)\,dx
+\theta_m \int_{\R^d}\langle \nabla \phi_{k,\theta_m,R_m}(x;\w), \nabla f(x)\rangle \,dx\\
&+
\frac{1}{2}\int_{\R^d}\int_{\R^d}\left(\phi_{k,\theta_m,R_m}(x+z;\w)-\phi_{k,\theta_m,R_m}(x;\w)\right)
\left(f(x+z)-f(x)\right)\nu(z)\,dz\,dx\\
&+\int_{\R^d} \langle b^{R_m}(x;\w), \nabla \phi_{k,\theta_m,R_m}(x;\w)\rangle f(x)\,dx=
-\int_{\R^d}b_k(x;\w)f(x)\,dx,
\end{split}
\end{equation}
where we used the fact that $\phi_{k,\theta_m,R_m}(\cdot;\w)\in W^{1,2}(\R^d)$ for a.s. $\w\in \Omega$,
due to \eqref{t2-1-5}.

Let $r_0\ge 1$ be the constant in \eqref{t2-1-3a}. Let $\psi_{r_0}:\R^d \to [0,1]$ be a cut-off function such that $\psi_{r_0}\in C_c^\infty(\R^d)$,
${\rm supp}[\psi_{r_0}] \subset B(0,r_0)$, $\sup_{x\in \R^d}|\nabla^k \psi_{r_0}(x)|\le c_1(k,r_0)$
for $k=1,2$, and $\psi_{r_0}(x)=1$ for every
$x\in B(0,r_0/2)$. Define $\phi_{\theta_m,R_m}^{r_0}(x;\w)=\phi_{\theta_m,R_m}(x;\w)\psi_{r_0}(x)$. Taking $f=f\psi_{r_0}$ in the equation \eqref{t2-1-9--}, we find that for every $f\in C_c^1(\R^d)$,
\begin{equation}\label{t2-1-9}
\begin{split}
&\theta_m \int_{\R^d}\phi_{k,\theta_m,R_m}^{r_0}(x;\w)f(x)\,dx
+\theta_m \int_{\R^d}\langle \nabla \phi_{k,\theta_m,R_m}^{r_0}(x;\w), \nabla f(x)\rangle\, dx\\
&+
\frac{1}{2}\int_{\R^d}\int_{\R^d}(\phi_{k,\theta_m,R_m}^{r_0}(x+z;\w)-\phi_{k,\theta_m,R_m}^{r_0}(x;\w))
\left(f(x+z)-f(x)\right)\nu(z)\,dz\,dx\\
&=
\int_{\R^d}F_{k,m}(x;\w)f(x)\,dx.
\end{split}
\end{equation}
Here,
\begin{align*}
F_{k,m}(x;\w)&:=\sum_{j=0}^3 F_{k,m,j}(x;\w),
\end{align*}
where
\begin{align*}
F_{k,m,0}(x;\w)&=-\langle b^{R_m}(x;\w), \nabla \phi_{k,\theta_m,R_m}^{r_0}(x;\w)\rangle,\\
F_{k,m,1}(x;\w)&=-b_k(x;\w)\psi_{r_0}(x)-\theta_m\tilde \phi_{k,\theta_m,R_m}(\w)\psi_{r_0}(x),\\
F_{k,m,2}(x;\w)&=-
\int_{\R^d}\left(\phi_{k,\theta_m,R_m}(x+z;\w)-\phi_{k,\theta_m,R_m}(x;\w)\right)\left(\psi_{r_0}(x+z)-\psi_{r_0}(x)\right)\nu(z)\,dz,\\
F_{k,m,3}(x;\w)&=-\phi_{k,\theta_m,R_m}(x;\w)
\left(L_0\psi_{r_0}(x)+\theta_m\Delta \psi_{r_0}(x)-\langle b^{R_m}(x;\w), \nabla \psi_{r_0}(x)\rangle\right) \\
&\quad-2\theta_m\langle \nabla \phi_{k,\theta_m,R_m}(x;\w), \nabla \psi_{r_0}(x)\rangle.
\end{align*}

Define $\phi_{\theta_m, R_m}^{r_0,\delta}(x;\w)=\phi_{\theta_m,R_m}^{r_0}*\eta_{\delta}(x;\w)$, where $\eta_{\delta}:\R^d \to [0,1]$ is the smooth mollifier used
in the proof of Proposition \ref{t1-1}.
Changing $x$ into $z$, taking $f(z)=\eta_\delta(x-z)$ and integrating with respect to the variable $z$ in \eqref{t2-1-9}, we can derive that
\begin{align*}
\theta_m \phi_{k,\theta_m, R_m}^{r_0,\delta}(x;\w)-\theta_m \Delta \phi_{k,\theta_m, R_m}^{r_0,\delta}(x;\w)-
L_0 \phi_{k,\theta_m, R_m}^{r_0,\delta}(x;\w)=F_{k,m}*\eta_\delta(x;\w),
\end{align*}
where
\begin{align*}
F_{k,m}*\eta_\delta(x;\w):=\int_{\R^d}F_{k,m}(z;\w)\eta_\delta(x-z)\,dz,\quad x\in \R^d.
\end{align*}
This implies
\begin{align}\label{t2-1-10}
\theta_m \phi_{k,\theta_m, R_m}^{r_0,\delta}(x;\w)-\theta_m \Delta \phi_{k,\theta_m, R_m}^{r_0,\delta}(x;\w)-
\Delta^{\alpha/2}\phi_{k,\theta_m, R_m}^{r_0,\delta}(x;\w)=G_{k,m}^{\delta}(x;\w),
\end{align}
where $\Delta^{\alpha/2}$ denotes the fractional Laplacian operator corresponding  to the L\'evy measure $|z|^{-d-\alpha}\,dz$, and
\begin{align*}
G_{k,m}^{\delta}(x;\w):=&F_{k,m}*\eta_\delta(x;\w)+\int_{\{|z|>1\}}
(\phi_{k,\theta_m, R_m}^{r_0,\delta}(x+z;\w)-\phi_{k,\theta_m, R_m}^{r_0,\delta}(x;\w))\nu(z)\,dz\\
&-\int_{\{|z|>1\}}
(\phi_{k,\theta_m, R_m}^{r_0,\delta}(x+z;\w)-\phi_{k,\theta_m, R_m}^{r_0,\delta}(x;\w))\frac{1}{|z|^{d+\alpha}}\,dz.
\end{align*}

According to \eqref{t2-1-10}, we have
\begin{align*}
\phi_{k,\theta_m,R_m}^{r_0,\delta}(x;\w)=\int_0^\infty e^{-\theta_m t}T_t^m  G_{k,m}^{\delta} (x;\w)\, dt,
\end{align*}
where $\{T_t^m\}_{t\ge 0}:L^2(\R^d;dx)\to L^2(\R^d;dx)$ is the Markov semigroup associated with the infinitesimal operator
$\theta_m\Delta+\Delta^{\alpha/2}$. Letting $\delta \to 0$ yields that
\begin{align}\label{t2-1-10a}
\phi_{k,\theta_m,R_m}^{r_0}(x;\w)=\int_0^\infty e^{-\theta_m t}T_t^m G_{k,m}(x;\w) \,dt,
\end{align}
where
\begin{align*}
G_{k,m}(x;\w):=&F_{k,m}(x;\w)+\int_{\{|z|>1\}}
 (\phi_{k,\theta_m, R_m}^{r_0}(x+z;\w)-\phi_{k,\theta_m, R_m}^{r_0}(x;\w))\nu(z)\,dz\\
&-\int_{\{|z|>1\}}
 (\phi_{k,\theta_m, R_m}^{r_0}(x+z;\w)-\phi_{k,\theta_m, R_m}^{r_0}(x;\w))\frac{1}{|z|^{d+\alpha}}\,dz
\end{align*}

Let $p_{t,\alpha}(x)$ and $q_t^m(x)$ be the fundamental solutions to the operators $\Delta^{\alpha/2}$ and $\theta_m \Delta$ respectively. Thus, it holds that
\begin{align*}
T_t^m f(x)=\int_{\R^d}p_t^m(x-y)f(y)\,dy,\quad  f\in L^2(\R^d;dx),
\end{align*}
where
\begin{align*}
p_t^m(x)=\int_{\R^d}p_{t,\alpha}(x-y)q_t^m(y)\,dy,\quad t>0,\ x\in \R^d.
\end{align*}
It is well known that (see e.g. \cite[Lemma 2.2]{CZ})
$$|\nabla^l p_{t,\alpha}(x)|\le c_2t^{-l/\alpha}p_{t,\alpha}(x)\le c_3t^{-({d+l})/{\alpha}},\quad t>0,\ x\in \R^d,\ l=0,1,2.$$
This yields that
\begin{align*}
|\nabla^l p_t^m(x)|\le c_2t^{-l/\alpha}p_{t}^m(x)\le c_3t^{-({d+l})/{\alpha}},\quad t>0,\ x\in \R^d,\ l=0,1,2,
\end{align*}
and here we emphasize that the positive constants $c_2$ and $c_3$ are independent of $m$.
Hence, for every $f\in L^2(\R^d;dx)$, $t>0$, $x\in \R^d$ and $m\ge 1$,
\begin{align}\label{t2-1-11}
|\nabla T_t^m f(x)|\le c_2t^{-1/\alpha}T_t^m |f|(x),
\end{align}
\begin{equation}\label{t2-1-11a}
\begin{split}
|\nabla T_t^m f(x)|&\le c_2t^{-1/\alpha}\left(\int_{\R^d}|p^m_t(x-z)|^2\,dz\right)^{1/2}
\left(\int_{\R^d}|f(z)|^2\,dz\right)^{1/2}\\
&\le c_2t^{-1/\alpha}p^m_{2t}(x)^{1/2}\left(\int_{\R^d}|f(z)|^2\,dz\right)^{1/2}\le c_4t^{-(d+2)/({2\alpha})}\left(\int_{\R^d}|f(z)|^2\,dz\right)^{1/2}
\end{split}
\end{equation}and
\begin{equation}\label{t2-1-12}
\begin{split}
\left|\nabla T_t^m \left(\frac{\partial f}{\partial x_j}\right)(x)\right|&=\left|\int_{\R^d}\frac{\partial \nabla p_t^m(x-z)}{\partial x_j}f(z)\,dz\right|
\le c_2t^{-2/\alpha}T_t^m |f|(x).
\end{split}
\end{equation}

(2) Recall that $\phi_{k,\theta_m,R_m}^{r_0}(x;\w)=\phi_{k,\theta_m,R_m}(x;\w)\psi_{r_0}(x)$, so $\phi_{k,\theta_m,R_m}^{r_0}(\cdot;\w)\in L^2(\R^d;dx)$.
Recall also that $N_{r_0}(\w):=\sup_{x\in B(0,r_0), 1\le j,l \le d}|D_j\tilde H_{jl}(\tau_x \w)|.$
This implies immediately that
\begin{equation}\label{t2-1-12a}
\begin{split}
\sup_{R\ge 1}\sup_{x\in B(0,r_0)}|b^R(x;\w)|&\le \|\rho_R'\|_\infty\sum_{j,l=1}^d  \left(\sup_{x\in B(0,r_0)}\left|D_j \tilde H_{jl}(\tau_x \w)\right|\right)\le c_5N_{r_0}(\w).
\end{split}
\end{equation}
According to the definitions of $F_{k,m,j}(\w)$, $j=0,1,2$, we have
\begin{align}\label{t2-1-7a}
|F_{k,m,0}(x;\w)|\le c_6N_{r_0}(\w)|\nabla \phi_{k,\theta_m,R_m}^{r_0}(x;\w)|,
\end{align}
\begin{align*}
|F_{k,m,1}(x;\w)|\le c_6(N_{r_0}(\w)+\theta_m|\tilde \phi_{k,\theta_m,R_m}(\w)|)\I_{B(0,r_0)}(x),
\end{align*}
and
\begin{align*}
&|F_{k,m,2}(x;\w)|\\
&\le c_{7}\|\nabla \psi_{r_0}\|_{L^\infty(\R^d)}\left(
\int_{\R^d}
\left|\phi_{k,\theta_m,R_m}(x+z;\w)-\phi_{k,\theta_m,R_m}(x;\w)\right|^2\nu(z)\,dz\right)^{1/2}\I_{B(0,K_0r_0)}(x)\\
&\quad+c_7\left(\int_{B(x,r_0)}\nu^2(z)\,dz\right)^{1/2}\left[r_0^{d/2}|\phi_{k,\theta_m,R_m}(x;\w)|+\left(\int_{B(0,r_0)}
|\phi_{k,\theta_m,R_m}(z;\w)|^2\,dz\right)^{1/2}\right]\I_{B(0,K_0r_0)^c}(x)\\
&\le c_8r_0^{-1}\left(
\int_{\R^d}
\left|\phi_{k,\theta_m,R_m}(x+z;\w)-\phi_{k,\theta_m,R_m}(x;\w)\right|^2\nu(z)\,dz\right)^{1/2}\I_{B(0,K_0r_0)}(x)\\
&\quad+c_{8}r_0^{d}\gamma(x)\nu(x)\left[|\phi_{k,\theta_m,R_m}(x;\w)|+ r_0^{-d/2}
\left(\int_{B(0,r_0)}|\phi_{k,\theta_m,R_m}(z;\w)|^2\,dz\right)^{1/2}\right]\I_{B(0,K_0r_0)^c}(x),
\end{align*}
where
$K_0\ge2$ and $\gamma(x)$ are given in Assumption \ref{a1-2--}(iii), and we used  \eqref{t2-1-2a} in the last inequality.

Similarly, by \eqref{t2-1-2a}, we also have
\begin{align*}
|L_0\psi(r)|&\le c_9\left(\|\nabla^2 \psi_{r_0}\|_{L^\infty(\R^d)}\I_{B(0,K_0r_0)}(x)+r_0^{d/2}\left(\int_{B(x,r_0)}\nu^2(z)\,dz\right)^{1/2}\I_{B(0,K_0r_0)^c}(x)\right)\\
&\le c_{10}\left(r_0^{-2}\I_{B(0,K_0r_0)}(x)+r_0^d\gamma(x)\nu(x)\I_{B(0,K_0r_0)^c}(x)\right),
\end{align*}
which along with \eqref{t2-1-12a} implies that
\begin{align*}
|F_{k,m,3}(x;\w)|\le &c_{11}\left(\left(N_{r_0}(\w)r_0^{-1}+r_0^{-2}\right)\I_{B(0,K_0r_0)}(x)+r_0^d\gamma(x)\nu(x)\I_{B(0,K_0r_0)^c}(x)\right)\left|\phi_{k,\theta_m,R_m}(x;\w)\right|\\
&+ c_{11}r_0^{-1}\theta_m|\nabla \phi_{k,\theta_m,R_m}(x;\w)| \I_{B(0,K_0r_0)}(x).
\end{align*}

Furthermore, by the fact
${\rm supp}[\phi_{\theta_m,R_m}^{r_0}] \subset B(0,r_0)$, we get
\begin{align*}
&\left|\int_{\{|z|>1\}}(\phi_{k,\theta_m,R_m}^{r_0}(x+z;\w)-\phi_{k,\theta_m,R_m}^{r_0}(x;\w))\,\nu(z)\,dz\right|\\
&\le\int_{\{|z|>1\}}|\phi_{k,\theta_m,R_m}^{r_0}(x;\w)|\nu(z)\,dz+\int_{\{|z|>1:x+z\in B(0,r_0)\}}|\phi_{k,\theta_m,R_m}^{r_0}(x+z;\w)|\nu(z)\,dz \\
&\le c_{12}|\phi_{k,\theta_m,R_m}^{r_0}(x;\w)|+c_{12}\left(\int_{B(x,r_0)\cap B(0,1)^c}\nu^2(z)\,dz\right)^{1/2}
\left(\int_{B(0,r_0)}|\phi_{k,\theta_m,R_m}^{r_0}(z;\w)|^2\,dz\right)^{1/2}\\
&\le c_{13}|\phi_{k,\theta_m,R_m}^{r_0}(x;\w)|+c_{13}r_0^{d/2}\Big(\I_{B(0,K_0r_0)}(x)+\gamma(x)\nu(x)\I_{B(0,K_0r_0)^c}(x)\Big)
\left(\int_{B(0,r_0)}|\phi_{k,\theta_m,R_m}^{r_0}(z;\w)|^2\,dz\right)^{1/2}
\end{align*} and
\begin{align*}
&\left|\int_{\{|z|>1\}}(\phi_{k,\theta_m,R_m}^{r_0}(x+z;\w)-\phi_{k,\theta_m,R_m}^{r_0}(x;\w))\frac{1}{|z|^{d+\alpha}}\,dz\right|\\
&\le c_{14}|\phi_{k,\theta_m,R_m}^{r_0}(x;\w)|+c_{14}r_0^{d/2}(1+|x|)^{-d-\alpha}
\left(\int_{B(0,r_0)}|\phi_{k,\theta_m,R_m}^{r_0}(z;\w)|^2\,dz\right)^{1/2}.
\end{align*}

According to all the estimates above, we obtain that
\begin{equation*}
\begin{split}
&\left|G_{k,m}(x;\w)-F_{k,m,0}(x;\w)\right|\\
&\le c_{15}
\Bigg(\left(1+N_{r_0}(\w)\right)|\phi_{k,\theta_m,R_m}(x;\w)|\\
&\qquad\qquad +\left(
\int_{\R^d}
\left|\phi_{k,\theta_m,R_m}(x+z;\w)-\phi_{k,\theta_m,R_m}(x;\w)\right|^2\nu(z)\,dz\right)^{1/2}\bigg)\I_{B(0,K_0r_0)}(x)\\
&\quad+c_{15}\left((1+N_{r_0}(\w))+\theta_m|\tilde \phi_{k,\theta_m,R_m}(\w)|+\theta_m|\nabla \phi_{k,\theta_m,R_m}(x;\w)|\right)\I_{B(0,K_0r_0)}(x)\\
&\quad +c_{15}\gamma(x)\nu(x)\left(|\phi_{k,\theta_m,R_m}(x;\w)|+
\left(\int_{B(0,r_0)}|\phi_{k,\theta_m,R_m}(z;\w)|^2\,dz\right)^{1/2}\right)\I_{B(0,K_0r_0)^c}(x)\\
&\quad+c_{15}\left(\int_{B(0,r_0)}|\phi_{k,\theta_m,R_m}(z;\w)|^2\,dz\right)^{1/2} (1+|x|)^{-d-\alpha},
\end{split}
\end{equation*}
where $c_{15}>0$  may depend on $r_0$. In particular,
\begin{equation}\label{t2-1-13a}
\begin{split}
&\int_{\R^d}\left|G_{k,m}(x;\w)-F_{k,m,0}(x;\w)\right| \,dx \le c_{16}\sqrt{K_{k,m}(\w)},\\
&\int_{\R^d}\left|G_{k,m}(x;\w)-F_{k,m,0}(x;\w)\right|^2\,dx\le c_{16}K_{k,m}(\w),
\end{split}
\end{equation}
where
\begin{align*}
K_{k,m}(\w):=&\theta_m^2 \int_{B(0,K_0r_0)}\left|\nabla \phi_{k,\theta_m,R_m}(x;\w)\right|^2\,dx
+\left(|N_{r_0}(\w)|^2+1\right)\left(\int_{B(0,K_0r_0)}\left|\phi_{k,\theta_m,R_m}(x;\w)\right|^2\,dx+1\right)\\
&+\int_{\R^d}|\phi_{k,\theta_m,R_m}(x;\w)|^2\nu(x)\,dx +\theta_m^2|\tilde \phi_{k,\theta_m,R_m}(\w)|^2\\
&+\int_{B(0,K_0r_0)}\left(\int_{\R^d}\left|\phi_{k,\theta_m,R_m}(x+z;\w)-\phi_{k,\theta_m,R_m}(x;\w)\right|^2\nu(z)\,dz\right)\,dx.
\end{align*}
Here, we used  \eqref{a1-2-1},
which (along with Cauchy-Schwartz inequality) implies that
\begin{align*}
\int_{\{|x|\ge K_0r_0\}}\left|\phi_{k,\theta_m,R_m}(x;\w)\right|\gamma(x)\nu(x)\,dx
&\le \left(\int_{\R^d}\left|\phi_{k,\theta_m,R_m}(x;\w)\right|^2\nu(x)\,dx\right)^{1/2}
\cdot \left(\int_{\{|x|\ge K_0r_0\}}\gamma^2(x)\nu(x)\,dx\right)^{1/2}\\
&\le c_{16}\left(\int_{\R^d}\left|\phi_{k,\theta_m,R_m}(x;\w)\right|^2\nu(x)\,dx\right)^{1/2}.
\end{align*}

Therefore, by \eqref{t2-1-11}, \eqref{t2-1-7a} and \eqref{t2-1-13a}, for any $t>0$,
\begin{align}\label{t2-1-13}
\int_{\R^d}|\nabla T^m_t G_{k,m}(x;\w)|^2\,dx
\le c_{17}t^{-2/\alpha}\left(|N_{r_0}(\w)|^2\int_{\R^d}\left|\nabla \phi_{k,\theta_m,R_m}^{r_0}(x;\w)\right|^2\,dx
+K_{k,m}(\w)\right).
\end{align}

According to \eqref{t2-1-12}, \eqref{t2-1-12a} and the fact ${\rm div}b_j^{R_m}(\cdot;\w)=0$,
for every $t>0$ and $x\in \R^d$,
\begin{align*}
\left|\nabla T_t^m\left(F_{k,m,0}(\cdot;\w)\right)(x)\right|&=
\left|\sum_{j=1}^d\nabla T_t^m\left(\frac{\partial (b^{R_m}_j(\cdot;\w)\phi_{k,\theta_m,R_m}^{r_0}(\cdot;\w))}{\partial x_j}\right)(x)\right|\\
&\le c_{18}t^{-2/\alpha}N_{r_0}(\w)T_t^m|\phi_{k,\theta_m,R_m}^{r_0}(\cdot;\w)|(x),
\end{align*}
which implies immediately that
\begin{equation}\label{t2-1-14a}
\int_{\R^d}\left|\nabla T_t^m\left(F_{k,m,0}(\cdot;\w)\right)(x)\right|^2dx
\le c_{19}t^{-4/\alpha}N_{r_0}^2(\w)\int_{B(0,r_0)}|\phi_{k,\theta_m,R_m}(x;\w)|^2dx.
\end{equation}
On the other hand, by \eqref{t2-1-11}, \eqref{t2-1-11a} and \eqref{t2-1-13a}, we derive that
for every $t>0$ and $x\in \R^d$,
\begin{align*}
\left|\nabla T^m_t \left(G_{k,m}(\cdot;\w)-F_{k,m,0}(\cdot;\w)\right)(x)\right|^2&\le c_{20}t^{-\frac{d+4}{2\alpha}}\sqrt{K_{k,m}(\w)}
T_t^m\left|G_{k,m}(\cdot;\w)-F_{k,m,0}(\cdot;\w)\right|(x),
\end{align*}
which together with \eqref{t2-1-13a} again gives us that
\begin{equation}\label{t2-1-14}
\int_{\R^d}\left|\nabla T^m_t \left(G_{k,m}(\cdot;\w)-F_{k,m,0}(\cdot;\w)\right)(x)\right|^2\,dx\le c_{20}t^{-\frac{d+4}{2\alpha}}K_{k,m}(\w).
\end{equation}
Putting \eqref{t2-1-10a}, \eqref{t2-1-13}, \eqref{t2-1-14} and \eqref{t2-1-14a} together, we know that for every
$t_0>0$,
\begin{align*}
\|\nabla \phi^{r_0}_{k,\theta_m,R_m}(\cdot;\w)\|_{L^2(\R^d)}
&\le
\int_0^{t_0}e^{-\theta_m t} \|\nabla T_t G_{k,m}(\cdot;\w)\|_{L^2(\R^d)}\,dt+
\int_{t_0}^{\infty}e^{-\theta_m t} \|\nabla T_t G_{k,m}(\cdot;\w)\|_{L^2(\R^d)}\,dt\\
&\le c_{21}\left(N_{r_0}(\w)\|\nabla \phi^{r_0}_{k,\theta_m,R_m}(\cdot;\w)\|_{L^2(\R^d)}+
\sqrt{K_{k,m}(\w)}\right)\left(\int_0^{t_0}t^{-1/\alpha}\,dt\right)\\
&\quad+c_{21}N_{r_0}(\w)
\left(\int_{B(0,r_0)}|\phi_{k,\theta_m,R_m}(x;\w)|^2dx\right)^{1/2}\left(\int_{t_0}^\infty t^{-2/\alpha}\,dt\right)\\
&\quad
+c_{21}\sqrt{K_{k,m}(\w)}\left(\int_{t_0}^\infty t^{-({d+4})/({4\alpha})}\,dt\right).
\end{align*}
Since $\alpha\in (1,2)$, we can find $t_0(r_0;\w)=c_{22}(1+N_{r_0}(\w))^{-\frac{\alpha}{\alpha-1}}$ with $c_{22}>0$ small enough so that $$\displaystyle \int_0^{t_0(r;\w)}t^{-1/\alpha}dt\le \frac{1}{2c_{21}(1+N_{r_0}(\w))}.$$
Thus, noting that $d> 4(\alpha-1)$,
with such choice $t_0=t_0(r_0,\w)$, we can derive that
\begin{align*}
&\|\nabla \phi^{r_0}_{k,\theta_m,R_m}(\cdot;\w)\|_{L^2(\R^d)}\\
&\le c_{23}\Bigg(\left((1+N_{r_0}(\w))^{\frac{1}{\alpha-1}}+(1+N_{r_0}(\w))^{\frac{d}{4(\alpha-1)}}\right)
\left[\left(\int_{B(0,K_0r_0)}|\phi_{k,\theta_m,R_m}(x;\w)|^2\,dx\right)^{1/2}+1\right]\\
&\qquad\quad\,\, +(1+N_{r_0}(\w))^{\frac{d+4-4\alpha}{4(\alpha-1)}}\sqrt{K^*_{k,m}(\w)}\Bigg),
\end{align*} where
\begin{align*}
K^*_{k,m}(\w):=&\theta_m^2 \int_{B(0,K_0r_0)}\left|\nabla \phi_{k,\theta_m,R_m}(x;\w)\right|^2\,dx
+\int_{\R^d}|\phi_{k,\theta_m,R_m}(x;\w)|^2\nu(x)\,dx +\theta_m^2|\tilde \phi_{k,\theta_m,R_m}(\w)|^2\\
&+\int_{B(0,K_0r_0)}\left(\int_{\R^d}\left|\phi_{k,\theta_m,R_m}(x+z;\w)-\phi_{k,\theta_m,R_m}(x;\w)\right|^2\nu(z)\,dz\right)\,dx.
\end{align*}

According to \eqref{t2-1-6a}, \eqref{t2-1-8},  \eqref{t2-1-8a} and \eqref{t2-1-5},
\begin{align*}
\sup_{m\ge 1}\left(\Ee\left[K^*_{k,m}(\w)\right]+\Ee\left[\int_{B(0,K_0r_0)}|\phi_{k,\theta_m,R_m}(x;\w)|^2dx\right]\right)<\infty.
\end{align*}
This along with the H\"older inequality and \eqref{t2-1-3a} yields that for every $q\in [1,2)$,
\begin{align*}
&\sup_{m\ge 1}\Ee\left[\int_{B(0,r_0/2)}|\nabla \phi_{k,\theta_m,R_m}(x;\w)|^q\,dx\right]\le
c_{24}\sup_{m\ge 1}\Ee\left[\|\nabla \phi^{r_0}_{k,\theta_m,R_m}(\cdot;\w)\|_{L^2(\R^d)}^q\right]\\
&\le c_{25}\left(\Ee\left[
(1+N_{r_0}(\w))^{\max\left\{\frac{2}{(\alpha-1)(2-q)},\frac{d}{2(\alpha-1)(2-q)}\right\}}
\right]\right)^{(2-q)/2}\\
&\quad\times\sup_{m\ge 1}\left(\Ee\left[K^*_{k,m}(\w)\right]+\Ee\left[\int_{B(0,r_0)}|\phi_{k,\theta_m,R_m}(x;\w)|^2\,dx\right]\right)^{q/2}<\infty,
\end{align*}
where in the first inequality we used the fact that $\psi_{r_0}(x)=1$ for $x\in B(0,r_0/2)$.
and the second inequality is due to the fact $d+4-4\alpha<d$.
Hence, according to the co-cycle property \eqref{t2-1-6} and the stationary property of $\{\tau_{x}\}_{x\in \R^d}$, we can obtain that for all $x_0\in \R^d$,
\begin{align*}
\sup_{m\ge 1}\Ee\left[\int_{B(x_0,r_0/2)}|\nabla \phi_{k,\theta_m,R_m}(x;\w)|^q\,dx\right]&=
\sup_{m\ge 1}\Ee\left[\int_{B(0,r_0/2)}|\nabla \phi_{k,\theta_m,R_m}(x;\tau_{x_0}\w)|^q\,dx\right]\\
&=\sup_{m\ge 1}\Ee\left[\int_{B(0,r_0/2)}|\nabla \phi_{k,\theta_m,R_m}(x;\w)|^q\,dx\right]<\infty.
\end{align*}
In particular, by the estimate above and \eqref{t2-1-8}, $\{\phi_{\theta_m,R_m}\}_{m\ge 1}$ is weakly compact in $L^2(\Omega,W_{loc}^{1,q}(\R^d,\R^d);\Pp)$, and
its weak limit $\phi$ in $L^2(\Omega,W_{loc}^{\alpha/2,2}(\R^d,\R^d);\Pp)\cap L^2(\R^d\times \Omega, \R^d;\nu(z)\,dz\times \Pp)$ satisfies that for all $x_0\in \R^d$,
\begin{equation}\label{t2-1-15}
\Ee\left[
 \int_{B(x_0,r_0/2)}(|\phi(x;\w)|^q+|\nabla \phi(x;\w)|^q)\,dx\right]<\infty.
\end{equation}

Therefore, there is a $\Pp$-null set $\Lambda\subset \Omega$ such that
$
\phi(\cdot;\w)\in W_{loc}^{1,q}(\R^d,\R^d)$ for all $\w\notin \Lambda.
$
Moreover, $x\mapsto \nabla \phi_{\theta_m,R_m}(x;\w)$ is stationary and
there exists $\tilde \Phi_m\in L^q(\Omega,\R^d\times \R^d;\Pp)$ such that $\nabla \phi_{\theta_m,R_m}(x;\w)=\tilde \Phi_m(\tau_x\w)$. According to \eqref{t2-1-15} and the fact
$\phi$ is the weak limit of $\phi_{\theta_m,R_m}$, we can find
$\tilde \Phi\in L^q(\Omega,\R^d\times \R^d;\Pp)$ so that
$\nabla \phi(x;\w)=\tilde \Phi(\tau_x\w)$ for a.e.\ $x\in \R^d$ and a.s.\ $\w\in \Omega$. The proof is complete.
\end{proof}

The following lemma shows the sublinear growing property of the scaled corrector $\phi^\e(x;\w):=\e\phi\left(\frac{x}{\e};\w\right)$.

\begin{lemma}\label{l1-3}
Suppose that Assumptions $\ref{a1-2--}$ and $\ref{a1-2}$ hold.
Then
\begin{equation}\label{l1-3-1}
\lim_{\e \to 0}\e^2\int_{B(0,r)}\left|\phi\left(\frac{x}{\e};\w\right)\right|^2\,dx=0,\quad r\ge 1,\ {\rm a.s.}\ \w\in \Omega.
\end{equation}
Moreover, if additionally \eqref{t2-1-3a} holds with some $q\in (1,2)$ and
$d> 4(\alpha-1)$,
then
\begin{equation}\label{l1-3-1a}
\lim_{\e \to 0}\e^{\frac{qd}{d-q}}\int_{B(0,r)}\left|\phi\left(\frac{x}{\e};\w\right)\right|^{\frac{qd}{d-q}}\,dx=0,\quad r\ge 1,\ {\rm a.s.}\ \w\in \Omega.
\end{equation}
\end{lemma}
\begin{proof}
By applying  \eqref{t2-1-2} and \eqref{t2-1-8}, and following
the proof of \cite[Proposition 4.4]{PZ} (in particular, in the proof of
\cite[Proposition 4.4]{PZ} the symmetric property of the random operator $L^\w$ is not required), we can establish the desired conclusion \eqref{l1-3-1}.

Define $\phi^\e(x;\w):=\e\phi\left(\frac{x}{\e};\w\right)$. If additionally
\eqref{t2-1-3a} holds for some $q\in (1,2)$ and $d> 4(\alpha-1)$, then, by Lemma \ref{l1-5} and the ergodic theorem, we obtain
$$
\lim_{\e \to 0}\int_{B(0,r)}|\nabla \phi^\e(x;\w)|^q \,dx=|B(0,r)|\cdot \Ee[|\tilde \Phi|^q]<\infty,\quad r\ge 1.
$$
Recall that $d\ge2$ and $q\in (1,2)$. According to the Sobolev inequality (cf.\ see \cite[Corollary 9.14, p.\ 284]{Bre}), it holds that
$$
\left(\int_{B(0,r)}\left|\phi^\e(x;\w)-\oint_{B(0,r)}\phi^\e(\cdot;\w)\right|^{\frac{qd}{d-q}}\,dx\right)^{\frac{d-q}{d}}\le
c_1(r)\int_{B(0,r)}|\nabla \phi^\e(x;\w)|^q\, dx,
$$
where $\displaystyle\oint_{B(0,r)}\phi^\e(\cdot;\w):=|B(0,r)|^{-1}\displaystyle\int_{B(0,r)}\phi^\e(z;\w)\,dz$.
Hence, for any  $r\ge 1$, $\left\{\phi^\e(x;\w)-\displaystyle\oint_{B(0,r)}\phi^\e(\cdot;\w)\right\}_{\e\in (0,1)}$ is weakly compact
in $W_{loc}^{1,q}(\R^d,\R^d)$. Therefore, for any sequence, there exists its subsequence,
which is denoted by
$\left\{\phi^{\e_m}(x;\w)-\displaystyle\oint_{B(0,r)}\phi^{\e_m}(\cdot;\w)\right\}_{m\ge 1}$, weakly converges to
some $\hat \phi\in W_{loc}^{1,q}(\R^d,\R^d)$ as $m\to\infty$.
On the other hand, according to \eqref{l1-3-1},
$\left\{\phi^{\e_m}(x;\w)-\displaystyle\oint_{B(0,r)}\phi^{\e_m}(\cdot;\w)\right\}_{m\ge 1}$ converges strongly
in $L^2(B(0,r);dx)$ to $0$.
Therefore, $\hat \phi(x;\w)=0$ for a.e. $x\in B(0,r)$ and a.s. $\w\in \Omega$.

Furthermore, due to the compact embedding from
$L^{\frac{qd}{d-q}}_{loc}(\R^d,\R^d)$ to $W_{loc}^{1,q}(\R^d,\R^d)$ (cf.\ see \cite[Theorem 9.16, p.\ 285]{Bre}), there also exists one of its subsequence,
which is convergent in $L^{\frac{qd}{d-q}}_{loc}(\R^d,\R^d)$ to $\hat \phi$.
Since the subsequence is arbitrary, we can finally deduce that
the only possible limit $\hat \phi$ is zero, and so \eqref{l1-3-1a} is proved.
\end{proof}

We now establish the relation between the corrector $\phi$ and the random drift term $b(x;\w)$. Let $\tilde \Phi=\nabla \phi\in L^q(\Omega,\R^d\times\R^d;\Pp)$
be that in Lemma $\ref{l1-5}$.

\begin{lemma}\label{l1-4}
Suppose that Assumptions $\ref{a1-2--}$ and $\ref{a1-2}$ hold. Assume
that $d> 4(\alpha-1)$,
\eqref{t2-1-3a} and \begin{equation}\label{e:aabb}
\Ee[|\tilde H_{jl}|^{q/(q-1)}]<\infty,\quad 1\le j,l\le d
\end{equation} hold with some $q\in (1,2)$.
Then,
\begin{equation}\label{l1-4-1}
\sum_{j=1}^d\Ee[\tilde \Phi_{kj}(\w)\tilde H_{kj}(\w)]=-
\frac{1}{2}\Ee\left[\int_{\R^d}|\phi_k(z;\w)|^2\nu(z)\,dz\right],\quad 1\le k\le d.
\end{equation}
\end{lemma}
\begin{proof}
For every $\e\in (0,1)$, let $\phi^\e(x;\w):=\e\phi\left(\frac{x}{\e};\w\right)$. By \eqref{t2-1-1}, for every $f\in C_c^1(\R^d)$ and $1\le k \le d$,
\begin{equation}\label{l1-4-2}
\begin{split}
&\frac{\e^{-d-1}}{2}\int_{\R^d}\int_{\R^d}\left(\phi_k\left(\frac{x+z}{\e};\w\right)-\phi_k\left(\frac{x}{\e};\w\right)\right)
\left(f(x+z)-f(x)\right)\nu\left(\frac{z}{\e}\right)\,dz\,dx\\
&-\e^{-1}\int_{\R^d}\left\langle b\left(\frac{x}{\e};\w\right),\nabla f(x)\right\rangle
\phi_k^\e(x;\w)\,dx=-\e^{-1}\int_{\R^d}b_k\left(\frac{x}{\e};\w\right)f(x)\,dx.
\end{split}
\end{equation}

For $r\ge 1$, choose a function $0\le \psi_r\in C_c^\infty(\R^d)$ such that ${\rm supp}[\psi_r]\subset B(0,r)$ and $\displaystyle\int_{\R^d}\psi_r(x)\,dx=1$. For $R\ge 1$, take $\tilde \rho_R:\R \to [-2R,2R]$ to be a $C^1$-function
such that $\tilde \rho_R(s)=s$ for $-R\le s\le R$, $\tilde \rho_R(s)=2R$ for $s\ge 2R$,
$\tilde \rho_R(s)=-2R$ for $s\le -2R$, and $|\tilde \rho_R(s)|\le 2|s|$ and $|\tilde \rho_R'(s)|\le 2$ for all $s\in \R$.
Recall that, by Theorem \ref{t2-1} and Lemma \ref{l1-5}, $\phi(\cdot;\w)\in W_{loc}^{\alpha/2,2}(\R^d,\R^d)\cap W_{loc}^{1,q}(\R^d,\R^d)$. By the standard approximation
procedure, we can take $f(x)=\tilde \rho_R\left(\phi_k^\e(x;\w)\right)\psi_r(x)$ in \eqref{l1-4-2}, and obtain that
\begin{align*}
I_{1}^{\e,R}(\w)+I_2^{\e,R}(\w)=-I_3^{\e,R}(\w),
\end{align*}
where, for fixed $1\le k\le d$,
\begin{align*}
I_{1}^{\e,R}(\w):&=\frac{\e^{-d-1}}{2}\int_{\R^d}\int_{\R^d}\left(\phi_k\left(\frac{x+z}{\e};\w\right)-\phi_k\left(\frac{x}{\e};\w\right)\right)\\
&\qquad\qquad\qquad\qquad\times\left(\tilde \rho_R\left(\phi_k^\e(x+z;\w)\right)\psi_r(x+z)-\tilde \rho_R\left(\phi_k^\e(x;\w)\right)\psi_r(x)\right)\nu\left(\frac{z}{\e}\right)\,dz\,dx\\
&=\frac{\e^{-d-1}}{2}\int_{\R^d}\int_{\R^d}\left(\phi_k\left(\frac{x+z}{\e};\w\right)-\phi_k\left(\frac{x}{\e};\w\right)\right)
\left(\tilde \rho_R\left(\phi_k^\e(x+z;\w)\right)-\tilde \rho_R\left(\phi_k^\e(x;\w)\right)\right)\psi_r(x)\nu\left(\frac{z}{\e}\right)\,dz\,dx\\
&\quad+\frac{\e^{-d-1}}{2}\int_{\R^d}\int_{\R^d}\left(\phi_k\left(\frac{x+z}{\e};\w\right)-\phi_k\left(\frac{x}{\e};\w\right)\right)
\left(\psi_r(x+z)-\psi_r(x)\right)\tilde \rho_R\left(\phi_k^\e(x;\w)\right)\nu\left(\frac{z}{\e}\right)\,dz\,dx\\
&=:I_{11}^{\e,R}(\w)+I_{12}^{\e,R}(\w)
\end{align*} and
\begin{align*}
I_{2}^{\e,R}(\w):&=-\e^{-1}\int_{\R^d}\left\langle b\left(\frac{x}{\e};\w\right),
\nabla \left(\tilde \rho_R\left(\phi_k^\e(\cdot;\w)\right)\psi_r(\cdot)\right)(x)\right\rangle\phi_k^\e(x;\w)\,dx\\
&=\e^{-1}\int_{\R^d}\left\langle b\left(\frac{x}{\e};\w\right), \nabla \phi_k^\e(x;\w)\right\rangle
\tilde \rho_R\left(\phi_k^\e(x;\w)\right)\psi_r(x)\,dx,\\
I_3^{\e,R}(\w):&=\e^{-1}\int_{\R^d}b_k\left(\frac{x}{\e};\w\right)\tilde \rho_R\left(\phi_k^\e(x;\w)\right)\psi_r(x)\,dx.
\end{align*} Here in the second equality for the estimate of $I_1^{\e,R}(\w)$ we changed the variables $x+z\mapsto x$ and $z\mapsto -z$, and used the fact that $\nu(z)=\nu(-z)$ for all $z\in \R^d$; and in the second equality for the estimate of $I_2^{\e,R}(\w)$ we used the anti-symmetry of $\{\tilde H_{jl}\}_{1\le j,l\le d}$ and the integration by parts formula.

Below, we will estimate $I_{1}^{\e,R}(\w)$, $I_2^{\e,R}(\w)$ and $I_3^{\e,R}(\w)$ respectively.

(1) We first consider $I_3^{\e,R}(\w)$. Since $b_k(x;\w)=-\sum_{j=1}^d \frac{\partial H_{kj}(x;\w)}{\partial x_j}$,
\begin{align*}
I_3^{\e,R}(\w)&=-\sum_{j=1}^d\int_{\R^d} \frac{\partial }{\partial x_j}\left(H_{kj}\left(\frac{\cdot}{\e};\w\right)\right)(x)
\tilde \rho_R\left(\phi_k^\e(x;\w)\right)\psi_r(x)\,dx\\
&=\sum_{j=1}^d \int_{\R^d}H_{kj}\left(\frac{x}{\e};\w\right)\frac{\partial \phi_k\left(\frac{x}{\e};\w\right)}{\partial x_j}
\tilde \rho_R'\left(\phi_k^\e(x;\w)\right)\psi_r(x)\,dx +\sum_{j=1}^d
\int_{\R^d}H_{kj}\left(\frac{x}{\e};\w\right)\tilde \rho_R\left(\phi_k^\e(x;\w)\right)\frac{\partial \psi_r(x)}{\partial x_j}\,dx\\
&=:I_{31}^{\e,R}(\w)+I_{32}^{\e,R}(\w).
\end{align*}
Using the fact that $|\tilde \rho_R(s)|\le 2|s|$ for all $s\in \R$ and the ergodic theorem,  for every $R\ge 1$,
\begin{align*}
\varlimsup_{\e \to 0}|I_{32}^{\e,R}(\w)|&\le c_1\sum_{j=1}^d
\left(\varlimsup_{\e \to 0}\int_{B(0,r)}\left|H_{kj}\left(\frac{x}{\e};\w\right)\right|^2\,dx\right)^{1/2}
\left(\varlimsup_{\e \to 0}\int_{B(0,r)}|\phi_k^\e(x;\w)|^2\,dx\right)^{1/2}\\
&=c_1\sum_{j=1}^d|B(0,r)|^{1/2}\Ee[|\tilde H_{kj}|^2]^{1/2} \left(\varlimsup_{\e \to 0}\int_{B(0,r)}|\phi_k^\e(x;\w)|^2\,dx\right)^{1/2}=0,
\end{align*}
where the last equality follows from \eqref{l1-3-1} and \eqref{e:aabb} (that implies $\Ee[|\tilde H_{kj}|^2]<\infty$ since $q/(q-1)>2$ for $q\in (1,2)$).

For every $M\ge1$, define
\begin{align*}
A_{M,\e,1}(\w)=\{x\in B(0,r): |\phi_k^\e(x;\w)|\ge M\},\quad A_{M,\e,2}(\w)=\left\{x\in B(0,r):
\left|\nabla \phi_k^\e(x;\w)\right|\ge M\right\}.
\end{align*}
Then, by the fact that $\tilde \rho_R'\left(\phi^\e_k(x;\w)\right)=1$ for every $x\in A^c_{R,\e,1}(\w)$,
\begin{align*}
I_{31}^{\e,R}(\w)=&\sum_{j=1}^d \int_{\R^d}H_{kj}\left(\frac{x}{\e};\w\right)\frac{\partial \phi_k\left(\frac{x}{\e};\w\right)}{\partial x_j}
\psi_r(x)\,dx\\
&-\sum_{j=1}^d
\int_{A_{R,\e,1}(\w)}H_{kj}\left(\frac{x}{\e};\w\right)\frac{\partial \phi_k\left(\frac{x}{\e};\w\right)}{\partial x_j}(1-\tilde \rho_R'\left(\phi_k^\e(x;\w)\right))
\psi_r(x)\,dx.
\end{align*}
Note that $\nabla \phi(x;\w)=\tilde \Phi(\tau_x \w)$ with $\tilde \Phi\in L^q(\Omega,\R^d\times \R^d;\Pp)$, and $\tilde H_{kj}\in L^{\frac{q}{q-1}}(\Omega;\Pp)$. According to the ergodic theorem
(see e.g. \cite[Proposition 2.1]{CCKW1}), we get
\begin{align*}
\lim_{\e \to 0}\int_{\R^d}H_{kj}\left(\frac{x}{\e};\w\right)\frac{\partial \phi_k\left(\frac{x}{\e};\w\right)}{\partial x_j}\psi_r(x)\,dx=
\left(\int_{\R^d}\psi_r(x)\,dx\right)\Ee[\tilde H_{kj}(\w)\tilde \Phi_{kj}(\w)]
=\Ee[\tilde H_{kj}(\w)\tilde \Phi_{kj}(\w)].
\end{align*}
On the other hand,  since $|\tilde \rho_R'(s)|\le 2$ for all $s\in \R$, it holds that for every $M\ge 1$ and $1\le j\le d$,
\begin{align*}
&\left|\int_{A_{R,\e,1}(\w)}H_{kj}\left(\frac{x}{\e};\w\right)\frac{\partial \phi_k\left(\frac{x}{\e};\w\right)}{\partial x_j}(1-\tilde \rho_R'\left(\phi_k^\e(x;\w)\right))\psi_r(x)\,dx\right|\\
&\le 3\int_{A_{R,\e,1}(\w)}\left|H_{kj}\left(\frac{x}{\e};\w\right)\right|\left|\frac{\partial \phi_k\left(\frac{x}{\e};\w\right)}{\partial x_j}\right|
\psi_r(x)\,dx\\
&\le c_2\int_{A_{R,\e,1}(\w)\cap A_{M,\e,2}(\w)^c\cap B(0,r)}
\left|H_{kj}\left(\frac{x}{\e};\w\right)\right|\left|\frac{\partial \phi_k^\e\left(x;\w\right)}{\partial x_j}\right|\,dx + c_2\int_{A_{M,\e,2}(\w)\cap B(0,r)}\left|H_{kj}\left(\frac{x}{\e};\w\right)\right|
\left|\tilde \Phi_{kj}\left(\tau_{\frac{x}{\e}}\w\right)\right|\,dx.
\end{align*}
By the definition of $A_{M,\e,2}(\w)$, we have
\begin{align*}
&\int_{A_{R,\e,1}(\w)\cap A_{M,\e,2}(\w)^c\cap B(0,r)}
\left|H_{kj}\left(\frac{x}{\e};\w\right)\right|\left|\frac{\partial \phi_k^\e\left(x;\w\right)}{\partial x_j}\right|\,dx\\
&\le M\int_{A_{R,\e,1}(\w)\cap B(0,r)}\left|H_{kj}\left(\frac{x}{\e};\w\right)\right|\,dx\le M\left|A_{R,\e,1}(\w)\cap B(0,r)\right|^{1/2}
\left(\int_{B(0,r)}\left|H_{kj}\left(\frac{x}{\e};\w\right)\right|^2\,dx\right)^{1/2}\\
&\le \frac{M}{R}\left(\int_{B(0,r)}|\phi_k^\e(x;\w)|^2\,dx\right)^{1/2}
\left(\int_{B(0,r)}\left|H_{kj}\left(\frac{x}{\e};\w\right)\right|^2\,dx\right)^{1/2},
\end{align*}
where in the last inequality we used
\begin{align}\label{l1-4-3}
\left|A_{R,\e,1}(\w)\cap B(0,r)\right|\le \frac{1}{R^2}\int_{B(0,r)}|\phi_k^\e(x;\w)|^2\,dx.
\end{align}
Thus, according to \eqref{l1-3-1}, \eqref{e:aabb} and the ergodic theorem, we obtain that for every fixed $R\ge 1$ and $M\ge 1$,
\begin{align*}
\lim_{\e\to 0}\int_{A_{R,\e,1}(\w)\cap A_{M,\e,2}(\w)^c\cap B(0,r)}
\left|H_{kj}\left(\frac{x}{\e};\w\right)\right|\left|\frac{\partial \phi_k^\e\left(x;\w\right)}{\partial x_j}\right|\,dx=0.
\end{align*}
Note again that $\nabla \phi_k(x;\w)=\tilde \Phi_k(\tau_x \w)$. Applying the ergodic theorem, we derive
\begin{align*}
&\lim_{M \to \infty}\lim_{\e \to 0}\int_{B(0,r)}\int_{A_{M,\e,2}(\w)\cap B(0,r)}\left|H_{kj}\left(\frac{x}{\e};\w\right)\right|
\left|\tilde \Phi_{kj}\left(\tau_{\frac{x}{\e}}\w\right)\right|\,dx\\
&=\lim_{M \to \infty}\lim_{\e \to 0}\int_{B(0,r)}\left|H_{kj}\left(\frac{x}{\e};\w\right)\right|
\left|\tilde \Phi_{kj}\left(\tau_{\frac{x}{\e}}\w\right)\right|\I_{\left\{\left|\tilde \Phi_k\left(\tau_{\frac{\cdot}{\e}\w}\right)\right|\ge M\right\}}(x)\,dx\\
&=|B(0,r)|\lim_{M \to \infty}\Ee\left[|\tilde H_{kj}(\w)\tilde \Phi_{kj}(\w)|\I_{\{|\tilde \Phi_{k}|\ge M\}}(\w)\right]=0,
\end{align*}
where the last equality follows from the facts that $\tilde H_{kj}\in L^{\frac{q}{q-1}}(\Omega;\Pp)$ and
$\tilde \Phi_{kj}\in L^q(\Omega;\Pp)$ again.

Therefore, combining with all the estimates above, firstly letting $\e \to 0$ and then $M \to \infty$, we have
for every fixed $R\ge 1$,
\begin{align*}
\lim_{\e \to 0}I_3^{\e,R}(\w)=\sum_{j=1}^d\Ee[\tilde H_{kj}(\w)\tilde \Phi_{kj}(\w)].
\end{align*}

(2) In this part, we consider $I_1^{\e,R}(\w)$. By the change of variable, we derive
\begin{align*}
&I_{11}^{\e,R}(\w)\\
&=\frac{\e^{-1}}{2}\int_{\R^d}\int_{\R^d}
\left(\phi_k\left(\frac{x}{\e}+z;\w\right)-\phi_k\left(\frac{x}{\e};\w\right)\right)
\left(\tilde \rho_R\left(\e\phi_k\left(\frac{x}{\e}+z;\w\right)\right)-\tilde \rho_R\left(\e\phi_k\left(\frac{x}{\e};\w\right)\right)\right)
\psi_r(x)\nu\left(z\right)\,dz\,dx\\
&=\frac{1}{2}\int_{\R^d}\left(\int_{\R^d}\left|\phi_k\left(z;\tau_{\frac{x}{\e}}\w\right)\right|^2\nu(z)dz\right)\psi_r(x)\,dx\\
&\quad-\frac{\e^{-1}}{2}\int_{\R^d}\int_{\R^d}
\left(\phi_k\left(\frac{x}{\e}+z;\w\right)-\phi_k\left(\frac{x}{\e};\w\right)\right)
\left(\chi_R\left(\e\phi_k\left(\frac{x}{\e}+z;\w\right)\right)-\chi_R\left(\e\phi_k\left(\frac{x}{\e};\w\right)\right)\right)
\psi_r(x)\nu\left(z\right)\,dz\,dx\\
&=:I_{111}^{\e}(\w)+I_{112}^{\e,R}(\w),
\end{align*}
where $\chi_R(s):=s-\tilde \rho_R(s)$ and we used the co-cycle property \eqref{t2-1-1a}.
Then, using the ergodic theorem and the fact that $\displaystyle\int_{\R^d} \psi_r(x)\,dx=1$, we have
\begin{align*}
\lim_{\e \to 0}I_{111}^{\e}(\w)=\frac{1}{2}\Ee\left[\int_{\R^d}\left|\phi_k\left(z;\w\right)\right|^2\nu(z)\,dz\right].
\end{align*}

For every $M\ge 1$, define
\begin{align*}
O_{M,\e}(\w):=\left\{(x,z)\in \R^d\times \R^d: \left|\phi_k\left(z;\tau_{\frac{x}{\e}}\w\right)\right|\ge M\right\}.
\end{align*}
Note that $|\chi_R(a)-\chi_R(b)|\le 2|b-a|$, then
for every $R\ge 1$ and $\delta\in (0,1)$,
\begin{equation}\label{l1-4-5}
\begin{split}
|I_{112}^{\e,R}(\w)|\le& c_3\int_{B(0,r)}\int_{\{|z|\le \delta\}}
\left|\phi_k\left(z;\tau_{\frac{x}{\e}}\w\right)\right|^2\nu(z)\,dz\,dx\\
&+c_3\int\int_{((B(0,r)\cap A_{R/2,\e,1}(\w))\times B(0,\delta)^c)\cap O_{R/2,\e}(\w)^c}\left|\phi_k\left(z;\tau_{\frac{x}{\e}}\w\right)\right|^2\nu(z)\,dz\,dx\\
&+c_3\int\int_{(B(0,r)\times \R^d )\cap O_{R/2,\e}(\w)}\left|\phi_k\left(z;\tau_{\frac{x}{\e}}\w\right)\right|^2\nu(z)\,dz\,dx\\
=&:c_3\left(I_{1121}^{\e,\delta}(\w)+I_{1122}^{\e,R,\delta}(\w)+I_{1123}^{\e,R}(\w)\right).
\end{split}
\end{equation}
Here we used the property that
\begin{align*}
\chi_R\left(\e\phi_k\left(\frac{x}{\e}+z;\w\right)\right)-\chi_R\left(\e\phi_k\left(\frac{x}{\e};\w\right)\right)=0,
\quad (x,z)\in \left(A_{R/2,\e,1}(\w)^c \times \R^d\right)\cap O_{R/2,\e}(\w)^c,
\end{align*}
which follows from the fact that $\chi_R(a)-\chi_R(b)=0$ for every $a,b\in \R$ with $|a|\le R/2$ and
$|a-b|\le R/2$.

Using \eqref{t2-1-2} and the ergodic theorem directly, we get
\begin{align*}
\lim_{\delta \to 0}\lim_{\e \to 0}I_{1121}^{\e,\delta}(\w)&=
|B(0,r)| \lim_{\delta \to 0}\Ee\left[\int_{\{|z|\le \delta\}}\left|\phi_k\left(z;\w\right)\right|^2\nu(z)\,dz\right]=0
\end{align*} and
\begin{align*}
\lim_{R \to \infty}\lim_{\e \to 0}I_{1123}^{\e,R}(\w)=
|B(0,r)| \lim_{R\to \infty}\Ee\left[\int_{\R^d}\left|\phi_k\left(z;\w\right)\right|^2
\I_{\{|\phi_k(\cdot;\w)|\ge R/2\}}(z)\nu(z)\,dz\right]=0.
\end{align*}
By the definition of $O_{M,\e}(\w)$, we obtain that for every $\delta\in (0,1)$ and $R\ge 1$,
\begin{align*}
\varlimsup_{\e \to 0}|I_{1122}^{\e,R,\delta}(\w)|&\le \frac{R^2}{4}\left(\lim_{\e \to 0}\left|B(0,r)\cap A_{R/2,\e,1}(\w)\right|\right)\left(\int_{\{|z|>\delta\}}\nu(z)\,dz\right)\\
&\le c_4(\delta,R)\varlimsup_{\e \to 0}
\left(\int_{B(0,r)}|\phi_k^\e(x;\w)|^2\,dx\right)^{1/2}=0,
\end{align*}
where we have used \eqref{l1-4-3} and \eqref{l1-3-1}. Hence, firstly letting $\e \to 0$, and then taking $R \to \infty$ and $\delta \to 0$, $$
\lim_{R \to \infty}\lim_{\varepsilon \to 0}|I_{112}^{\e,R}(\w)|=0.$$
Therefore,
\begin{align*}
\lim_{R \to \infty}\lim_{\e \to 0}I_{11}^{\e,R}(\w)=\frac{1}{2}\Ee\left[\int_{\R^d}|\phi_k(z;\w)|^2\nu(z)\,dz\right].
\end{align*}

By the change of variable and the mean value theorem, we have
\begin{align*}
|I_{12}^{\e,R}(\w)|
&=\frac{\e^{-1}}{2}
\left|\int_{\R^d}\int_{\R^d}\left(\phi_k\left(\frac{x}{\e}+z;\w\right)-\phi_k\left(\frac{x}{\e};\w\right)\right)
\left(\psi_r(x+\e z)-\psi_r(x)\right)\tilde \rho_R\left(\phi_k^\e(x;\w)\right)\nu(z)\,dz\,dx\right|\\
&\le \|\nabla \psi_r\|_{L^\infty(\R^d)}\int_{B(0,K_0r)}
\left(\int_{\R^d}\left|\phi_k\left(z;\tau_{\frac{x}{\e}}\w\right)\right||z|\nu(z)\,dz\right)\cdot |\phi_k^\e(x;\w)|\,dx\\
&\quad+c_5\e^{-1}R\int_{B(0,K_0r)^c}\left|\phi_k\left(\frac{x}{\e};\w\right)\right|
\left(\int_{B\left(\frac{x}{\e},\frac{r}{\e}\right)}\nu(z)\,dz\right)\,dx\\
&\quad+c_5\e^{-1}R\int_{B(0,K_0r)^c}\left(\int_{B\left(0,\frac{r}{\e}\right)}|\phi_k(z;\w)|^2\,dz\right)^{1/2}
\left(\int_{B\left(\frac{x}{\e},\frac{r}{\e}\right)}\nu(z)^2\,dz\right)^{1/2}\,dx\\
&=:I_{121}^{\e,R}(\w)+I_{122}^{\e,R}(\w)+I_{123}^{\e,R}(\w),
\end{align*}
where in the inequality above we used the facts that $|\tilde \rho_R(s)|\le 2R$ and $|\tilde \rho_R(s)|\le 2s$ for all $s\in \R$, and $K_0>1$ is the constant in Assumption \ref{a1-2--}(iii).

By the H\"older inequality and the ergodic theorem,
\begin{align*}
\varlimsup_{\e \to 0}|I_{121}^{\e,R}(\w)|&\le
c_6
\left(\lim_{\e \to 0}\int_{B(0,K_0r)}\int_{\R^d}\left|\phi_k\left(z;\tau_{\frac{x}{\e}}\w\right)\right|^2\nu(z)\,dz\,dx\right)^{1/2}
\left(\lim_{\e \to 0}\int_{B(0,K_0r)}|\phi_k^\e(x;\w)|^2\,dx\right)^{1/2}\\
&=c_6|B(0,K_0r)|^{1/2}\Ee\left[\int_{\R^d}|\phi_k(z;\w)|^2\nu(z)\,dz\right]^{1/2}
\left(\lim_{\e \to 0}\int_{B(0,K_0r)}|\phi_k^\e(x;\w)|^2\,dx\right)^{1/2}=0,
\end{align*}
where in the inequality we used the fact $\displaystyle\int_{\R^d}|z|^2\nu(z)\,dz<\infty$, and the last equality is due to \eqref{t2-1-2} and \eqref{l1-3-1}.

According to \eqref{t2-1-2a}, \eqref{a1-2-1}, \eqref{t2-1-2} and the H\"older inequality, we derive that for every fixed $R\ge 1$,
\begin{align*}
\varlimsup_{\e \to 0}|I_{122}^{\e,R}(\w)|&\le
c_7r^d R\varlimsup_{\e \to 0}\e^{-d-1}\int_{B(0,K_0r)^c}\left|\phi_k\left(\frac{x}{\e};\w\right)\right|\gamma\left(\frac{x}{\e}\right)\nu\left(\frac{x}{\e}\right)dx\\
&\le c_8(r,R)\left(\int_{\R^d}\left|\phi_k\left(x;\w\right)\right|^2\nu(x)\,dx\right)^{1/2}
\left(\varlimsup_{\e \to 0}\e^{-2}\int_{B\left(0,\frac{K_0r}{\e}\right)^c}\gamma^2(x)\nu(x)\,dx\right)^{1/2}=0
\end{align*}
where in the last inequality we used
\begin{align*}
\int_{B\left(0,\frac{K_0r}{\e}\right)^c}\gamma^2(x)\nu(x)\,dx&\le \frac{\e^2}{(K_0r)^2}
\left(\int_{B\left(0,\frac{K_0r}{\e}\right)^c}|x|^2\gamma^2(x)\nu(x)\,dx\right).
\end{align*}

Using \eqref{t2-1-2a}, \eqref{l1-3-1} and \eqref{t2-1-2} again, we find that
\begin{align*}
 \varlimsup_{\e \to 0}|I_{123}^{\e,R}(\w)|
&\le c_9Rr^{d/2}
\left(\lim_{\e \to 0}\int_{B\left(0,r\right)}|\phi_k^\e(x;\w)|^2\,dx\right)^{1/2}\\
&\quad\times
\left(\varlimsup_{\e \to 0}\e^{-2}\left(\int_{B\left(0,\frac{K_0r}{\e}\right)^c}\gamma^2(x)\nu(x)\,dx\right)^{1/2}
\left(\int_{B\left(0,\frac{K_0r}{\e}\right)^c}\nu(x)\,dx\right)^{1/2}\right)\\
&\le c_{10}(r,R) \left(\lim_{\e \to 0}\int_{B\left(0,r\right)}|\phi_k^\e(x;\w)|^2\,dx\right)^{1/2}\\
&\quad\times
\varlimsup_{\e \to 0}\left[\left(\int_{B\left(0,\frac{K_0r}{\e}\right)^c}|x|^2\gamma^2(x)\nu(x)\,dx\right)^{1/2}
\left(\int_{B\left(0,\frac{K_0r}{\e}\right)^c}|x|^2\nu(x)\,dx\right)^{1/2}\right]\\
&=0.
\end{align*}

Combining with all the estimates above yields that for every fixed $R\ge 1$,
\begin{align*}
\lim_{\e \to 0}I_{12}^{\e,R}(\w)=0
\end{align*} and so
\begin{align*}
\lim_{R\to \infty }
\lim_{\e \to 0}I_{1}^{\e, R}(\w)=\frac{1}{2}\Ee\left[\int_{\R^d}\left|\phi_k\left(z;\w\right)\right|^2\nu(z)\,dz\right].
\end{align*}

(3) Now, we consider $I_2^{\e, R}(\w)$. According to $b_k(x;\w)=-\sum_{j=1}^d \frac{\partial H_{kj}(x;\w)}{\partial x_j}$ and the integration by parts formula, we find
\begin{equation}\label{l1-4-4}
\begin{split}
I_2^{\e,R}(\w)&=-\sum_{j,l=1}^d \int_{\R^d}\frac{\partial}{\partial x_j}\left(H_{lj}\left(\frac{\cdot}{\e};\w\right)\right)(x)
\frac{\partial \phi_k^\e(x;\w)}{\partial x_l}\tilde \rho_R\left(\phi_k^\e(x;\w)\right)\psi_r(x)\,dx\\
&= \sum_{j,l=1}^d \int_{\R^d}H_{lj}\left(\frac{x}{\e};\w\right)\tilde \rho_R\left(\phi_k^\e(x;\w)\right)\frac{\partial \phi_k^\e(x;\w)}{\partial x_l}
\frac{\partial \psi_r(x)}{\partial x_j}\,dx,
\end{split}
\end{equation}
where in the second equality we used the anti-symmetry of $\{H_{lj}\}_{1\le l,j\le d}$.
Furthermore,  it holds that
\begin{align*}
|I_2^{\e,R}(\w)|&\le c_{11}\sum_{j,l=1}^d\int_{B(0,r)\cap A_{M,\e,2}(\w)^c}\left|H_{lj}\left(\frac{x}{\e};\w\right)\right|
\left|\tilde \rho_R\left(\phi_k^\e(x;\w)\right)\right|\left|\frac{\partial \phi_k^\e(x;\w)}{\partial x_l}\right|\,dx\\
&\quad +c_{11}\sum_{j,l=1}^d\int_{B(0,r)\cap A_{M,\e,2}(\w)}\left|H_{lj}\left(\frac{x}{\e};\w\right)\right|
\left|\tilde \rho_R\left(\phi_k^\e(x;\w)\right)\right|\left|\frac{\partial \phi_k^\e(x;\w)}{\partial x_l}\right|\,dx\\
&=:c_{11}\left(I_{21}^{\e,R,M}(\w)+I_{22}^{\e,R,M}(\w)\right).
\end{align*}
By the definition of $A_{M,\e,2}(\w)$, we deduce that for every fixed $M\ge 1$,
\begin{align*}
\varlimsup_{\e \to 0}I_{21}^{\e,R,M}(\w)&\le c_{12}M\sum_{j,l=1}^d
\left(\lim_{\e \to 0}\int_{B(0,r)}\left|H_{lj}\left(\frac{x}{\e};\w\right)\right|^2\,dx\right)^{1/2}
\left(\lim_{\e \to 0}\int_{B(0,r)}|\phi_k^\e(x;\w)|^2\,dx\right)^{1/2}\\
&=c_{12}M|B(0,r)|^{1/2}\left(\sum_{j,l=1}^d \Ee[|\tilde H_{lj}|^2]^{1/2}\right)
\left(\lim_{\e \to 0}\int_{B(0,r)}|\phi_k^\e(x;\w)|^2\,dx\right)^{1/2}=0,
\end{align*}
where the last equality follows from \eqref{l1-3-1} again.
Note that $|\rho(s)|\le 2R$ for all $s\in \R$ and $\tilde \Phi_k(\tau_x \w)=\nabla \phi_k(x;\w)$, then
\begin{align*}
\varlimsup_{M \to \infty}\varlimsup_{\e \to 0}I_{22}^{\e,R,M}(\w)&\le c_{13}R\sum_{j,l=1}^d
\varlimsup_{M \to \infty}\lim_{\e \to 0}
\int_{B(0,r)}\left|H_{lj}\left(\frac{x}{\e};\w\right)\right|\left|\tilde \Phi_{kl}(\tau_{\frac{x}{\e}}\w)\right|
\I_{\tilde \Phi(\tau_{\frac{\cdot}{\e}}\w)\ge M}(x)\,dx\\
&=c_{13}R|B(0,r)|\left(\varlimsup_{M \to \infty}\Ee[|\tilde H_{lj}\tilde \Phi_{kl}(\w)|\I_{\{|\tilde \Phi_k|\ge M\}}(\w)
]\right)=0,
\end{align*}
where the last equality follows from \eqref{e:aabb}.
Putting all the estimates above together, and first letting $\e \to 0$ and then $M \to \infty$, we can show that for
every fixed $R\ge 1$,
\begin{align*}
\lim_{\e \to 0}I_2^{\e,R}(\w)=0.
\end{align*}

According to all the limit properties of $I_1^{\e,R}(\w)$, $I_2^{\e,R}(\w)$ and
$I_3^{\e,R}(\w)$, we get the desired conclusion \eqref{l1-4-1}.
\end{proof}

\begin{remark}\label{r1-1}
According to the proof above, we can obtain that under the assumptions of Lemma \ref{l1-4},
for every $1\le k,l\le d$,
\begin{equation}\label{r1-1-1}
\sum_{j=1}^d \Ee[(\tilde \Phi_{kj}(\w)+\tilde \Phi_{lj}(\w))
(\tilde H_{kj}(\w)+\tilde H_{lj}(\w))]=
-\frac{1}{2}\Ee\left[\int_{\R^d}|\phi_k(z;\w)+\phi_l(z;\w)|^2\nu(z)\,dz\right].
\end{equation}
We emphasize  the crucial step to prove \eqref{r1-1-1} is that one can still make the cancelation
in \eqref{l1-4-4} with $\phi_k^\e$ replaced by $\phi_k^\e+\phi_l^\e$.
\end{remark}

\section{Proof of Theorem \ref{t1-2}}

\begin{proof}[Proof of Theorem $\ref{t1-2}$] The proof is split into three steps.

(1) For every $r\ge 2$ and $\theta\in (0,1)$, define
$B_{\theta}(0,r)=\{y\in B(0,r):|y-\partial B(0,r)|>\theta\}$.
According to \eqref{t1-1-1a}, it holds that
\begin{align}\label{t1-2-3}
\int_{B(0,r)\backslash B_{\theta}(0,r)}|u^\e(x;\w)|\,dx\le
\|u^\e(\cdot;\w)\|_{L^\infty(\R^d)}\left|B(0,r)\backslash B_{\theta}(0,r)\right|
\le c_1(r)\sup_{\e\in (0,1)}\|u^\e(\cdot;\w)\|_{L^\infty(\R^d)}\theta.
\end{align}

For every $y\in \R^d$ with $|y|\le \theta$, letting $N_0^\e:=\left[\frac{2\theta}{\e}\right]+1$,
we have
\begin{align*}
\int_{B_{\theta}(0,r)}\left|u^\e(x+y;\w)-u^\e(x;\w)\right|\,dx&
\le \sum_{k=0}^{N_0^\e-1} \int_{B_{\theta}(0,r)}\left|u^\e\left(x+\frac{(k+1)y}{N_0^\e};\w\right)-u^\e\left(x+\frac{ky}{N_0^\e};\w\right)\right|\,dx\\
&=:\sum_{k=0}^{N_0^\e-1} J_k^\e(y;\w).
\end{align*}
Let $y_k:=\frac{ky}{N_0^\e}$ for $0\le k \le N_0^\e-1$. Then, for every $0\le k \le N_0^\e-1$,
\begin{align*}
J_k^\e(y;\w)&=\left|B\left(0,\e/2\right)\right|^{-d}\int_{B_{\theta}(0,r)}
\left(\int_{B\left(x+y_k,\e/2\right)}\left|u^\e(x+y_{k+1};\w)-u^\e(x+y_k;\w)\right|dz\right)\,dx\\
&\le c_2\e^{-d}\Bigg(\int_{B_{\theta}(0,r)}\int_{B\left(x+y_k,\e/2\right)}\left|u^\e(x+y_{k+1};\w)-u^\e(z;\w)\right|\,dz\,dx\\
&\quad\quad\quad\quad +\int_{B_{\theta}(0,r)}\int_{B\left(x+y_k,\e/2\right)}\left|u^\e(x+y_k;\w)-u^\e(z;\w)\right|\,dz\,dx\Bigg)\\
&=:c_2\e^{-d}\left(J_{k,1}^\e(y;\w)+J_{k,2}^\e(y;\w)\right).
\end{align*}
Since  $|y_k-y_{k+1}|\le \e/2$ and $|y_k|\le |y|\le \theta\le r$, it holds that
\begin{align*}
J_{k,1}^\e(y;\w)&\le \int_{B_{\theta}(0,r)}\int_{B\left(0,\e\right)}\left|u^\e(x+y_{k+1};\w)-u^\e(x+y_{k+1}+z;\w)\right|\,dz\,dx\\
&\le \int_{B(0,2r)}\left(\int_{\{|z|\le \e\}}
\frac{\left|u^\e(x+z;\w)-u^\e(x;\w)\right|^2}{|z|^{d+\alpha}}\,dz\right)^{1/2}
\left(\int_{\{|z|\le \e\}}|z|^{d+\alpha}\,dz\right)^{1/2}\,dx\\
&\le c_3(r)\e^{d+\alpha/2}\left(\int_{B(0,2r)}\int_{\{|z|\le \e\}}
\frac{\left|u^\e(x+z;\w)-u^\e(x;\w)\right|^2}{|z|^{d+\alpha}}\,dz\,dx\right)^{1/2}.
\end{align*}
By  \eqref{t1-1-0a} and the definition of $\nu(z)$,
\begin{align*}
\sup_{\e\in (0,1)}\e^{-(2-\alpha)}\left(\int_{B(0,2r)}\int_{\{|z|\le \e\}}
\frac{\left|u^\e(x+z;\w)-u^\e(x;\w)\right|^2}{|z|^{d+\alpha}}\,dz\,dx\right)<\infty,
\end{align*}
which implies immediately that $J_{k,1}^\e(y;\w)\le c_4(r)\e^{d+1}$. Following the same argument, we can find that $J_{k,2}^\e(y;\w)\le c_5(r)\e^{d+1}$ for every $0\le k\le N_0^\e-1$.
Hence, we can obtain that for every $y\in \R^d$ with $|y|\le \theta$,
\begin{align*}
\sup_{\e\in (0,1)}\int_{B_{\theta}(0,r)}\left|u^\e(x+y;\w)-u^\e(x;\w)\right|\,dx\le
\sup_{\e\in (0,1)}\left(\sum_{k=0}^{N_0^\e-1} J_k^\e(y;\w)\right)\le c_6(r)\theta.
\end{align*}

By this and \eqref{t1-2-3}, we can apply \cite[Theorem 1.95, p.\ 37]{DD}
to conclude that $\{u^\e(\cdot;\w): \e\in (0,1]\}$ is
precompact as $\e \to 0$ in $L^1 (B(0,R);dx)$ for all $R\ge1$. So,
we can find a subsequence $\{u^{\e_m}(\cdot;\w)\}_{m\ge 1}$ and $u_0(\cdot;\w)$
such that
\begin{equation*}
\lim_{m \to \infty}\int_{B(0,R)}|u^{\e_m}(x;\w)-u_0(x;\w)| \,dx=0,\quad R\ge 1.
\end{equation*}
Combining this with \eqref{t1-1-1a} yields that for every $p>0$,
\begin{equation}\label{t1-2-7}
\lim_{m \to \infty}\int_{B(0,R)}|u^{\e_m}(x;\w)-u_0(x;\w)|^p\, dx=0,\quad R\ge 1.
\end{equation}
Thus it remains to prove that, for every convergent subsequence $\{u^{\e_m}\}_{m\ge 1}$ in $L_{loc}^p(\R^d)$, the limit
$u_0(\cdot;\w)$ is the solution to the equation \eqref{t1-2-2}.

(2) For every $\e\in (0,1)$, let $\phi^\e(x;\w):=\e\phi\left(\frac{x}{\e};\w\right)$.
Given any $f\in C_c^\infty(\R^d)$, we take the test function
\begin{equation}\label{e:test}f_{\e_m}(x;\w):=f(x)+\sum_{j=1}^d\phi_j^{\e_m}\left(x;\w\right)\frac{\partial f(x)}{\partial x_j}\end{equation}(which belongs to
$W^{1,q}(\R^d)$ with compact supports by Lemma \ref{l1-5}) in \eqref{t1-1-0}, and derive that
\begin{equation}\label{t1-2-8}
\begin{split}
&\lambda\int_{\R^d}u^{\e_m}(x;\w)f_{\e_m}(x;\w)\,dx-
\int_{\R^d}u^{\e_m}(x;\w)L_0^{\e_m}f(x)(x)\,dx+\e_m^{-1}\int_{\R^d}\left\langle b\left(\frac{x}{\e_m};\w\right), \nabla f_{\e_m}(x;\w)\right\rangle u^{\e_m}(x;\w)\,dx\\
&+\frac{\e_m^{-d-2}}{2}\sum_{j=1}^d\int_{\R^d}\int_{\R^d}\left(\phi_j^{\e_m}\left(x+z;\w\right)\frac{\partial f(x+z)}{\partial x_j}- \phi_j^{\e_m}\left(x;\w\right)\frac{\partial f(x)}{\partial x_j}\right)
(u^{\e_m}(x+z;\w)-u^{\e_m}(x;\w))\nu\left(\frac{z}{\e_m}\right)\,dz\,dx\\
&=\int_{\R^d}h(x)f_{\e_m}(x;\w)\,dx.
\end{split}
\end{equation}
Note that $f\in C_c^\infty(\R^d)$. By \eqref{l1-3-1} and \eqref{t1-2-7}, we immediately get that
\begin{align*}
\lim_{m \to \infty}\int_{\R^d}u^{\e_m}(x;\w)f_{\e_m}(x;\w)\,dx=\int_{\R^d}u_0(x;\w)f(x)\,dx
\end{align*} and
$$\lim_{m\to\infty}\int_{\R^d}h(x)f_{\e_m}(x;\w)\,dx=\int_{\R^d} h(x)f(x)\,dx.$$

Below we will estimate other terms involved in the left hand side of \eqref{t1-2-8}. We write for every $M\ge 1$ that
\begin{align*}
L_0^{\e_m}f(x)&=\e_m^{-2}\left(\int_{\{|z|\le M\}}+\int_{\{|z|>M\}}\right)\left(f(x+\e_m z)-f(x)-\e_m\langle \nabla f(x), z \rangle \right)\nu(z)\,dz\\
&=:\frac{1}{2}\int_{\{|z|\le M\}}\left\langle \nabla^2 f(x), z\otimes z\right\rangle \nu(z)\,dz+G^{m}_{1,M}(x).
\end{align*}
Below without loss of generality we suppose that ${\rm supp}[f]\subset B(0,R_0)$ for some $R_0>1$. Let $K_0\ge2$ be the constant in Assumption \ref{a1-2--}(iii). Then, for all $x\in B(0,K_0R_0)$,  by Taylor's expansion,
\begin{align*}
|G_{1,M}^m(x)|\le  \frac{\e_m\|\nabla^3 f\|_{L^\infty(\R^d)}}{6}\int_{\{|z|\le M\}}|z|^3\nu(z)\,dz+
\frac{\|\nabla^2 f\|_{L^\infty(\R^d)}}{2}\int_{\{|z|>M\}}|z|^2\nu(z)\,dz;
\end{align*}
while for all $x\in B(0,K_0R_0)^c$, by \eqref{t2-1-2a},
\begin{align*}
|G_{1,M}^m(x)|&\le \e_m^{-2}\|f\|_{L^\infty(\R^d)} \int_{B\left(\frac{x}{\e_m},\frac{R_0}{\e_m}\right)}\nu(z)\,dz\le c_7(R_0)\e_m^{-d-2}\gamma\left(\frac{x}{\e_m}\right)\nu\left(\frac{x}{\e_m}\right),
\end{align*}
which  along with the H\"older inequality and \eqref{a1-2-1} implies that
\begin{align}\label{t1-2-9a}
\int_{\{|x|>K_0R_0\}}|G_{1,M}^m(x)|\,dx
&\le c_8(R_0)
\left(\int_{\{|z|>\frac{R_0}{\e_m}\}}|z|^2\gamma^2(z)\nu(z)\,dz\right)^{1/2}\left(\int_{\{|z|>\frac{R_0}{\e_m}\}}\nu(z)\,dz\right)^{1/2}\to 0
\end{align} as $m\to \infty$.
Hence, combining all the estimates above  with \eqref{t1-1-1a} and \eqref{t1-2-7}, and firstly letting $m\to \infty$ and then $M\to \infty$, we get
\begin{align}\label{t1-2-9}
\lim_{m \to \infty}-\int_{\R^d}L_0^{\e_m}f(x)u^{\e}(x;\w)\,dx=-
\frac{1}{2}\int_{\R^d}\left\langle \nabla^2 f(x), \int_{\R^d}(z\otimes z) \nu(z)\,
dz\right\rangle u_0(x;\w)\,dx.
\end{align}

On the other hand, it holds that
\begin{align*}
&\frac{\e_m^{-d-2}}{2}\int_{\R^d}\int_{\R^d}\left(\phi_j^{\e_m}\left(x+z;\w\right)\frac{\partial f(x+z)}{\partial x_j}-
\phi_j^{\e_m}\left(x;\w\right)\frac{\partial f(x)}{\partial x_j}\right)
\left(u^{\e_m}(x+z;\w)-u^{\e_m}(x;\w)\right)
\nu\left(\frac{z}{\e_m}\right)
\,dz\,dx\\
&=\frac{\e_m^{-d-2}}{2}\int_{\R^d}\int_{\R^d}\left(\phi_j^{\e_m}\left(x+z;\w\right)-\phi_j^{\e_m}\left(x;\w\right)\right)
\left(u^{\e_m}(x+z;\w)\frac{\partial f(x+z)}{\partial x_j}-u^{\e_m}(x;\w)\frac{\partial f(x)}{\partial x_j}\right)\nu\left(\frac{z}{\e_m}\right)\,dz\,dx\\
&\quad -\e_m^{-d-2}\int_{\R^d}\int_{\R^d}\left(\frac{\partial f(x+z)}{\partial x_j}-\frac{\partial f(x)}{\partial x_j}\right)
\left(\phi_j^{\e_m}\left(x+z;\w\right)-\phi_j^{\e_m}\left(x;\w\right)\right)
\nu\left(\frac{z}{\e_m}\right)u^{\e_m}(x;\w)\,dz\,dx\\
&\quad-\int_{\R^d}L_0^{\e_m}\left(\frac{\partial f(\cdot)}{\partial x_j}\right)(x)\phi_j^{\e_m}\left(x;\w\right)u^{\e_m}(x;\w)\,dx\\
&=:I_{21}^{m,j}(\w)+I_{22}^{m,j}(\w)+I_{23}^{m,j}(\w).
\end{align*}
By the change of variable and the definition of $\phi^\e$, we have
\begin{align*}
I_{21}^{m,j}(\w)&=\frac{\e_m^{-1}}{2}
\int_{\R^d}\int_{\R^d}
\left(\phi_j\left(\frac{x}{\e_m}+z;\w\right)-\phi_j\left(\frac{x}{\e_m};\w\right)\right)\\
&\quad\quad\quad\quad\quad\quad\quad\times \left(u^{\e_m}(x+\e_m z;\w)\frac{\partial f(x+\e_m z)}{\partial x_j}-u^{\e_m}(x;\w)\frac{\partial f(x)}{\partial x_j}\right)\nu(z)\,dz\,dx
\end{align*} and
\begin{align*}
I_{22}^{m,j}(\w)=-
\e_m^{-1}
\int_{\R^d} \int_{\R^d}\left(\phi_j\left(\frac{x}{\e_m}+z;\w\right)-\phi_j\left(\frac{x}{\e_m};\w\right)\right)
\left(\frac{\partial f(x+\e_m z)}{\partial x_j}-\frac{\partial f(x)}{\partial x_j}\right) \nu\left(z\right)u^{\e_m}(x;\w)\,dz \,dx.
\end{align*}
Furthermore, by Taylor's expansion, we know that for every $M\ge 1$,
\begin{align*}
&\e_m^{-1}
\int_{\R^d} \left(\phi_j\left(\frac{x}{\e_m}+z;\w\right)-\phi_j\left(\frac{x}{\e_m};\w\right)\right)
\left(\frac{\partial f(x+\e_m z)}{\partial x_j}-\frac{\partial f(x)}{\partial x_j}\right) \nu\left(z\right)\,dz\\
&=\sum_{k=1}^d\frac{\partial^2 f(x)}{\partial x_j \partial x_k}
\left(\int_{\{|z|\le M\}}\left(\phi_j\left(\frac{x}{\e_m}+z;\w\right)
-\phi_j\left(\frac{x}{\e_m};\w\right)\right)z_k\nu\left(z\right)dz\right)+G_{2,M}^{m,j}(x;\w).
\end{align*}
Here, for every $x\in B(0,K_0R_0)$,
\begin{align*}
|G_{2,M}^{m,j}(x;\w)|&\le \frac{\|\nabla^3 f\|_{L^\infty(\R^d)}M\e_m}{2}
\int_{\{|z|\le M\}}\left|\phi_j\left(\frac{x}{\e_m}+z;\w\right)-\phi_j\left(\frac{x}{\e_m};\w\right)\right||z|^2\nu(z)\,dz\\
&\quad+\|\nabla^2 f\|_{L^\infty(\R^d)}\int_{\{|z|>M\}}\left|\phi_j\left(\frac{x}{\e_m}+z;\w\right)-\phi_j\left(\frac{x}{\e_m};\w\right)\right||z|\nu(z)\,dz\\
&\le c_{8}\left(\int_{\R^d}\left|\phi_j\left(\frac{x}{\e_m}+z;\w\right)-\phi_j\left(\frac{x}{\e_m};\w\right)\right|^2\nu(z)\,dz\right)^{1/2}\\
&\quad\times \left(\e_m M\left(\int_{\{|z|\le M\}}|z|^4\nu(z)dz\right)^{1/2}+\left(\int_{\{|z|>M\}}|z|^2\nu(z)\,dz\right)^{1/2}\right),
\end{align*}
which, together with the co-cycle property \eqref{t2-1-1} and the stationary property of the transformation $\{\tau_x\}_{x\in \R^d}$, implies that
\begin{align*}
  \lim_{M \to \infty}\lim_{m \to \infty}\int_{B(0,K_0R_0)}|G_{2,M}^{m,j}(x;\w)|^2\,dx
&\le |B(0,K_0R_0)| \Ee\left[\int_{\R^d}|\phi_j(z;\w)|^2\nu(z)\,dz\right]\left(\lim_{M \to \infty}\int_{\{|z|>M\}}|z|^2\nu(z)\,dz\right)\\
&=0;
\end{align*}
while for every $x\in B(0,K_0R_0)^c$, we find by using \eqref{t2-1-2a} again that
\begin{align*}
|G_{2,M}^{m,j}(x;\w)|
&\le \e_m^{-1}\|\nabla f\|_{L^\infty(\R^d)}\Bigg[
\left(\int_{B\left(\frac{x}{\e_m},\frac{R_0}{\e_m}\right)}\nu(z)\,dz\right)\left|\phi_j\left(\frac{x}{\e_m};\w\right)\right|\\
&\qquad \qquad\qquad\qquad\quad+
\left(\int_{B\left(0,\frac{R_0}{\e_m}\right)}\left|\phi_j\left(z;\w\right)\right|^2\,dz\right)^{1/2}
\left(\int_{B\left(\frac{x}{\e_m},\frac{R_0}{\e_m}\right)}\nu(z)^2\,dz\right)^{1/2}\Bigg]\\
&\le c_{9}\e_m^{-d-1}\gamma\left(\frac{x}{\e_m}\right)\nu\left(\frac{x}{\e_m}\right)
\left(\left|\phi_j\left(\frac{x}{\e_m};\w\right)\right|+
\left(\e_m^d\int_{B\left(0,\frac{R_0}{\e_m}\right)}\left|\phi_j\left(z;\w\right)\right|^2\,dz\right)^{1/2}\right),
\end{align*}
which, along with the H\"older inequality, gives us
\begin{align*}
& \int_{\{|x|>K_0R_0\}}|G_{2,M}^{m,j}(x;\w)|\, dx\\
&
\le c_{10}\left(\int_{\{|x|>\frac{K_0R_0}{\e_m}\}}\left|\phi_j\left(x;\w\right)\right|^2\nu(x)\,dx \right)^{1/2}\left(\int_{\{|x|>\frac{K_0R_0}{\e_m}\}}|x|^2\gamma^2(x)\nu(x)\,dx\right)^{1/2}\\
&\quad+c_{10}
\left(\int_{B\left(0,R_0\right)}\e_m^2\left|\phi_j\left(\frac{x}{\e_m};\w\right)\right|^2\,dx\right)^{1/2}
\left(\int_{\{|x|>\frac{K_0R_0}{\e_m}\}}|x|^2\nu(x)\,dx\right)^{1/2} \left(\int_{\{|x|>\frac{K_0R_0}{\e_m}\}}|x|^2\gamma^2(x)\nu(x)\,dx\right)^{1/2}\\
&\to 0
\end{align*} as $m\to \infty$.
Here we used \eqref{a1-2-1} and the argument for \eqref{t1-2-9a}.

Hence, combining all the estimates above together with \eqref{t1-1-1a} and \eqref{t1-2-7} yields that
\begin{align*}
&\lim_{m \to \infty}I_{22}^{m,j}(\w)\\
&=-
\lim_{M \to \infty}\lim_{m \to \infty}
\sum_{k=1}^d\int_{\R^d} \frac{\partial^2 f(x)}{\partial x_j \partial x_k}
\left(\int_{\{|z|\le M\}}\left(\phi_j\left(\frac{x}{\e_m}+z;\w\right)-\phi_j\left(\frac{x}{\e_m};\w\right)\right)
z_k \nu\left(z\right)dz\right)u_0(x;\w)\,dx\\
&=-
\lim_{M \to \infty}\lim_{m \to \infty}
\sum_{k=1}^d\int_{\R^d} \frac{\partial^2 f(x)}{\partial x_j \partial x_k}
\left(\int_{\{|z|\le M\}} \phi_j\left(z;\tau_{\frac{x}{\e_m}}\w\right)
z_k \nu\left(z\right)dz\right)u_0(x;\w)\,dx\\
&=-
\sum_{k=1}^d\Ee\left[\int_{\R^d}\phi_j\left(z;\w\right)
z_k \nu\left(z\right)dz\right]\int_{\R^d} \frac{\partial^2 f(x)}{\partial x_j \partial x_k}
u_0(x;\w)\,dx,
\end{align*} where in the second equality we used the co-cycle property \eqref{t2-1-1a}.

Since ${\rm supp}\left[\frac{\partial f}{\partial x_j}\right]\subset B(0,R_0)$, we can easily use the mean value theorem and \eqref{t2-1-2a} to get that
\begin{align*}
\left|L_0^{\e_m}\left(\frac{\partial f}{\partial x_j}\right)(x)\right|&\le c_{11}\left(\I_{B(0,K_0R_0)}(x)+\I_{B(0,K_0R_0)^c}(x)\e_m^{-d-2}\gamma\left(\frac{x}{\e_m}\right)\nu\left(\frac{x}{\e_m}\right)\right).
\end{align*}
This along with the second inequality in \eqref{t1-2-9a} yields that
\begin{align*}
 \left|I_{23}^{m,j}(\w)\right|\le c_{12}&\sup_{\e\in (0,1)}\|u^{\e}(\cdot;\w)\|_{L^\infty(\R^d)}\Bigg[
\e_m\int_{B(0,K_0R_0)}\left|\phi_j \left(\frac{x}{\e_m};\w\right)\right|\,dx\\
&+\left(\int_{\{|x|>\frac{K_0R_0}{\e_m}\}}\left|\phi_j\left(x;\w\right)\right|^2\nu(x)\,dx \right)^{1/2}\left(\int_{\{|x|>\frac{K_0R_0}{\e_m}\}}|x|^2\gamma^2(x)\nu(x)\,dx\right)^{1/2}\Bigg].
\end{align*}
Then, according to  \eqref{t2-1-2},  \eqref{l1-3-1} and \eqref{a1-2-1}, we can obtain immediately that
\begin{align*}
\lim_{m \to \infty}|I_{23}^{m,j}(\w)|=0.
\end{align*}

Therefore, putting both estimates for $I_{22}^{m,j}(\w)$ and $I_{23}^{m,j}(\w)$ together, we find that for every $1\le j \le d$,
\begin{equation}\label{t1-2-10}
\begin{split}
&\lim_{m\to \infty}\Bigg[\frac{\e_m^{-d-2}}{2}\int_{\R^d}\int_{\R^d}\left(\phi_j^{\e_m}\left(x+z;\w\right)\frac{\partial f(x+z)}{\partial x_j}- \phi_j^{\e_m}\left(x;\w\right)\frac{\partial f(x)}{\partial x_j}\right)
\left(u^{\e_m}(x+z;\w)-u^{\e_m}(x;\w)\right)\nu\left(\frac{z}{\e}\right)\,dz\,dx\\
&\quad -\frac{\e_m^{-1}}{2}
\int_{\R^d}\int_{\R^d}
\left(\phi_j\left(\frac{x}{\e_m}+z;\w\right)-\phi_j\left(\frac{x}{\e_m};\w\right)\right)
\left(u^{\e_m}(x+\e_m z;\w)\frac{\partial f(x+\e_m z)}{\partial x_j}-u^{\e_m}(x;\w)\frac{\partial f(x)}{\partial x_j}\right)\nu(z)\,dz\,dx\Bigg]\\
&=-
\sum_{k=1}^d\Ee\left[\int_{\R^d}\phi_j\left(z;\w\right)
z_k \nu\left(z\right)\,dz\right]\int_{\R^d} \frac{\partial^2 f(x)}{\partial x_j \partial x_k}
u_0(x;\w)\,dx.
\end{split}
\end{equation}

(3) In this part, we deal with the drift term in the right hand side of \eqref{t1-2-8}. It holds that
\begin{align*}
&\int_{\R^d}\e_m^{-1}\left\langle b\left(\frac{x}{\e_m};\w\right), \nabla f_{\e_m}(x;\w)\right\rangle u^{\e_m}(x;\w)\,dx\\
&=\sum_{j=1}^d\int_{\R^d}\e_m^{-1}   b_j\left(\frac{x}{\e_m};\w\right)\frac{\partial f(x)}{\partial x_j} u^{\e_m}(x;\w)\,dx +\sum_{k,j=1}^d\int_{\R^d}\e_m^{-1}
 b_k\left(\frac{x}{\e_m};\w\right)\frac{\partial \phi_j\left(\frac{x}{\e_m};\w\right)}{\partial x_k}
\frac{\partial f(x)}{\partial x_j}u^{\e_m}(x;\w)\,dx\\
&\quad+I_{3}^m(\w),
\end{align*}
where
\begin{align*}
&I_{3}^m(\w):=\sum_{k,j=1}^d\int_{\R^d}\e_m^{-1}b_k\left(\frac{x}{\e_m};\w\right)
\phi_j^{\e_m}\left(x;\w\right)\frac{\partial^2 f(x)}{\partial x_j\partial x_k}u^{\e_m}(x;\w)\,dx.
\end{align*}

Note that $\e_m^{-1}b_k\left(\frac{x}{\e_m};\w\right)=-\sum_{l=1}^d \frac{\partial}{\partial x_l}\left(H_{kl}\
\left(\frac{\cdot}{\e_m}\right)\right)(x)$. By using the integration by parts formula,
we derive
\begin{align*}
I_{3}^m(\w)& = \sum_{j,k,l=1}^d \int_{\R^d}H_{kl}\left(\frac{x}{\e_m};\w\right)
\frac{\partial \phi_j\left(\frac{x}{\e_m};\w\right)}{\partial x_l}\frac{\partial^2 f(x)}{\partial x_j\partial x_k}u^{\e_m}(x;\w)\,dx\\
&\quad+ \sum_{j,k,l=1}^d\int_{\R^d}H_{kl}\left(\frac{x}{\e_m};\w\right)\phi_j^{\e_m}(x;\w)
\frac{\partial}{\partial x_l}\left(\frac{\partial^2 f(\cdot)}{\partial x_j\partial x_k}u^{\e_m}(\cdot;\w)\right)(x)\,dx\\
&=:I_{31}^m(\w)+\sum_{j,k,l=1}^d I_{32}^{m,jkl}(\w).
\end{align*}

According to Lemma \ref{l1-5}, $\nabla \phi(x;\w)=\tilde \Phi(\tau_x \w)$ with $\tilde \Phi\in L^q(\Omega;\Pp)$ for some $q\in (1,2)$.
Since, by
\eqref{a1-3-1},
$\tilde H_{kl}\in L^{r}(\Omega;\Pp)$ for some $r>{q}/({q-1})$, it follows from \eqref{t1-2-7} and the
ergodic theorem
that
\begin{align*}
&\lim_{m \to \infty}I_{31}^m(\w)\\
&=\lim_{m \to \infty}\sum_{j,k,l=1}^d \int_{\R^d}H_{kl}\left(\frac{x}{\e_m};\w\right)
\frac{\partial \phi_j\left(\frac{x}{\e_m};\w\right)}{\partial x_l}\frac{\partial^2 f(x)}{\partial x_j\partial x_k}u_0(x;\w)\,dx\\
&= \frac{1}{2}\lim_{m \to \infty}\sum_{j,k,l=1}^d\int_{\R^d}\left(H_{kl}\left(\frac{x}{\e_m};\w\right)+H_{jl}\left(\frac{x}{\e_m};\w\right)\right)
\left(\frac{\partial \phi_k\left(\frac{x}{\e_m};\w\right)}{\partial x_l}+\frac{\partial \phi_j\left(\frac{x}{\e_m};\w\right)}{\partial x_l}\right)
\frac{\partial^2 f(x)}{\partial x_j\partial x_k}u_0(x;\w)\,dx\\
&\quad -
\lim_{m \to \infty}
\sum_{j,k,l=1}^d
\int_{\R^d}H_{kl}\left(\frac{x}{\e_m};\w\right)
\frac{\partial \phi_k\left(\frac{x}{\e_m};\w\right)}{\partial x_l}\frac{\partial^2 f(x)}{\partial x_j\partial x_k}u_0(x;\w)\,dx\\
&=\sum_{j,k=1}^d
\left(\frac{1}{2}\sum_{l=1}^d\left(\Ee\left[(\tilde \Phi_{kl}(\w)+\tilde \Phi_{jl}(\w))
(\tilde H_{kl}(\w)+\tilde H_{jl}(\w))\right]-2\Ee\left[\tilde \Phi_{kl}(\w)\tilde H_{kl}(\w)\right]\right)\right)\int_{\R^d}\frac{\partial^2 f(x)}{\partial x_j\partial x_k}u_0(x;\w)\,dx\\
&=-\frac{1}{2}\sum_{j,k=1}^d\int_{\R^d}\left(\frac{1}{2}\Ee\left[\int_{\R^d}\left(\phi_j(z;\w)+\phi_k(z;\w)\right)^2\nu(z)\,dz\right]-
\Ee\left[\int_{\R^d}|\phi_k(z;\w)|^2\nu(z)\,dz\right]\right)\frac{\partial^2 f(x)}{\partial x_j\partial x_k}u_0(x;\w)\,dx\\
&=-\frac{1}{2}\sum_{j,k=1}^d\int_{\R^d}\Ee\left[\int_{\R^d}\phi_j(z;\w)\phi_k(z;\w)\nu(z)\,dz\right]\frac{\partial^2 f(x)}{\partial x_j\partial x_k}u_0(x;\w)\,dx,
\end{align*}
where in the
fourth
equality we used \eqref{l1-4-1} and \eqref{r1-1-1}.

Let $\psi\in C_c^\infty(\R^d)$ be such that $\psi(x)=1$ for every $x\in B(0,R_0)$ and ${\rm supp}[\psi]\subset B(0,R_0+1)$. According to the
Parseval equality,
\begin{align*}
|I_{32}^{m,jkl}(\w)|&= \left|\int_{\R^d}H_{kl}\left(\frac{x}{\e_m};\w\right)\phi_j^{\e_m}(x;\w)\psi(x)
\frac{\partial}{\partial x_l}\left(\frac{\partial^2 f(\cdot)}{\partial x_j\partial x_k}u^{\e_m}(\cdot;\w)\right)(x)\,dx\right|\\
&=\left|i\int_{\R^d}\mathcal{F}\left(H_{kl}\left(\frac{\cdot}{\e_m};\w\right)\phi_j^{\e_m}(\cdot;\w)\psi(\cdot)\right)(\xi)\,\,
\mathcal{F}\left(\frac{\partial^2 f(\cdot)}{\partial x_j\partial x_k}u^{\e_m}(\cdot;\w)\right)(\xi)\,\xi_l \,d\xi\right|\\
&\le \left(\int_{\R^d}\left|\mathcal{F}\left(H_{kl}\left(\frac{\cdot}{\e_m};\w\right)\phi_j^{\e_m}(\cdot;\w)\psi(\cdot)\right)(\xi)\right|^2
\left(\I_{\{|\xi|\le \e_m^{-1}\}}+\e_m^{2-\alpha}|\xi|^{2-\alpha}\I_{\{|\xi|> \e_m^{-1}\}}\right)\,d\xi\right)^{1/2}\\
&\quad\times \left(\int_{\R^d}\left|\mathcal{F}\left(\frac{\partial^2 f(\cdot)}{\partial x_j\partial x_k}u^{\e_m}(\cdot;\w)(\xi)\right)\right|^2
\left(|\xi|^2\I_{\{|\xi|\le \e_m^{-1}\}}+\e_m^{-(2-\alpha)}|\xi|^{\alpha}\I_{\{|\xi|> \e_m^{-1}\}}\right)d\xi\right)^{1/2}\\
&=:\left(I_{321}^{m,jkl}(\w)\right)^{1/2}\times\left(I_{322}^{m,jkl}(\w)\right)^{1/2}.
\end{align*}
Here $\mathcal{F}(f)(\xi)$ denotes the Fourier transform of $f\in L^1(\R^d;dx)$.

According to \eqref{l1-1-2} and its proof, we  can verify directly that
\begin{equation}\label{t1-2-11}
\begin{split}
&\e_m^{-(2-\alpha)}\int_{\{|z|\le \e_m\}}\frac{1-e^{-i\langle \xi,z\rangle}}{|z|^{d+\alpha}}\,dz\ge c_{11}\left(|\xi|^2\I_{\{|\xi|\le \e_m^{-1}\}}+\e_m^{-(2-\alpha)}|\xi|^{\alpha}\I_{\{|\xi|> \e_m^{-1}\}}\right),\\
&\e_m^{-(2-\alpha)}\int_{\{|z|\le \e_m\}}\frac{1-e^{-i\langle \xi,z\rangle}}{|z|^{d+\alpha}}\,dz\le c_{12}\left(|\xi|^2\I_{\{|\xi|\le \e_m^{-1}\}}+\e_m^{-(2-\alpha)}|\xi|^{\alpha}\I_{\{|\xi|> \e_m^{-1}\}}\right).
\end{split}
\end{equation}
So, by \eqref{t1-2-11}, we get
\begin{align*}
I_{322}^{m,jkl}(\w)&\le c_{13}\e_m^{-(2-\alpha)}\int_{\R^d}\int_{\{|z|\le \e_m\}}\frac{\left|\frac{\partial^2 f(x+z)}{\partial x_j\partial x_k}
u^{\e_m}(x+z;\w)-\frac{\partial^2 f(x)}{\partial x_j\partial x_k}u^{\e_m}(x)\right|^2}{|z|^{d+\alpha}}\,dz\,dx\\
&\le c_{14}\|\nabla^2 f\|_{L^\infty(\R^d)}\e_m^{-(2-\alpha)}\int_{\R^d}\int_{\{|z|\le \e_m\}}
\frac{\left|u^{\e_m}(x+z;\w)-u^{\e_m}(x)\right|^2}{|z|^{d+\alpha}}\,dz\,dx\\
  &\quad +c_{14}\|u^{\e_m}(\cdot;\w)\|_{L^\infty(\R^d)}\e_m^{-(2-\alpha)}
\int_{\R^d}\int_{\{|z|\le \e_m\}}
\frac{\left|\frac{\partial^2 f(x+z)}{\partial x_j\partial x_k}-\frac{\partial^2 f(x)}{\partial x_j\partial x_k}\right|^2}{|z|^{d+\alpha}}\,dz\,dx.
\end{align*}
According to \eqref{t1-1-1a} and \eqref{t1-1-0a},
\begin{align*}
\sup_{m\ge 1}I_{322}^{m,jkl}(\w)<\infty.
\end{align*}
On the other hand, using \eqref{t1-2-11} again, we obtain
\begin{align*}
I_{321}^{m,jkl}(\w)&\le c_{15}\int_{\R^d}\left|H_{kl}\left(\frac{x}{\e_m};\w\right)\phi_j^{\e_m}(x;\w)\psi(x)\right|^2\,dx\\
&\quad+c_{15}\e_m^{2-\alpha}\int_{\R^d}\int_{\{|z|\le \e_m\}}\frac{\left|H_{kl}\left(\frac{x+z}{\e_m};\w\right)\phi_j^{\e_m}(x+z;\w)\psi(x+z)-
H_{kl}\left(\frac{x}{\e_m};\w\right)\phi_j^{\e_m}(x;\w)\psi(x)\right|^2}{|z|^{d+2-\alpha}}\,dz\,dx\\
&\le c_{16}\int_{B(0,R_0+2)}\left|H_{kl}\left(\frac{x}{\e_m};\w\right)\phi_j^{\e_m}(x;\w)\right|^2
\left(1+\e_m^{2-\alpha}\sup_{x\in \R^d}\int_{\{|z|\le \e_m\}}\frac{\left(\psi(x+z)-\psi(x)\right)^2}{|z|^{d+2-\alpha}}\,dz\right)\,dx\\
&\quad +c_{16}\e_m^{2-\alpha}\int_{B(0,R_0+2)}\left|H_{kl}\left(\frac{x}{\e_m};\w\right)\right|^2
\left(\int_{\{|z|\le \e_m\}}\frac{\left|\phi_j^{\e_m}(x+z;\w)-\phi_j^{\e_m}(x;\w)\right|^2}{|z|^{d+2-\alpha}}\,dz\right)\,dx\\
&\quad +c_{16}\e_m^{2-\alpha}\int_{B(0,R_0+2)}|\phi_j^{\e_m}(x;\w)|^2\left(\int_{\{|z|\le \e_m\}}\frac{\left(H_{kl}\left(\frac{x+z}{\e_m};\w\right)-
H_{kl}\left(\frac{x}{\e_m};\w\right)\right)^2}{|z|^{d+2-\alpha}}\,dz\right)\,dx\\
&=:I_{3211}^{m,jkl}(\w)+I_{3212}^{m,jkl}(\w)+I_{3213}^{m,jkl}(\w).
\end{align*}

Recall that $p_0=dq/(d-q)$, so $\frac{2p_0}{p_0-2}= \frac{2qd}{q(d+2)-2d}$; on the other hand, the condition $q>\frac{2d}{d+2}$ implies that
$p_0>2$.
According to \eqref{a1-3-1} and the fact $p_0'<p_0$, we have $\tilde H_{kl}\in L^{\frac{2qd}{q(d+2)-2d}}(\Omega;\Pp)$. By the H\"older inequality and \eqref{l1-3-1a},
we can show immediately that
\begin{align*}
& \lim_{m \to \infty}I_{3211}^{m,jkl}(\w)\\
&\le c_{17}
\lim_{m \to \infty}\left(\int_{B(0,R_0+2)}\left|H_{kl}\left(\frac{x}{\e_m};\w\right)\right|^{\frac{2qd}{q(d+2)-2d}}\,dx\right)^{\frac{q(d+2)-2d}{qd}}
\left(\int_{B(0,R_0+2)}\left|\phi_j^{\e_m}(x;\w)\right|^{\frac{qd}{d-q}}\,dx\right)^{\frac{2(d-q)}{qd}}=0.
\end{align*}
Similarly, also by
\eqref{a1-3-2},
\eqref{l1-3-1a} and the H\"older inequality, we derive that
\begin{align*}
\lim_{m \to \infty}I_{3213}^{m,jkl}(\w)=0.
\end{align*}

Now we turn to the estimate for $I_{3212}^{m,jkl}(\w)$. For every $p>2$, by the H\"older inequality,
\begin{equation}\label{t1-2-10a}
\begin{split}
 I_{3212}^{m,jkl}(\w)
&\le c_{16}\e_m^{2-\alpha}\left(\int_{B(0,R_0+2)}\left|H_{kl}\left(\frac{x}{\e_m};\w\right)\right|^{\frac{2p}{p-2}}\,dx\right)^{\frac{p-2}{p}}\\
&\quad \times
\left(\int_{B(0,R_0+2)}\left(\int_{\{|z|\le \e_m\}}\frac{\left|\phi_j^{\e_m}\left(x+z;\w\right)-\phi_j^{\e_m}\left(x;\w\right)\right|^2}{|z|^{d+2-\alpha}}\,dz\right)^{\frac{p}{2}}\,dx\right)^{\frac{2}{p}}.
\end{split}
\end{equation}
Furthermore, for every $\delta>1$,
\begin{align*}
&\left(\int_{B(0,R_0+2)}\left(\int_{\{|z|\le \e_m\}}\frac{\left|\phi_j^{\e_m}\left(x+z;\w\right)-\phi_j^{\e_m}\left(x;\w\right)\right|^2}{|z|^{d+2-\alpha}}\,dz\right)^{\frac{p}{2}}dx\right)^{\frac{2}{p}}\\
&\le \e_m^{\alpha}\left(\int_{B(0,R_0+2)}\left(\int_{\{|z|\le 1\}}
\frac{\left|\phi_j\left(\frac{x}{\e_m}+z;\w\right)-\phi_j\left(\frac{x}{\e_m};\w\right)\right|^p}{|z|^{d+\frac{(2-\alpha)p\delta}{2}}}\,dz\right)
\left(\int_{\{|z|\le 1\}}\frac{1}{|z|^{d-\frac{(\delta-1)(2-\alpha)p}{p-2}}}dz\right)^{\frac{p-2}{2}}dx\right)^{\frac{2}{p}}\\
&\le c_{18}\e_m^{\alpha}\left(\int_{B(0,R_0+2)}\int_{\{|z|\le 1\}}
\frac{\left|\phi_j\left(\frac{x}{\e_m}+z;\w\right)-\phi_j\left(\frac{x}{\e_m};\w\right)\right|^p}{|z|^{d+\frac{(2-\alpha)p\delta}{2}}}\,dz\,dx\right)^{\frac{2}{p}}\\
&=c_{18}\e_m^{\alpha+\frac{2d}{p}}\left(\int_{B\left(0,\frac{R_0+2}{\e_m}\right)}\int_{\{|z|\le 1\}}
\frac{\left|\phi_j\left(x+z;\w\right)-\phi_j\left(x;\w\right)\right|^p}{|z|^{d+\frac{(2-\alpha)p\delta}{2}}}\,dz\,dx\right)^{\frac{2}{p}}.
\end{align*}
Let $\chi_m\in C_c^\infty(\R^d)$ be such that $\chi_m(x)=1$ for every
$x\in B\left(0,\frac{R_0+2}{\e_m}+1\right)$, $\chi_m(x)=0$ for every $x\in B\left(0,\frac{R_0+2}{\e_m}+2\right)^c$
and $\|\nabla \chi_m\|_{L^\infty(\R^d)}\le 2$. Then,
\begin{align*}
&\left(\int_{B\left(0,\frac{R_0+2}{\e_m}\right)}\int_{\{|z|\le 1\}}
\frac{\left|\phi_j\left(x+z;\w\right)-\phi_j\left(x;\w\right)\right|^p}{|z|^{d+\frac{(2-\alpha)p\delta}{2}}}\,dz\,dx\right)^{\frac{2}{p}}\\
&= \left(\int_{B\left(0,\frac{R_0+2}{\e_m}\right)}\int_{\{|z|\le 1\}}
\frac{\left|\phi_j\left(x+z;\w\right)\chi_m(x+z)-\phi_j\left(x;\w\right)\chi_m(x)\right|^p}{|z|^{d+\frac{(2-\alpha)p\delta}{2}}}\,dz\,dx\right)^{\frac{2}{p}}\\
&\le \|\phi_j \chi_m\|_{W^{\frac{(2-\alpha)\delta}{2},p}(\R^d)}^2.
\end{align*}
Thanks to $\alpha\in (1,2)$, we can take $\delta>1$  such that $\theta_0:=\frac{(2-\alpha)\delta}{\alpha}\in (0,1)$, and
$p_0>p_1>2$ so that
\begin{align*}
\frac{1}{p_1}=\frac{1-\theta_0}{p_0}+\frac{\theta_0}{2},
\end{align*}
where $p_0:=\frac{dq}{d-q}>2$. In particular, $\frac{(2-\alpha)\delta}{2}=(1-\theta_0)\cdot 0+\theta_0\frac{\alpha}{2}$. Below, we take $p=p_1$. By the interpolation of Besov spaces, see e.g. \cite[Theorem 6.4.5]{BJ}
or \cite[Section 2.4.1, Theorem (a)]{T},
\begin{align*}
 \|\phi_j \chi_m\|_{W^{\frac{(2-\alpha)\delta}{2},p_1}(\R^d)}^2
&\le
\|\phi_j\chi_m\|_{L^{p_0}(\R^d)}^{2(1-\theta_0)}\|\phi_j\chi_m\|_{W^{\frac{\alpha}{2},2}(\R^d)}^{2\theta_0}\\
&\le c_{16}\left(\int_{B\left(0,\frac{R_0+3}{\e_m}\right)}|\phi_j(x;\w)|^{p_0}\,dx\right)^{\frac{2(1-\theta_0)}{p_0}}\\
&\quad\times \left(\int_{B\left(0,\frac{R_0+3}{\e_m}\right)}\int_{\{|z|\le 1\}}\frac{|\phi_j(x+z;\w)-\phi_j(x;\w)|^2}{|z|^{d+\alpha}}\,dz\,dx
+\int_{B\left(0,\frac{R_0+3}{\e_m}\right)}|\phi_j(x;\w)|^{2}\,dx\right)^{\theta_0}.
\end{align*}
Here in the last inequality we used $\|\nabla \chi_m\|_{L^\infty(\R^d)}\le 2$.

Hence, according to all the estimates above, we get
\begin{align*}
&\e_m^{2-\alpha}\left(\int_{B(0,R_0+3)}\left(\int_{\{|z|\le \e_m\}}\frac{\left|\phi_j^{\e_m}\left(x+z;\w\right)-\phi_j^{\e_m}\left(x;\w\right)\right|^2}{|z|^{d+2-\alpha}}\,dz\right)^{\frac{p_1}{2}}\,dx\right)^{\frac{2}{p_1}}\\
&\le c_{17}\e_m^{2}\left(\int_{B\left(0,R_0+3 \right)}\left|\phi_j\left(\frac{x}{\e_m};\w\right)\right|^{p_0}\,dx\right)^{\frac{2(1-\theta_0)}{p_0}}\\
&\quad\times \left(\int_{B\left(0,R_0+3\right)}\int_{\{|z|\le 1\}}\frac{\left|\phi_j\left(z;\tau_{\frac{x}{\e_m}}\w\right)\right|^2}{|z|^{d+\alpha}}\,dz\,dx
+\int_{B\left(0,R_0+3\right)}\left|\phi_j\left(\frac{x}{\e_m};\w\right)\right|^{2}\,dx\right)^{\theta_0}.
\end{align*}
Putting this into \eqref{t1-2-10a} with $p=p_1$, and applying \eqref{l1-3-1a} and \eqref{t2-1-2},
as well as $2<p_1<p_0$ with $\delta>1$ close to $1$ (and then $p_1$ is close to $p_0'=\frac{2p_0\alpha}{4(\alpha-1)+p_0(2-\alpha)}$) so that $2p_1/(p_1-2)<r$, we can prove
 that
\begin{align*}
\lim_{m \to \infty}I_{3212}^{m,jkl}(\w)=0.
\end{align*}

Therefore, combining the estimate for $I_3^m(\w)$  with \eqref{t1-2-9} and \eqref{t1-2-10}
yields that
\begin{align*}
&-\int_{\R^d}u^{\e_m}(x;\w)L_0^{\e_m}f(x)(x)\,dx+\e_m^{-1}\int_{\R^d}\left\langle b\left(\frac{x}{\e_m};\w\right), \nabla f_{\e_m}(x;\w)\right\rangle u^{\e_m}(x;\w)\,dx\\
&+\frac{\e_m^{-d-2}}{2}\sum_{j=1}^d\int_{\R^d}\int_{\R^d}\left(\phi_j^{\e_m}\left(x+z;\w\right)\frac{\partial f(x+z)}{\partial x_j}- \phi_j^{\e_m}\left(x;\w\right)\frac{\partial f(x)}{\partial x_j}\right)
\left(u^{\e_m}(x+z;\w)-u^{\e_m}(x;\w)\right)\nu\left(\frac{z}{\e}\right)\,dz\,dx\\
&=J^m(\w)-\frac{1}{2}\sum_{j,k=1}^d\Ee\left[\int_{\R^d}\left(z_j+\phi_j(z;\w)\right)\left(z_k+\phi_k(z;\w)\right)\nu(z)\,dz\right]
\int_{\R^d}\frac{\partial^2 f(x)}{\partial x_j \partial x_k}u_0(x;\w)\,dx\\
&\quad+\frac{\e_m^{-1}}{2}\sum_{j=1}^d
\int_{\R^d}\int_{\R^d}
\left(\phi_j\left(\frac{x}{\e_m}+z;\w\right)-\phi_j\left(\frac{x}{\e_m};\w\right)\right)
\left(u^{\e_m}(x+\e_m z;\w)\frac{\partial f(x+\e_m z)}{\partial x_j}-u^{\e_m}(x;\w)\frac{\partial f(x)}{\partial x_j}\right)\nu(z)\,dz\,dx\\
&\quad+\sum_{k,j=1}^d\int_{\R^d}\e_m^{-1}
 b_k\left(\frac{x}{\e_m};\w\right)\frac{\partial \phi_j\left(\frac{x}{\e_m};\w\right)}{\partial x_k}
\frac{\partial f(x)}{\partial x_j}u^{\e_m}(x;\w)\,dx+\sum_{j=1}^d\int_{\R^d}\e_m^{-1} b_j\left(\frac{x}{\e_m};\w\right)\frac{\partial f(x)}{\partial x_j} u^{\e_m}(x;\w)\,dx\\
&=J^m(\w)-\frac{1}{2}\sum_{j,k=1}^d\bar a_{jk}
\int_{\R^d}\frac{\partial^2 f(x)}{\partial x_j \partial x_k}u_0(x;\w)\,dx,
\end{align*}
where $J^m(\w)$ satisfies that $\lim_{m \to \infty}J^m(\w)=0$, in the second equality we have used the anti-symmetry of $\{\tilde H_{lj}\}_{1\le l,j\le d}$,
\eqref{t2-1-1} and
the definition of $\bar A=\{\bar a_{jk}\}_{1\le j,k\le d}$.

Hence, we can take $m\to \infty$ in \eqref{t1-2-8}, and obtain that for every $f\in C_c^\infty(\R^d)$,
$$
\lambda \int_{\R^d}u_0(x;\w)f(x)\,dx-\frac{1}{2}\sum_{j,k=1}^d \int_{\R^d}\bar a_{jk}\frac{\partial^2 f(x)}{\partial x_j \partial x_k}u_0(x;\w)\,dx
=\int_{\R^d}h(x)f(x)\,dx.
$$
This implies that the limit $u_0(x;\w)$ is the unique solution to the equation \eqref{t1-2-2}. In particular, $u_0(x;
\w)=u_0(x)$ is non-random. Thus, we
prove the desired conclusion \eqref{t1-2-1}.
\end{proof}

\ \

\noindent {\bf Acknowledgements.}\,\,  The research of Xin Chen is supported by the National Natural Science Foundation of China
(No.\ 12122111).
The research
of Jian Wang is supported by the NSF of China the National Key R\&D Program of China (2022YFA1006003) and the National Natural Science
Foundation of China (Nos. 12225104 and 12531007).

\vskip 0.3truein
{\small
{\bf Xin Chen:}
   School of Mathematical Sciences, Shanghai Jiao Tong University, 200240 Shanghai, P.R. China.
   \newline Email: \texttt{chenxin217@sjtu.edu.cn}

\bigskip

{\bf Kun Yin:}
    School of Mathematical Sciences, Shanghai Jiao Tong University, 200240 Shanghai, P.R. China.
  \newline  Email: \texttt{epsilonyk@sjtu.edu.cn}

\bigskip

{\bf Jian Wang:}
 School of Mathematics and Statistics \& Key Laboratory of Analytical Mathematics and Applications (Ministry of Education) \& Fujian Provincial Key Laboratory
of Statistics and Artificial Intelligence, Fujian Normal University, 350007 Fuzhou, P.R. China.
\newline Email:\texttt{jianwang@fjnu.edu.cn}

}

\end{document}